\documentclass[11pt,leqno]{amsart}

\usepackage{graphicx}
\usepackage{epstopdf,epsfig}
\usepackage[font=small]{caption}
\usepackage{accents}

\usepackage[utf8]{inputenc}
\usepackage[T1]{fontenc}

 \usepackage{lipsum}

\usepackage{xcolor}
\usepackage{url}

\usepackage{palatino}
\usepackage[mathcal]{euler}

\usepackage{amsmath,amsthm,amssymb,slashed, mathtools}

\usepackage{MnSymbol}
\allowdisplaybreaks

\usepackage[all]{xy}
\xyoption{arc}

\usepackage{amsfonts}
\usepackage{enumerate}

 \swapnumbers

\theoremstyle{plain}
\newtheorem*{theorem*}{Theorem} 
\newtheorem*{proposition*}{Proposition} \newtheorem*{lemma*}{Lemma}
\newtheorem*{assumption*}{Assumption}
\newtheorem*{conjecture*}{Conjecture}

\newtheorem{theorem}[equation]{Theorem} 
\newtheorem{lemma}[equation]{Lemma}
\newtheorem{corollary}[equation]{Corollary}
\newtheorem{proposition}[equation]{Proposition}

\theoremstyle{definition}
\newtheorem{definition}[equation]{Definition}

\newtheorem{example}[equation]{Example}
\newtheorem{examples}[equation]{Examples}
\newtheorem{notation}[equation]{Notation}

\newtheorem{remark}[equation]{Remark}
\newtheorem{blank}[equation]{}

\newtheorem*{remark*}{Remark}
\newtheorem*{remarks*}{Remarks}
\newtheorem*{observation*}{Observation}

\theoremstyle{remark}

\numberwithin{equation}{subsection}

\newcommand{\lie}[1]{\mathfrak{#1}}
\DeclareMathOperator{\Ad}{Ad}

\newcommand{\Compact}{\mathfrak{K}}

\newcommand{\N}{\mathbb{N}}
\newcommand{\R}{\mathbb{R}}
\newcommand{\C}{\mathbb{C}}
\newcommand{\Z}{\mathbb{Z}}

 \newcommand{\beq}[1]{\begin{equation} \label{#1}}
\newcommand{\eeq}{\end{equation}}

 \DeclareMathOperator{\Hom}{Hom}

\DeclareMathOperator{\Ind}{Ind}

\DeclareMathOperator{\rank}{rank}

\DeclareMathOperator{\SL}{SL}
\DeclareMathOperator{\GL}{GL}
\DeclareMathOperator{\SU}{SU}

\usepackage[hidelinks]{hyperref}
\usepackage{tikz}

\def\X{X}
\def\R{\mathbb R}
\def\C{\mathbb C}
\def\N{\mathbb N}
\def\Z{\mathbb Z}

\def\R{\mathbb R}
\def\C{\mathbb C}
\def\N{\mathbb N}
\def\Z{\mathbb Z}

\def\q{\mathfrak{q}}

\newcommand{\s}[1]{\langle #1 \rangle}

\newcommand{\fah}{\mathfrak{a}_\mathfrak{h}}
\newcommand{\fbh}{\mathfrak{b}_\mathfrak{h}}
\newcommand{\ftq}{\mathfrak{t}_\mathfrak{q}}

\newcommand{\faq}{\mathfrak{a}_\mathfrak{q}}
\newcommand{\fbq}{\mathfrak{b}_\mathfrak{q}}
\newcommand{\fth}{\mathfrak{t}_\mathfrak{h}}

\newcommand{\fsl}{\mathfrak{sl}}

\newcommand{\bZ}{\mathbb{Z}}
\newcommand{\bR}{\mathbb{R}}
\newcommand{\bC}{\mathbb{C}}

\newcommand{\fa}{\mathfrak{a}}
\newcommand{\fb}{\mathfrak{b}}
\newcommand{\fg}{\mathfrak{g}}
\newcommand{\fh}{\mathfrak{h}}
\newcommand{\fk}{\mathfrak{k}}
\newcommand{\fl}{\mathfrak{l}}
\newcommand{\fp}{\mathfrak{p}}
\newcommand{\fq}{\mathfrak{q}}
\newcommand{\ft}{\mathfrak{t}}
\newcommand{\fm}{\mathfrak{m}}
\newcommand{\fn}{\mathfrak{n}}

\newcommand{\supp}{\mathrm{supp}}
\newcommand{\temp}{\mathrm{temp}}
\newcommand{\ds}{\mathrm{ds}}

\begin{document}

\title[Well-tempered symmetric spaces]{Well-tempered symmetric spaces}

\author[A. Afgoustidis]{Alexandre Afgoustidis}
\author[P. Hochs]{Peter Hochs}
\author[S. Nishikawa]{\linebreak Shintaro Nishikawa} 
\author[Y. Song]{Yanli Song}

\address{A. Afgoustidis: CNRS \& Institut Élie Cartan de Lorraine, Nancy \& Metz, France}
\email{alexandre.afgoustidis@math.cnrs.fr}

\address{P. Hochs: Institute for Mathematics, Astrophysics and Particle Physics, Radboud University, Nijmegen, the Netherlands}
\email{p.hochs@math.ru.nl}

\address{S. Nishikawa: School of Mathematical Sciences, University of Southampton, Southampton, United Kingdom}
\email{s.nishikawa@soton.ac.uk}

\address{Y. Song: Department of Mathematics, Washington University, St. Louis, MO, 63130, United States}
\email{yanlisong@wustl.edu}




\begin{abstract}
Let $X=G/H$ be a real reductive symmetric space. 
We consider the support of the Plancherel measure for the regular representation of $G$ on $L^2(X)$, and its Fell topology. 

We first recall how the Plancherel formula for~$X$ (known from work of Delorme and van den Ban--Schlicktrull, among others) yields a description of the support in terms of the discrete spectra of smaller symmetric spaces. We then point out that this description simplifies for a class of real symmetric spaces~$X$ which we call well-tempered. These are the spaces whose tangent space at the identity coset admits a Cartan subspace whose centralizer in the Lie algebra of~$G$ is abelian. Well-tempered spaces are tempered in the sense of  Benoist and Kobayashi. 

Under mild assumptions on the disconnectedness of~$G$, we give a description of the support of the Plancherel measure for well-tempered symmetric spaces in terms of characters of Cartan subgroups of~$G$, and describe the Fell topology of the support. 

If $G$ is a complex group and $H$ a real form, then the support decomposes as a union
of explicit quotients of relative continuous-parameter spaces by finite Weyl groups. As an application, we describe the $C^*$-algebra associated to the regular representation of~$G$ on $L^2(X)$, and its $K$-theory.
\end{abstract}

\maketitle

\tableofcontents

\section{Introduction}

\subsection{The Plancherel theorem for reductive symmetric spaces}

Let \(X=G/H\) be a real reductive symmetric space. Much of harmonic analysis on~$X$ can be framed as the study of the quasi-regular unitary representation
\[
\lambda_X:G\longrightarrow U(L^2(X))
\]
of~$G$ on $L^2(X)$, where $\lambda_X(g)$, for $g \in G$, is the operator $f \mapsto f(g^{-1} \cdot)$ on~$L^2(X)$. For this one wants to understand the decomposition of \(\lambda_X\) into irreducible
unitary representations of \(G\); the corresponding Plancherel measure and the multiplicity spaces in the Plancherel formula; one wants an explicit version of the Fourier transform, etc.

 Plancherel theory for non-Riemannian symmetric spaces developed enormously between the 1970s and 1990s.  Rank-one hyperboloids and related
pseudo-Riemannian examples were studied explicitly in the 1970s: let us mention work of  Strichartz~\cite{Strichartz73}, 
Rossmann~\cite{Rossmann78}, Faraut~\cite{Faraut79}, among others. In this setting they developed explicit spectral decompositions, and methods based on invariant
differential operators methods, or spherical distributions.  Next came the structural work of Oshima and Sekiguchi on affine symmetric spaces and Poisson transformations
\cite{Oshima79,OshimaSekiguchi80}, and a substantial body of work on the discrete series by Flensted-Jensen~\cite{FJ80}, Oshima and Matuski~\cite{OM84}, Schlichtkrull~\cite{S83, S84}, and Vogan~\cite{V88}. This supplied much of the representation-theoretic
foundations on which the general theory later developed.

The general
Plancherel theorem for reductive symmetric spaces was established by
Delorme~\cite{Delorme98} and, independently,  van den Ban and Schlicht-krull
\cite{BS05-1,BS05-2}.  Accessible surveys and complementary
accounts are given in \cite{D05, S05, Ban05, BanFlenstedSchlichtkrull97}. Symmetric spaces of type \(G_{\C}/G_{\R}\) are better understood: in particular, Harinck studied the corresponding orbital integrals and proved a related inversion formula
\cite{Harinck92, Harinck94, Harinck95, Harinck98base, Harinck98orbital}; see Delorme's Bourbaki exposition~\cite{Delorme97Bourbaki}.   

In the setting of symmetric spaces, and in contrast to the situation for reductive groups, the spectral data used in the Plancherel theorem consist
not only of an irreducible representation of \(G\), but also of an
\(H\)-fixed distribution vector. This additional datum is central to the
formulation of the Plancherel formula \cite{Delorme98,BS05-1,BS05-2}.

\subsection{Aims of this paper}

The present paper studies the regular representation $\lambda_X$ of $G$ on $L^2(X)$ at a coarser level: we discuss the irreducible unitary representations of~$G$ that are weakly contained in $\lambda_X$, without the distributional
and multiplicity data. 
Thus our main object of study is the support 
\[
\supp(\lambda_X)
\]
of the quasi-regular representation, viewed as a subset of the unitary dual $\widehat{G}$ of~$G$. The Fell topology of $\widehat{G}$ induces a canonical topology on $\supp(\lambda_X)$, and we wish to understand $\supp(\lambda_X)$ as a topological (rather than a measure) space. This amounts to asking:  which irreducible
representations of \(G\) are seen by \(X\), and what is the ``geometry'' of the corresponding set of representations (or rather its  Fell topology)?

These questions, when applied to reductive groups rather than symmetric spaces, have traditionally been approached through operator algebras.  For the left regular representation of \(G\) on \(L^2(G)\), the corresponding support and operator algebra are the tempered dual \(\widehat G_{\temp}\) and the reduced group
\(C^*\)-algebra \(C_r^*(G)\), respectively. These have been studied in great detail. 

In our setting 
\(\supp(\lambda_X)\) is the also the spectrum of a  \(C^*\)-algebra 
\(C^*_{\lambda_X}(G)\), namely the operator-norm closure of the image of \(L^1(G)\) in the algebra
\(B(L^2(X))\) of bounded operators on $L^2(X)$. But little seems to have been done on either the structure of \(C^*_{\lambda_X}(G)\) or the topology of \(\supp(\lambda_X)\).

The present paper does not reprove the Plancherel theorem, and its goals are modest:  we try to understand what the Plancherel theorem says about the topological space \(\supp(\lambda_X)\), and whether the space \(\supp(\lambda_X)\) can be understood in simple terms -- in the spirit of what is known for reductive groups. We shall only succeed in arriving at a complete picture for certain classes of symmetric spaces which we call \emph{well-tempered}. These include the spaces of type $G_\C/G_\R$.

 We should mention that  considering only the support $\supp(\lambda_X)$ of the Plancherel measure, instead of the full Plancherel formula, disregards more than multiplicity
 questions: as we shall see the description of the support involves generalized \(\sigma\)-principal series representations, some of which are reducible. A finer analysis would have to determine which irreducible
 constituents admit \(H\)-fixed distribution vectors and, for each such
 constituent, identify the tempered \(H\)-fixed distribution vectors that enter
 the Plancherel multiplicity space. These constituent-level questions are
 difficult and fundamental problems of relative harmonic analysis, but they
 are finer than the support problem addressed here. 

In the rest of this introduction, we outline the themes and results of the paper: general observations about the Plancherel support, the definition of well-tempered symmetric spaces and how $\supp(\lambda_X)$ may be understood in that case, and what our results imply in  important examples including the spaces of type $G_\C/G_\R$. 
 
\subsection{The support of the Plancherel measure}\label{subsec-intro-support}

We first recall, following Delorme and van den Ban--Schlichtkrull, what the general Plancherel theorem says about the support $\supp(\lambda_X)$  (see Section \ref{subsec:support_Plancherel} for details). 
We shall assume throughout that $G$ is in Harish-Chandra's class and that $H$~is an open subgroup of the fixed-point group $G^\sigma$ for an involutive automorphism $\sigma$ of~$G$. Among the parabolic subgroups of~$G$ is a distinguished collection, the  \emph{cuspidal \(\sigma\)-parabolic subgroups} (see \S\,\ref{sec-cuspidal-sigma-parabolics}). For each such subgroup \(P=MAN\) of \(G\), the Lie algebra $\fa$ of~$A$ decomposes into the direct sum \(\fa=\fah\oplus\faq\) of $(+1)$ and $(-1)$-eigenspaces of $\sigma$ respectively,  and the symmetric space $M/(M \cap H)$ has \emph{relative discrete series} --- irreducible representations of~$M$ which occur discretely in $L^2(M/(M \cap H))$. 

To each relative discrete series
\(\xi\in\widehat M_{M\cap H,\ds}\) and each imaginary-valued linear form \(\nu\in i\faq^*\), associate the parabolically induced representation
\[
\pi_{P,\xi,\nu}=\Ind_P^G(\xi\otimes e^\nu\otimes 1),
\]
called a \emph{generalized \(\sigma\)-principal series representation}.
Define
\[
S_X
=
\bigcup_{P,\xi,\nu}
\Bigl\{
\tau\in\widehat G
\;\Bigm|\;
\tau \text{ is an irreducible constituent of } \pi_{P,\xi,\nu}
\Bigr\},
\]
where \(P\), \(\xi\), and \(\nu\) range over all such data. The Plancherel theorem gives:

\begin{theorem}[\cite{Delorme98, BS05-2}, see Theorem~\ref{thm_description0}]
Let $X=G/H$ be a reductive symmetric space. The subset 
 $\supp(\lambda_X)$ of $\widehat G$ is the closure of $S_X$ in the Fell topology.
\end{theorem}

Since the Plancherel formula is
naturally stated using the direct integral of Hilbert spaces over regular imaginary parameters, it initially yields only
\[
\supp(\lambda_X)=\overline{S_X}.
\]
Our first observation, which we prove in Section~\ref{sec:support} using $C^\ast$-algebras and Hilbert modules, is that  \(S_X\)
 is always closed, for any $X=G/H$ in Harish-Chandra's class.

\begin{proposition}[Theorem~\ref{thm_description}]
Let \(X=G/H\) be a reductive symmetric space.
Then \(S_X\) is a closed subset of \(\widehat G\). Consequently,
\[
\supp(\lambda_X)=S_X. 
\]
In other words, for any irreducible unitary representation \(\pi\) of \(G\),
\(\pi\) belongs to \(\mathrm{supp}(\lambda_{X})\) if and only if there exist a
cuspidal \(\sigma\)-parabolic subgroup \(P=MAN\) of \(G\), a representation
\(\xi \in {\hat{M}}_{M\cap H, \ds}\) in the discrete series of
\(M/(M\cap H)\), and \(\nu \in i\faq^*\), such that \(\pi\) is equivalent to an
irreducible constituent of the generalized \(\sigma\)-principal series
representation \(\pi_{P, \xi, \nu}\).
\end{proposition}

Thus the support problem requires a good understanding of the relative discrete series of the
Levi sub-symmetric spaces \(M/(M\cap H)\) and their behavior under cuspidal
\(\sigma\)-parabolic induction.  To obtain an explicit parame-trization, one
therefore needs a condition which is stable under taking Levi sub-symmetric spaces and which
makes the discrete series on each of these smaller symmetric spaces manageable.

\subsection{Well-tempered symmetric spaces}

This leads to the class of spaces considered in the paper.  Let  \(X=G/H\) be a reductive symmetric space, and let $\fq$ be  the $(-1)$-eigenspace $\fq$ of $\sigma$ on~$\fg$. Recall that a \emph{Cartan subspace} of~$\fq$ is a maximal abelian subspace of~\(\fq\) consisting of semisimple elements. We say $X$ is \emph{well-tempered} if there exists a Cartan subspace $\fbq \subset \fq$ whose Lie algebra centralizer
\(Z_\fg(\fbq)\) is abelian.  When that happens, the centralizer of every Cartan
subspace of \(\fq\) is abelian (Lemma~\ref{lem:welltemp-all-cartan}). An easy but crucial observation is that all Levi sub-symmetric spaces $M/(M\cap H)$ occurring in the Plancherel decomposition are then also well-tempered (see Lemma~\ref{lem:welltemp_hereditary}).

Well-tempered spaces are \emph{tempered} in the sense of Benoist and Kobayashi \cite{BK15}, which means all representations in $\supp(\lambda_X)$ are tempered as representations of~$G$. For complex symmetric spaces, Theorem~\ref{prop:complex-temp-welltemp} shows well-temperedness is
equivalent to temperedness; in general, we shall see that well-temperedness is a slightly stronger condition than temperedness.

The representation-theoretic effect of  well-temperedness  is first
visible when \(X\) has discrete series. Then \(X\) has a ``compact'' Cartan
subspace (see \S\,\ref{sec:cartan_and_rank}), and well-temperedness says that its centralizer is abelian. Much of the theory of the discrete series simplifies. The discrete spectrum is entirely described by the construction of Flensted-Jensen~\cite{FJ80}, and we point out several equivalent descriptions of the discrete series: by cohomological induction from an
abelian fundamental Cartan subgroup and, ultimately, by lowest \(K\)-types (where $K$ is a $\sigma$-stable maximal compact subgroup of~$G$).

Here is a summary. Following~\cite{AA24}, we call a representation
\emph{tempiric} if it is tempered, irreducible and has real infinitesimal character.

\begin{theorem}[Theorem~\ref{thm:ds_para}]
Let \(X=G/H\) be a well-tempered symmetric space. Suppose that \(G\) is a
real algebraic reductive linear group with abelian Cartan subgroups and that \(X\) has discrete
series. For an irreducible unitary representation \((\pi,V)\) of \(G\), the
following conditions are equivalent:
\begin{enumerate}[(1)]
\item \(\pi\) belongs to the discrete series of the symmetric space~$X$; 
\item \(\pi\) is tempiric with regular infinitesimal character, and its unique
lowest \(K\)-type \(\tau_\pi\) has nonzero $(K \cap H)$-fixed vectors; 
\item \(\pi\) is irreducibly induced from a discrete character datum $(TA, \Gamma, \lambda, 0)$ (see Definition~\ref{def:dataC})  based on the fundamental Cartan subgroup
\(TA=Z_G(\fl)\), where \(\lambda\in i\ftq^\ast\) is \(G\)-regular and $
\Gamma|_{T\cap H}=1$.
\end{enumerate}
Moreover, for any \((\pi,V)\) satisfying these equivalent conditions, the space of square-integrable $H$-invariant distribution vectors of~$\pi$ is isomorphic with the space of $(K\cap H)$-invariant vectors of~$\tau_\pi$. In particular,
\[
\dim_\bC (V^{-\infty})^H_{\ds}
=
\dim_\bC \tau_\pi^{K\cap H}.
\]
\end{theorem}

The theorem says that, for well-tempered symmetric spaces, the relative discrete spectrum is controlled by the
compact space \(K/(K\cap H)\): the lowest \(K\)-type detects the discrete spectrum of~$X$ inside the tempiric spectrum of~$G$, and also detects the dimension of the square-integrable \(H\)-fixed distribution
space for representations in the discrete series of~$X$.

Let us move beyond the discrete series, to the full Plancherel support $\supp(\lambda_X)$. For a general well-tempered symmetric space, this support is
obtained by inducing discrete data from smaller Levi symmetric spaces.
This leads to a simple parametrization of $\supp(\lambda_X)$, which we shall spell out uniformly in terms of a relative version of
Vogan's notion of character data for~$G$~\cite{Voganbook}.  In Section~\ref{subsec-chardata-X} below, we define a \emph{character datum for \(X\)} to be a quadruple
\((TA,\Gamma,\lambda,\nu)\) where $TA$ is a $\sigma$-stable Cartan subgroup of~$G$, $\Gamma$ is a character of~$T$, and $\lambda$, $\nu$ are linear forms on $\lie{t}$ and $\lie{a}$ respectively, satisfying certain regularity and compatibility conditions determined by~$X$. These include 
\[
\Gamma|_{T\cap H}=1,
\qquad
\lambda\in i\ftq^*,
\qquad
\nu\in i\faq^*,
\]
and regularity and compatibility conditions on \(\Gamma\) and \(\lambda\). We shall see that a character datum for~$X$ is a particular case of the data that Vogan uses to parametrize tempered representations in~\cite{Voganbook}. In particular, we may associate to it a \emph{standard tempered representation of~$G$}, a possibly reducible (but tempered) representation  \(X_G(TA,\Gamma,\lambda,\nu)\) of~$G$. See  Section~\ref{subsec:vogan-character-data} for the details.

We can then recast the description of $\supp(\lambda_X)$ in Section~\ref{subsec-intro-support}, for well-tempered symmetric spaces, into the following statement where the representations are parametrized by concrete character data.

\begin{theorem}[Theorem~\ref{thm:tempered spec}]
Let \(X=G/H\) be a well-tempered symmetric space, and suppose that \(G\) is
a real reductive linear group in Vogan's class. Let \(\pi\in\widehat G\) be an
irreducible unitary representation of \(G\). Then \(\pi\) belongs to
\(\operatorname{supp}(\lambda_X)\) if and only if \(\pi\) is equivalent to an
irreducible constituent of \(X_G(TA,\Gamma,\lambda,\nu)\) for some character
datum \((TA,\Gamma,\lambda,\nu)\) for \(X\).
\end{theorem}

The preceding theorem describes the support as a set and, in particular,
places it inside \(\widehat G_{\temp}\). Since we understand the geometry of  \(\widehat G_{\temp}\) very well, this makes it possible to understand $\supp(\lambda_X)$ as a topological space. 

The tempered dual \(\widehat G_{\temp}\) is famously non-Hausdorff, and its lack of Hausdorfness is precisely connected with the reducibility of the representations \(X_G(TA,\Gamma,\lambda,\nu)\). Since this reducibility depends on the representation theory of~$G$, and has little to do with~$X$, to understand the geometry of  $\supp(\lambda_X)$ amounts to describing the image of $\supp(\lambda_X)$ in the universal Hausdorff quotient \(\widehat G_{\temp,\mathrm{Haus}}\) of $\widehat{G}_{\temp}$ (for the notion of universal Hausdorff quotient see Bourbaki~\cite{BourbakiTG}). Therefore, we shall describe the image  $\supp(\lambda_X)_{\mathrm{Haus}}$ of $\supp(\lambda_X)$ in  \(\widehat G_{\temp,\mathrm{Haus}}\). 

It is not difficult to spell out the connected components of  
\(\widehat G_{\temp,\mathrm{Haus}}\) in terms of Vogan's parametrization. Each component is attached to a character datum of the form $(TA, \Gamma, \lambda, 0)$, and is homeomorphic with
\[
i\fa^*/W(\Gamma,\lambda),
\]
where $W(\Gamma, \lambda)$ is a certain finite group acting on $\fa^\ast$. 

An essential point to understand the geometry of $\supp(\lambda_X)$ is to understand which  character data for \(X\) contribute to a given connected component of
\(\widehat G_{\temp,\mathrm{Haus}}\). We shall see that several continuous families of representations can contribute to a fixed connected component of $\supp(\lambda_X)_{\mathrm{Haus}}$. These families are associated to distinct ``discrete character data for~$X$'' (meaning triples $(TA, \Gamma, \lambda)$ as above, where we forget $\nu$). Ultimately, this exhibits a connected component of $\supp(\lambda_X)_{\mathrm{Haus}}$ as the image in $i\fa^*/W(\Gamma,\lambda)$ of a finite union of ``branches'', all Weyl-translates of the subspace $i\faq^*\subset i\fa^*$.

Let us state the result more formally. Let \(\mathcal D_X\) be the set of discrete character data for \(X\), and let
\(\mathcal C_X\) be set of their \(K\)-conjugacy classes as discrete character data
for \(G\). In Section~\ref{subsec-shape-support}, we associate to any element \(c\in\mathcal C_X\) a canonical connected subset $\Sigma_X(c)$ of $\operatorname{supp}(\lambda_X)_{\mathrm{Haus}}$. If we choose a representative 
\(D_0=(TA,\Gamma,\lambda)\) of~$c$, then $\Sigma_X(c)$ identifies with a subset of  $i\fa^*/W(\Gamma,\lambda)$ which is the image of a union of linear subspaces of $i\fa^*$ -- all obtained from the subspace
\(i\faq^*\subset i\fa^*\) under the action of a Weyl group $W(\lambda)$ which is distinct from $W(\Gamma, \lambda)$ in general.

\begin{theorem}[Theorem~\ref{thm:shape-practical}]\label{thm-shape-intro}
Let \(X=G/H\) be a well-tempered symmetric space, where $G$ is a linear reductive group with abelian Cartan subgroups.   Then 
\[
\operatorname{supp}(\lambda_X)_{\mathrm{Haus}}
=
\bigsqcup_{c\in\mathcal C_X}\Sigma_X(c).
\]
\end{theorem}

This is an explicit description of $\supp(\lambda)_{\mathrm{Haus}}$ as a topological space, for well-tempered symmetric spaces. We include several results which make it easier to determine the sets $\Sigma_X(c)$ in significant special cases (see Propositions~\ref{prop:same-cartan} and ~\ref{prop:branch-collapse}).

Distinct ``branches'' in one of the sets $\Sigma_X(c)$ have disjoint regular loci (they have no regular element in common), but may meet along
singular strata. They certainly meet at~$0$. We spell out a simple but interesting example in
\(\SL(3,\R)/\GL(2,\R)_0\): one connected component of
\(\operatorname{supp}(\lambda_X)_{\mathrm{Haus}}\) is a tripod consisting of
three closed half-lines meeting at the image of \(0\). 
    
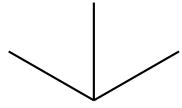
\begin{center}
\bigskip
\begin{tikzpicture}
  \coordinate (A) at (0,0);
  \coordinate (C) at (30:1.3);  
  \coordinate (D) at (90:1.3);  
  \coordinate (E) at (150:1.3); 
                              
  \draw[thick] (C) -- (A);
      
  \draw[thick] (D) -- (A);
      
  \draw[thick] (E) -- (A);

\end{tikzpicture}
\bigskip
\end{center}
\begin{center}
\begin{minipage}{0.92\textwidth}
   \footnotesize This is a picture of a connected component of the Plancherel spectrum for $\SL(3,\R)/\GL(2,\R)_0$, studied in Example~\ref{subsec_example_SL3}. The component has three branches. Its topological $K$-theory is $\Z^2$ in degree~$1$. This contrasts to the situation for reductive groups where each component contributes at most one factor~$\Z$ to the $K$-theory.
\end{minipage}
\end{center}
\vspace{0.2cm}
\normalsize

Let us mention that in a finer Plancherel parametrization, which involves keeping track of distributional data, some of the intersections between bran-ches may appear to be artificial intersections.

\subsection{Symmetric spaces of the form $G_{\C}/G_{\R}$}

The final section explains what our more general results imply for symmetric spaces of type
\(G_{\C}/G_{\R}\), and an operator-algebraic corollary of our results for these spaces.

In this class much simplifies: the ambient complex group has connected
Cartan subgroups, so the character \(\Gamma\) in Vogan’s character datum is determined by \(\lambda\), and  \(W(\Gamma,\lambda)=W(\lambda)\). As a result, each of the sets $\Sigma_X(c)$ in Theorem~\ref{thm-shape-intro} collapses to a single branch.  Moreover, Wallach's irreducibility
theorem for the unitary principal series implies that the tempered
standard representations are irreducible, so the tempered dual \(\widehat G_{\temp}\) is Hausdorff. Summing up, we obtain the following description of $\supp(\lambda_X)$ (see \S\,\ref{sec-complex-support} for  notation  on the Weyl groups $W^\sigma(\cdot)$ involved in the statement).

\begin{theorem}[Theorem~\ref{thm:GcGr-main}]\label{thm-gc-gr-intro}
Let \(X=G/H\) be a symmetric space of type \(G_{\mathbb C}/G_{\mathbb R}\).
Then:
\begin{enumerate}[(1)]
\item \(X\) is well-tempered.

\item For every \(\theta\)-stable Cartan subalgebra
\(\fbh=\fth\oplus\fah\subset\fh\), and for every discrete character datum
\((TA,\Gamma,\lambda)\) for \(X\) based on
\(TA=Z_G(\fbh\oplus i\fbh)\), the map
\[
 i\faq^\ast/W^\sigma(\lambda)
 \longrightarrow \widehat G_{\temp},
 \qquad
 \overline\nu\longmapsto X_G(TA,\Gamma,\lambda,\nu),
\]
is a homeomorphism onto its image.  Denote this image by
$\widehat X_{\fbh,\lambda}$; then
\[
\supp(\lambda_X)
=
\bigsqcup_{[\fbh]}
\ \bigsqcup_{[(\Gamma,\lambda)]}
\widehat X_{\fbh,\lambda},
\]
where \([\fbh]\) runs over the \(K\)-conjugacy classes of \(\theta\)-stable
Cartan subalgebras of \(\fh\), and \([(\Gamma,\lambda)]\) runs over the
\(W^\sigma(\fbh)\)-orbits of discrete character data for \(X\) based on
\(TA=Z_G(\fbh\oplus i\fbh)\).

\item \(X\) has discrete series if and only if the real Lie algebra \(\fh\) is
split.  In that case the discrete series of \(X\) is precisely the set of
representations
\[
X_G(TA,\Gamma,\lambda,0)
\]
attached to split Cartan subalgebras of \(\fh\) and to the corresponding
discrete character data for \(X\).  Equivalently, it is the set of tempiric
representations of \(G\) with regular infinitesimal character whose unique
lowest \(K\)-type is \(K\cap H\)-spherical.

\item A representation in \(\supp(\lambda_X)\) belongs to the discrete series
of \(X\) if and only if it is an isolated point of \(\supp(\lambda_X)\).
\end{enumerate}
\end{theorem}

The last item characterizing discrete series as isolated points in the spectrum holds more generally for any well-tempered symmetric space (see Corollary ~\ref{cor_singleton}). However, to the best of our knowledge, it is unclear whe-ther the same characterization holds for an arbitrary symmetric space. This question seems nontrivial since, for a general  space $X$, some representations in the discrete series of $X$ can have non-regular infinitesimal characters, although their relative parameters are regular in an appropriate sense. 

The description of $\supp(\lambda_X)$ in Theorem~\ref{thm-gc-gr-intro} can be easily upgraded to a complete description of the $C^\ast$-algebra $C^\ast_{\lambda_X}(G)$ attached to the quasi-regular representation $\lambda$.   The
 $C^\ast$-Plancherel theorem for the reduced \(C^*\)-algebra of a connected
complex semisimple group \cite{CCH16}, combined together with the support decomposition above,
leads to the following quotient description.

\begin{corollary}[Proposition~\ref{cor:GcGr-Cstar}]
With the notation of Theorem~\ref{thm-gc-gr-intro}, there is a canonical
isomorphism
\[
C^*_{\lambda_X}(G)
\cong
\bigoplus_{[\fbh]}
\ \bigoplus_{[(\Gamma,\lambda)]}
C_0\bigl(i\faq^\ast/W^\sigma(\lambda),\Compact\bigr).
\]
\end{corollary}

As an application, we compute the operator $K$-theory of $C^\ast_{\lambda_X}(G)$, and the topological $K$-theory of the support $\supp(\lambda_X)$. We explain  when the \(K\)-theory  is nonzero (see Corollary~\ref{thm_K_GcGr}): in the \(G_{\C}/G_{\R}\) case this is equivalent to the
quasi-splitness of the real form \(\fh\), or equivalently to the condition that 
the centralizer of a Cartan subspace in \(\fk\cap\fq\) be abelian.

\subsection*{Organization of the paper} 
Section~\ref{sec:support} recalls the necessary representation-theoretic
preliminaries for reductive symmetric spaces, including \(H\)-fixed
distribution vectors and cuspidal \(\sigma\)-parabolic subgroups, and proves
the closed support theorem extracted from the Plancherel formula. Section~\ref{sec:dis} treats symmetric spaces with a compact Cartan subspace having abelian centralizer, using cohomological induction and lowest \(K\)-types to
describe the discrete series and the square-integrable \(H\)-fixed
distributions.  Section~\ref{sec:welltempered} introduces well-tempered
symmetric spaces, proves their hereditary property, and relates them to the
Benoist--Kobayashi notion of tempered homogeneous space.
Section~\ref{sec:discrete-welltemp} reformulates the support in terms of
relative Vogan character data and proves the finite-branch description of the
Hausdorff support, including the tripod example.  Section~\ref{sec:complex}
specializes the theory to symmetric spaces of type \(G_{\C}/G_{\R}\) and
derives the \(C^*\)-algebraic corollaries.  The appendix proves a torus form
of the compact Cartan--Helgason criterion which is needed in the main text in the presence of disconnectedness for $G$ or its Levi subgroups. 

\subsection*{Acknowledgments}

We are very much indebted to Nigel Higson, who originally brought us together to think about possible connections between harmonic analysis on symmetric spaces and operator algebras. We benefited greatly from many discussions with him. These discussions were partially made possible by the Radboud Excellence Initiative, which funded visits by Higson to Radboud University in 2022 and 2023.

We also thank Erik van den Ban, Toshiyuki Kobayashi, Bernhard Krötz,
Job Kuit, Henrik Schlichtkrull, Haluk S\c{e}ngün, and David Vogan for helpful discussions and correspondence. Their insights clarified several aspects of
the representation theory and harmonic analysis underlying this work. This
project was developed in part during the 2025 trimester 
\emph{Representation Theory and Noncommutative Geometry} at the Institut
Henri Poincaré, whose stimulating environment served as a catalyst for the
work.

Afgoustidis's research was partially supported by project OpART of the Agence Nationale de la Recherche (ANR-23-CE40-0016).

Hochs's research was partially supported by Dutch Research Council grants OCENW.M.21.176, OCENW.M.23.063 and OCENW.M.24.386.

Song’s research was partially supported by the NSF grant DMS-1952557 and Simons Foundation grant MPS-TSM-00014295.

\subsection*{Disclosure}
All mathematical content, including the results, arguments, and proofs, was developed by the authors without computer assistance. The paper was entirely written by the authors; an AI tool was used solely for grammatical correction and stylistic polishing of parts of the exposition, which were then fully reworked by the authors.

\section{{The support of the Plancherel measure}} \label{sec:support}

In this section we recall some of the basics of harmonic analysis on a reductive symmetric space, and record the description of the support of the Plancherel measure on $L^2(G/H)$ that will be used throughout the paper.

\subsection{Setup and Notation}

\begin{blank}
Let $G$ be a real reductive group in Harish-Chandra's class as in~\cite[Appendix]{Ban05}.
Let $\sigma$ be an involutive automorphism of $G$, and let $H$ be an open subgroup of $G^\sigma$. We say $(G, H)$ is a \emph{reductive symmetric pair},
and call the homogeneous space $X=G/H$ a \emph{reductive symmetric space}. 

For any Lie group~$U$, we shall use the Fraktur letter~$\mathfrak{u}$ for the real Lie algebra of~$U$,
and write  $\mathfrak{u}^{\mathbb{C}}$ or $\mathfrak{u}_\bC$ for the complexified Lie algebra.
We shall write $U_0$ for the identity component of~$U$. 
\end{blank}

\begin{blank} We shall need to fix a convenient maximal compact subgroup of~$G$. In our setting there exists a Cartan involution~$\theta$ of~$G$ that commutes with~$\sigma$  (see \cite[Proposition 1.1]{Ban87-1}); the fixed-point set $K=G^\theta$ is a maximal compact subgroup of~$G$, and has the additional property that $K \cap H = H^\theta$ is a maximal compact subgroup of~$H$. The $K$-homogeneous space $K/(K \cap H)$ is a compact symmetric space\footnote{If $G$ is disconnected, $K$ need not belong to Harish-Chandra's class, but we shall still call $K/K\cap H$ a compact symmetric space when $K$ is compact.}, which will play a role in the paper. 

The involutions $\sigma$ and $\theta$ induce decompositions of the Lie algebra~$\fg$ which are critical to the theory. We have the direct sum decompositions
\[
\fg=\fh\oplus \fq,
\qquad
\fg=\fk\oplus \fp,
\]
where $\fh$ and $\fq$ are the $(\pm 1)$-eigenspaces of $\sigma$, and $\fk$ and $\fp$ are the $(\pm 1)$-eigenspaces of $\theta$. 

If $\fa$ is a $\sigma$-stable subspace of~$\fg$, we write
\[
\fah=\fa\cap \fh,
\qquad
\faq=\fa\cap \fq,
\]
for the $(\pm 1)$-eigenspaces of~$\sigma$ on~$\fa$, so that $\fa = \fah \oplus \faq$. This convention will be used throughout.   
\end{blank}

\begin{blank} Fix once and for all a nonzero $G$-invariant measure on $X=G/H$. We denote by $\lambda_X$ the left regular representation of~$G$ on~$L^2(X)$. Thus $\lambda_X$ is equivalent to the (normalized, unitary) induction $\mathrm{Ind}_H^G 1_H$, where $1_H$ is the trivial representation of $H$ (see \cite[Appendix E]{BHV}). 
\end{blank}

\subsection{$H$-fixed distributions and the $H$-spherical dual}

\begin{definition}[{\cite[p.~14]{Ban_note}, \cite[p.~13]{Ban05}}]\label{def_Hfixed_dist}
Let $(\pi,V)$ be a unitary representation of $G$.
We denote by $V^\infty$ the Fr\'echet space of smooth vectors, and by $V^{-\infty}$ its strong dual. The action of~$G$ on~$V$ determines an action on~$V^{-\infty}$. An \emph{$H$-fixed distribution vector} for $\pi$ is a continuous linear functional $\eta\in V^{-\infty}$ which is $H$-invariant.
We write
\[
(V^{-\infty})^H
\]
for the space of all such $H$-fixed distributions for $(\pi, V)$.

Suppose $\eta \in (V^{-\infty})^H$. For all $v \in V^\infty$, we let $\eta_v\colon G/H \to \C$ denote the $C^\infty$ function defined by $\eta_v(gH) = \eta(\pi(g)^{-1}v)$ for all $g \in G$. We define  
\[ m_\eta\colon V^\infty \to C^\infty(G/H)\]
by setting $m_\eta(v)=\eta_v$ for all $v \in V^\infty$. 
\end{definition}

\begin{lemma}[{\cite[Lemma 2.1]{Ban_note}}, {\cite[Lemma 2.1]{Ban05}}] \label{lem_Hdist_embedding}
Let $(\pi,V)$ be an irreducible unitary representation of $G$. The map $\eta \to m_\eta$ defines a linear isomorphism from $ (V^{-\infty})^H$ to the space $\mathrm{Hom}_G(V^\infty, C^\infty(G/H))$ of $G$-equivariant continuous linear maps $V^\infty\to C^\infty(G/H)$. If $\eta\neq 0$, then $m_\eta$ is an embedding.
\end{lemma}

\begin{remark}
In \cite[p.~14]{Ban_note}, \cite[p.~13]{Ban05}, $H$-fixed distributions ($H$-fixed generalized vectors) are taken to be continuous anti-linear functionals on $V^\infty$.
Passing from a linear functional $\eta$ to the anti-linear functional $v\mapsto \overline{\eta(v)}$ gives an anti-linear isomorphism between the two conventions.
\end{remark}

\begin{definition}[{\cite[p.~14]{Ban05}}] Let $(G, H)$ be a reductive symmetric pair. The $H$-spherical unitary dual of $G$, denoted by  ${\widehat{G}}_H$, is the subset of the unitary dual $\widehat{G}$ consisting of the equivalence classes of irreducible unitary representations with a nonzero space of $H$-fixed distributions. 
\end{definition}

\subsection{The discrete series}

\begin{definition}[{\cite[p.~16]{Ban05}}]
Let $(\pi, V)$ be an irreducible unitary representation of~$G$. We say $\pi$ is \emph{in the discrete series of $X=G/H$} if it is equivalent to an irreducible subrepresentation of the regular representation~$\lambda_X$ of~$G$ on~$L^2(X)$, i.e. if
\[
\mathrm{Hom}_G(V, L^2(X)) \neq0.
\]
We denote by ${\widehat{G}}_{H, \ds}$ the set of equivalence classes of irreducible unitary representations of~$G$ in the discrete series of~$X$;  we say  $X$ \emph{has a discrete series}, or \emph{has discrete series}, when this set is nonempty. 
\end{definition}

 \begin{definition} Let $(\pi, V)$ be an irreducible unitary representation of~$G$ on a Hilbert space~$V$, with space of smooth vectors $V^\infty$.  We call an element $\eta\in (V^{-\infty})^H$  \emph{square-integrable} if there exists $v \in V^\infty\smallsetminus\{0\}$ such that the element~$\eta_v$ of $C^\infty(X)$ (defined in Definition~\ref{def_Hfixed_dist}) is a square-integrable function on~$X$. We denote by $(V^{-\infty})^H_{\ds}$ the vector space of square-integrable $H$-fixed distributions on $V$.
\end{definition}

\begin{lemma}\label{lem_sqdis_iff} Let $(\pi, V)$ be an irreducible unitary representation of~$G$, and let $\eta$ be an element of $(V^{-\infty})^H$. 
\begin{enumerate}[(a)]

\item The element $\eta$ is square-integrable if and only if for every $v \in V^\infty$, the function $\eta_v$ is square-integrable on~$X$.
\item Suppose $\eta$ is square-integrable. The map \mbox{$m_\eta\colon V^\infty \to C^\infty(X)$} has image in $C^\infty(X)\cap L^2(X)$ and extends uniquely to a $G$-equivariant continuous map $V \to L^2(X)$.
\end{enumerate} 
\end{lemma}
\begin{proof} Assertion~(a) immediately follows from~(b): indeed,  if $\eta_v$ is square-integrable on~$X$, then $\eta$ is square-integrable by definition; and if~$m_\eta$ extends uniquely to a $G$-equivariant continuous map $V \to L^2(X)$, then $\eta_v$ is square-integrable for all $v \in V^\infty$. 

Let us prove~(b). The case $\eta=0$ is immediate. Suppose $\eta\neq0$ is square-integrable. Choose $v\in V^\infty\smallsetminus\{0\}$ such that $m_\eta(v)=\eta_v\in C^\infty(X)\cap L^2(X)$. By applying Gårding smoothing $\lambda(f)v = \int_G f(g)\pi(g)vdg$ for some appropriate $f\in C^\infty_c(G)$ and by replacing $v$ with $\lambda(f)v$ if necessary, we may assume that $\eta_v\in L^2(X)^\infty$ (see~\cite[Proposition 3.14]{Knapp}). By projecting $v$ to its nonzero $K$-isotypical component, we can further assume that $v$ is $K$-finite.

Set
\[
M_\eta=m_\eta(V_K)\subset C^\infty(X),
\]
where $V_K$ denotes the space of $K$-finite vectors in~$V$. By
Lemma~\ref{lem_Hdist_embedding}, the map $m_\eta$ is injective; hence
$M_\eta$ is an irreducible $(\fg,K)$-module isomorphic to~$V_K$. We have
\[
0\neq \eta_v
\in M_\eta\cap\bigl(L^2(X)^\infty\bigr)_K.
\]
Since
$\bigl(L^2(X)^\infty\bigr)_K$ is a $(\fg,K)$-module and $M_\eta$ is
irreducible, the $(\fg,K)$-module generated by $\eta_v$ is all of $M_\eta$.
Hence $m_\eta$ restricts to an embedding
\[
m_\eta|_{V_K}\colon V_K\hookrightarrow L^2(X).
\]
Pulling back the $L^2$-inner product along $m_\eta|_{V_K}$ therefore defines a positive-definite invariant Hermitian form on $V_K$.
By uniqueness of invariant Hermitian forms on an irreducible admissible $(\fg,K)$-module \cite[proof of Theorem~8]{HC53}, this form is proportional to the one induced from the given Hilbert structure on $V$. Thus, $m_\eta|_{V_K}$ extends to a positive scalar multiple of a $G$-equivariant isometric embedding
\[
V\hookrightarrow L^2(X),
\]
whose restriction to $V^\infty$ coincides with $m_\eta$ by continuity and by density of~$V_K$. Thus $m_\eta(V^\infty)\subset C^\infty(X)\cap L^2(X)$, and $m_\eta$ extends to the embedding $V\to L^2(X)$ above. The uniqueness is clear since $V^\infty$~is dense in~$V$.
\end{proof}

\begin{lemma}[{\cite[p.~16]{Ban05}}]\label{lem_sq_int_dis} Let $(\pi,V)$ be an irreducible unitary representation of~$G$. For  $\eta \in  (V^{-\infty})^H_{\ds}$, let $\widetilde{m_\eta}$ be the $G$-equivariant map $V \to L^2(X)$ defined in Lemma~\ref{lem_sqdis_iff}(b). The correspondence $\eta\mapsto \widetilde{m_\eta}$ induces a linear isomorphism
    \[
  (V^{-\infty})^H_\ds \simeq  \Hom_G(V, L^2(G/H)).
\]

  In particular, $\pi\in {\widehat{G}}_{H, \ds}$ if and only if $(V^{-\infty})_\ds^H\neq 0$.
\end{lemma}
\begin{proof}
 The injectivity of the map $\eta\mapsto \widetilde{m_\eta}$, $(V^{-\infty})^H_\ds \to  \Hom_G(V, L^2(G/H))$, follows from the fact that the map  $\eta\mapsto m_\eta$ is an isomorphism between $(V^{-\infty})^H $ and $\mathrm{Hom}_G(V^\infty,C^{\infty}(G/H))$ (Lemma~\ref{lem_Hdist_embedding}). 
 
 For $T\in  \Hom_G(V, L^2(G/H))$, the restriction $T|_{V^\infty}$ must be an element of $\mathrm{Hom}_G(V^\infty,C^{\infty}(G/H))$; by Lemma~\ref{lem_Hdist_embedding} it is $m_\eta$ for some $H$-fixed distribution vector $\eta\in (V^{-\infty})^H$. The function $\eta_v=m_\eta(v)=T|_{V^\infty}(v)$ is square-integrable for all $v \in V^\infty$, so $\eta$ is square-integrable by  Lemma~\ref{lem_sqdis_iff}, and $T|_{V^\infty} = m_\eta$. By continuity, this implies $T=\widetilde{m_\eta}$, showing the surjectivity.

 \end{proof}

\begin{blank} To summarize the above: 
 \emph{an irreducible unitary representation~$\pi$ on~$V$ belongs to ${\widehat{G}}_{H, \ds}$ if and only if $(V^{-\infty})^H_{\ds}\neq 0$.} The multiplicity of $V$ in the discrete series of $L^2(G/H)$ is given by
 \[
 \dim_{\mathbb{C}}(\mathrm{Hom}_G(V, L^2(G/H)))=\dim_{\mathbb{C}}((V^{-\infty})^H_{\ds}).
 \]
Given an embedding $T\colon V\to L^2(G/H)$, the corresponding   $\eta \in (V^{-\infty})^H_\ds$ is obtained by composing $T|_{V^\infty}\colon V^\infty\to L^2(G/H)\cap C^\infty(G/H)$ with the evaluation map  $\delta_{eH}\colon C^\infty(G/H)\to \bC$ at the identity coset.
\end{blank}

\begin{blank} \emph{Cartan subspaces and the rank condition.}\label{sec:cartan_and_rank} Recall that~$\fq$ denotes the $(-1)$-eigenspace of $\sigma$ on~$\fg$; this vector subspace of $\fg$ can be identified with the tangent space to $G/H$ at $eH$. 
A \emph{Cartan subspace} of $\fq$ is a maximal abelian subspace of~$\fq$ consisting of semisimple elements. Every Cartan subspace of~$\fq$ is conjugate by $H_0$ to a $\theta$-stable Cartan subspace \cite[p.~406, Remark]{OM80}. All Cartan subspaces of $\fq$ have the same dimension; this common dimension is called the (real) \emph{rank} of $X$ and is denoted by
\[
\rank(X)=\rank(G/H).
\]
Since $K/(K \cap H)$ is itself a symmetric space, we may consider the rank  of $K/(K\cap H)$. 
 We say $G/H$  \emph{satisfies the equal-rank condition} if 
 \[
 \rank(G/H) = \rank(K/(K\cap H)).
 \]
 This is equivalent to the existence of a Cartan subspace of $\fq$ which is contained in $\fk \cap \fq$; such a subspace is called a \emph{compact Cartan subspace of~$\fq$} for~$G/H$. If there exists a compact Cartan subspace of $\fq$, then it is unique up to conjugation by $(K\cap H)_0$.
\end{blank}

 \begin{theorem}[\cite{FJ80}, \cite{OM84}, {\cite[Theorem 16.1]{BS05-1}}] \label{thm:equal-rank} The discrete series of a reductive symmetric space $X=G/H$ is nonempty if and only if $X$ satisfies the equal-rank condition.
\end{theorem}

\subsection{Cuspidal \texorpdfstring{$\sigma$}{sigma}-parabolics and induced representations}\label{sec-cuspidal-sigma-parabolics}

Throughout this section, \(G/H\) is a reductive symmetric space. 

\begin{blank}

Fix a $\theta$-stable Cartan subspace
$\fbq=\ftq\oplus \faq$
of $\fq$, where
\[
\ftq=\fbq\cap \fk,
\qquad
\faq=\fbq\cap \fp.
\]
Let $L=Z_G(\faq)$ be the centralizer of $\faq$ in~$G$. Then $L$ is a $\sigma$-stable and $\theta$-stable Levi subgroup of $G$.
Write $L=MA$ for its Langlands decomposition. The group $M$ is reductive and in Harish-Chandra's class. Since $\ftq\subset \fm\cap \fq$ is a compact Cartan subspace of $\fm \cap \fq$, the smaller symmetric space
\[
M/(M\cap H)
\]
has discrete series by Theorem~\ref{thm:equal-rank}.

Choose a positive system $\Sigma^+(\fg,\faq)$ for the restricted root system $\Sigma(\fg,\faq)$, and let
\[
\mathfrak{n}=\bigoplus_{\alpha\in \Sigma^+(\fg,\faq)} \fg_\alpha.
\]
Then \(\fm\oplus\fa\oplus\fn\) is a \(\sigma\theta\)-stable parabolic subalgebra of \(\fg\) containing \(\fbq\). We define \(P=LN=MAN\), where \(N=\exp(\fn)\).

This discussion determines a map 
\begin{equation}\label{sigma_st_parab} \{ (\fbq, \Sigma^+(\fg, \faq))\} \to \{ \text{parabolic subgroups of~$G$}\},\end{equation}
sending a pair consisting of a $\theta$-stable Cartan subspace of~$\fq$ and a positive system to the corresponding $\sigma\theta$-stable parabolic subgroup $P$.
\end{blank}

\begin{definition} A \emph{cuspidal $\sigma$-parabolic subgroup} of $G$ is a parabolic subgroup of~$G$ in the image of~\eqref{sigma_st_parab}. We denote by $\mathcal{P}_{\sigma,\mathrm{cusp}}$ the collection of all cuspidal $\sigma$-parabolic subgroups of $G$. 
\end{definition}

\begin{definition} Two cuspidal $\sigma$-parabolic subgroups $P_1$, $P_2$ of $G$ are called \emph{$\sigma$-associate} if $(\fa_1)_\fq$ and $(\fa_2)_\fq$ are conjugate by~$K$. We denote by $\mathcal{P}_{\sigma,\mathrm{cusp}}/\!\sim$ the set of $\sigma$-associate classes of cuspidal $\sigma$-parabolic subgroups of $G$. 
\end{definition}

The map~\eqref{sigma_st_parab} descends to a bijection between $\mathcal{P}_{\sigma,\mathrm{cusp}}/\!\sim$ and the set of $K$-conjugacy classes of $\theta$-stable Cartan subspaces of $\fq$.

\begin{definition} \label{def:sigma-principal-series}
Let $P=MAN$ be a cuspidal $\sigma$-parabolic subgroup of $G$. Let $(\xi, V_\xi)$ be a unitary irreducible representation of~$M$ in the discrete series of $M/(M\cap H)$. Let $\nu\in i\faq^\ast$. Write $e^\nu$ for the unitary character of $A$ given by
$e^\nu(a)=e^{\nu(\log a)}$ for $a\in A$. Here, we regard $\nu$ as a functional on $\fa$ by extending it by zero on $\fah$. We denote by
\[
\pi_{P,\xi,\nu}
=
\Ind_P^G(\xi\otimes e^\nu\otimes 1)
\]
the $G$-representation obtained by normalized unitary induction from these data. 
We refer to $\pi_{P,\xi,\nu}$ as a \emph{generalized $\sigma$-principal series representation}. We can realize $\pi_{P, \xi, \nu}$ on the space of (a.e. equivalence classes of) measurable functions $f\colon G\to V_\xi$ such that \[
f\mid_K \in L^2(K, V_\xi) \text{ and } f(gman)=e^{-\nu-\rho}(a)\xi(m)^{-1}f(g)
\]
for all $g \in G$ and $man \in MAN$, where $\rho$ is the half-sum of the roots in $\Sigma(\mathfrak{n}, \fa)$. The action of~$G$ on such ``functions'' is by left translation and the inner product  of ``functions'' is the $L^2$-inner product between their restrictions to~$K$:
\[
\s{f_1, f_2}\coloneq \int_{K} \langle f_1(k),  f_2(k) \rangle_{V_\xi}dk.
\]
\end{definition}

\subsection{The support of the Plancherel measure} \label{subsec:support_Plancherel}

For each cuspidal $\sigma$-parabolic subgroup $P=MAN$ of $G$, each $\xi\in \widehat M_{M\cap H,\ds}$, and each $\nu\in i\faq^\ast$,
consider the irreducible constituents of the $\sigma$-principal series representation $\pi_{P,\xi,\nu}$.
Define
\[
S_X
=
\bigcup_{P,\xi,\nu}
\Bigl\{
\tau\in \widehat G
\;\Bigm|\;
\tau \text{ is an irreducible constituent of } \pi_{P,\xi,\nu}
\Bigr\},
\]
where $P$, $\xi$, and $\nu$ range over all such data.
Our next statement is a direct consequence of the Plancherel theorem \cite{Delorme98, BS05-1, BS05-2}.

\begin{theorem}\label{thm_description0}
Let $X=G/H$ be a reductive symmetric space. The subset 
 $\supp(\lambda_X)$ of $\widehat G$ is the closure of $S_X$ in the Fell topology.
\end{theorem}

\begin{proof}
The Plancherel theorem for reductive symmetric spaces expresses $\lambda_X$ as a direct sum of direct integrals of generalized $\sigma$-principal series representations $\pi_{P,\xi,\nu}$, with $\nu$ ranging over regular elements of $i\faq^\ast$; see \cite[Theorem~10.9]{BS05-2} and \cite[Theorem~10.15]{Ban05}.
The finite-dimensional multiplicity spaces occurring in those formulas do not affect the support.
Hence $\supp(\lambda_X)$ is the closure of the union of the supports of the corresponding fiber representations.

For the reader's convenience, let us make the connection with the parame-trization of \cite{BS05-2}, in which the discrete-series data are  attached to symmetric spaces of the form
\[
M/(M\cap w^{-1}Hw),
\qquad
w\in W_{K\cap H}\backslash W_K(\faq^0)/Z_W(\faq),
\]
where $W_K(\faq^0)$ is the Weyl group of a fixed maximal abelian subspace $\faq^0\subset \fp \cap \fq$ containing $\faq$, $W_{K\cap H}$ is its subgroup generated by the image of the normalizer $N_{K\cap H}(\faq^0)$  of $\faq^0$ in $K\cap H$ and $Z_W(\faq)$ is the centralizer of $\faq$. Conjugation by $w$ identifies this symmetric space with $M^w/(M^w\cap H)$ where $M^w=wMw^{-1}$, transports the discrete-series representation accordingly, and replaces $P$ by a $\sigma$-associate cuspidal $\sigma$-parabolic subgroup attached to $w\faq w^{-1}\subset \faq^0$.
Thus the union of irreducible constituents occurring in the Plancherel formula is exactly the union defining $S_X$, except that the Plancherel formula only uses regular parameters.

Since the set of regular elements is dense in $i\faq^\ast$, and unitary parabolic induction depends continuously on the imaginary parameter in the Fell topology \cite{Fell1962}, allowing all $\nu\in i\faq^\ast$ changes the union only up to closure. Hence $\supp(\lambda_X)=\overline{S_X}$.
\end{proof}

We end this section by observing that $S_X$ is already closed in $\widehat G$:

\begin{theorem}\label{thm_description}
Let \(X=G/H\) be a reductive symmetric space.
Then \(S_X\) is a closed subset of \(\widehat G\). Consequently,
\[
\supp(\lambda_X)=S_X\subset \widehat G.
\]
In other words, for any irreducible unitary representation \(\pi\) of \(G\),
\(\pi\) belongs to \(\mathrm{supp}(\lambda_{X})\) if and only if there exist a
cuspidal \(\sigma\)-parabolic subgroup \(P=MAN\) of \(G\), a representation
\(\xi \in {\hat{M}}_{M\cap H, \ds}\) in the discrete series of
\(M/(M\cap H)\), and \(\nu \in i\faq^*\), such that \(\pi\) is equivalent to an
irreducible constituent of the generalized \(\sigma\)-principal series
representation \(\pi_{P, \xi, \nu}\).
\end{theorem}

We shall prove this  using the connection between the Fell topology of the unitary dual and operator algebras. Our next statement is a general fact about the Fell topology of the unitary dual, in which we use the connection between unitary representations of a reductive group and nondegenerate representations of its maximal $C^\ast$-algebra, as in \cite[Sections~3.4 and~13.9]{Dixmier}.
 \begin{proposition} \label{prop_closed_induction} 
 Let $G$ be a real reductive group. 
 Let $P$ be any parabolic subgroup of $G$,
 with Langlands decomposition $P=MAN$. 
 Let $\mathbf{\xi} = (\xi_n)_{n \in \N}$
 be a sequence  of unitary representations of~$M$.
 Let $\mathbf{Y} = (Y_n)_{n \in \N}$ be a sequence of closed subsets of $i\mathfrak{a}^\ast$.
 Let $S(P, \mathbf{\xi}, \mathbf{Y})$ be the subset of the unitary dual $\widehat{G}$ consisting of the equivalence classes for all irreducible constituents of the parabolically induced representations $\pi_{P, \xi_n, \nu}$ for $n \in \N$ and $\nu \in Y_n$. 
 
 If the maximal group $C^\ast$-algebra of~$M$ acts by compact operators in the direct sum representation $\bigoplus_{n \in \N} \xi_n$, then $S(P, \mathbf{\xi}, \mathbf{Y})$ is a closed subset of $\widehat{G}$.
 \end{proposition}

 \begin{proof}
 Let $C^*(G)$ be the maximal group $C^*$-algebra of $G$, and $C^*(L)$ be the maximal group $C^*$-algebra of $L=MA$. Let $C^*(G/N)$ be the Hilbert $C^*(G)$-$C^*(L)$ module  introduced in \cite[Section 3.2]{Clare13}, and denoted  $\mathcal{E}(G/N)$ there (see also \cite{CCH16}, where the notation $C^*(G/N)$ comes from). A key property of this module is that $C^*(G)$ acts compactly on $C^*(G/N)$: the proof of \cite[Proposition 4.5]{CCH16} for $C_r^*(G/N)$ works verbatim for $C^*(G/N)$. 
 
 Let $(\xi_n, V_n)_{n \in \N}$ be a sequence of unitary representations of $M$ such that $C^*(M)$ acts compactly on $V=\bigoplus_n V_n$ through $\xi=\bigoplus_n\xi_n$. Then $C^*(L)$ acts compactly on the $C^*(L)$-$C^*(A)$ module $\mathcal{E} = V \otimes C^*(A)\cong C_0(i\fa^*, V)\cong \bigoplus_n C_0(i\fa^*, V_n)$. By the theory of $C^*$-parabolic induction \cite{Clare13, CCH16}, we have $C^*(G/N)\otimes_{C^*(L)} \mathcal{E} \cong  \bigoplus_n C_0(i\fa^*, \mathrm{Ind}_P^GV_n)$ where, on the fiber at \(\nu\in i\fa^*\), \(G\) acts on
\(\Ind_P^G V_n\) by \(\pi_{P,\xi_n,\nu}\)
\cite[Corollary~3.5]{Clare13}. Moreover, the $C^\ast$-algebra $C^*(G)$ acts on $C^*(G/N)\otimes_{C^*(L)} \mathcal{E}$ compactly.
 
After restriction to the closed subspaces $Y_n\subset i\fa^*$, it follows that the family $\pi_{P, \xi_n, \nu}$ of unitary representations of $G$, for $n\in \N$ and for $\nu \in Y_n$, induces a $\ast$-homomorphism
\begin{equation}\label{eq-fourier-pi}
\pi\colon C^*(G) \to \bigoplus_{n} C_0(Y_n, \Compact(\mathrm{Ind}_P^GV_n)), 
\end{equation}
where $\Compact(\mathrm{Ind}_P^GV_n)$ is the $C^*$-algebra of compact operators on the Hilbert space $\mathrm{Ind}_P^GV_n$ of the representation $\pi_{P, \xi_n, \nu}$. Here, the composition of $\pi$ with the evaluation at a point $\nu$ in $Y_n$ gives a representation
\[
 \mathrm{ev}_{\nu} \circ \pi \colon C^*(G) \to  \Compact(\mathrm{Ind}_P^GV_n) 
 \]
which is equivalent to the representation of $C^\ast(G)$ on $\mathrm{Ind}_P^GV_n$   induced by the representation $\pi_{P, \xi_n, \nu}$ of $G$. 

Now, let $B$ be the image of $C^\ast(G)$ under~\eqref{eq-fourier-pi}.  Since $B$ is a subalgebra of $\bigoplus_{n} C_0(Y_n, \Compact(\mathrm{Ind}_P^GV_n))$, by extension of pure states \cite[Lemma 2.10.1]{Dixmier}, any irreducible representation of $B$ is equivalent to one that factors through the evaluation at some point $\nu$ in $Y_n$. Therefore, the spectrum $\widehat{B}$, which is a closed subspace of $\widehat G$, naturally coincides with the subset $S(P, \{\xi_n\}_n, \{Y_n\}_n)$ of $\hat G$. Hence $S(P, \mathbf{\xi}, \mathbf{Y})$ is closed. 
\end{proof}

\begin{corollary}
 \label{cor_closed_induction} Let $G$ be any real reductive group, and let $P=MAN$ be any parabolic subgroup of $G$. Let $\xi \in \hat{M}$ be an irreducible unitary representation of $M$, and let $Y$ be a closed subset of $i\mathfrak{a}^\ast$. The subset $S(P, \xi,Y) \subset \widehat G$, consisting of all the irreducible constituents of the  induced representations $\pi_{P, \xi, \nu}=\Ind^G_{MAN}(\xi\otimes e^{\nu}\otimes 1)$ for $\nu \in Y$, is closed in $\widehat{G}$. In particular, the subset $S(P, \xi) = S(P, \xi, i\mathfrak{a}^\ast)$ is closed in~$\widehat{G}$. 
\end{corollary}

 \begin{proof}[Proof of Theorem \ref{thm_description}]
 By Theorem \ref{thm_description0}, we only need to show that $S_X$ is closed in $\hat G$. We may write $S_X$ as the union of the subsets $S(P, \xi, i\faq^*)$ defined in Corollary \ref{cor_closed_induction}, as $P$ ranges over cuspidal $\sigma$-parabolic subgroups and $\xi$ ranges over $ \widehat{M}_{M\cap H, \ds}$. By Corollary~\ref{cor_closed_induction}, each subset $S(P, \xi, i\faq^*)$ is closed. We can say more: for fixed~$P$, the union $S_X(P)=\bigcup_{\xi\in {\hat{M}}_{M\cap H, \ds}} S(P, \xi, i\faq^*)$ is closed in $\widehat{G}$. This follows from Proposition~\ref{prop_closed_induction} and the uniform admissibility theorem \cite[Proposition 3.1]{BS05-2} (or \cite[Theorem 11.5]{Ban05}). Indeed, the latter directly implies $C^*(K\cap M)$ acts compactly in the representation $\bigoplus_{\xi\in {{\hat{M}}_{M\cap H, \ds}}  }\xi$, and this implies that $C^*(M)$ acts compactly on $\bigoplus_{\xi\in {{\hat{M}}_{M\cap H, \ds}}  }\xi$ by the density of $C^*(M)C^*(K\cap M)$ in $C^*(M)$. Note also that the set ${\hat{M}}_{M\cap H, \ds}$ is countable by the same uniform admissibility.
 
Now let  $\mathcal{P}_0$ be a finite set of cuspidal $\sigma$-parabolic subgroups comprising one per associate class of cuspidal $\sigma$-parabolic subgroups. We have the equality  $S_X=\bigcup_{P \in \mathcal{P}_0}S_X(P)$; therefore $S_X$ is a finite union of closed subsets of~$\widehat{G}$. This implies the theorem. 
\end{proof}

 \section{{The discrete series of certain symmetric spaces}} \label{sec:dis}
 
Throughout this section, $(G, H)$ is a reductive symmetric pair in which $G$~is connected. We assume 
$
X=G/H
$
has discrete series. By Theorem~\ref{thm:equal-rank}, the equal-rank condition holds for $X$, so compact Cartan subspaces $\ftq$ of $\fq$ exist. We shall see that under an additional hypothesis on the symmetric space~$X$, there is a simple parametrization of the discrete series of~$X$ in terms of linear functionals on~$\ftq$, or in terms of lowest $K$-types.

\subsection{A compact Cartan subspace with abelian centralizer}
Fix a compact Cartan subspace $\ftq\subset \fk\cap \fq$. We assume throughout this section that
\begin{equation}\label{centralizer_compact_cartan}
\text{the centralizer} \ Z_\fg(\ftq)\ \text{is abelian.}
\end{equation}

Since all compact Cartan subspaces of $\fq$ are conjugate under $(K\cap H)_0$, condition~\eqref{centralizer_compact_cartan} is independent of the choice of $\ftq$.

\begin{lemma}\label{lem:fundamental-cartan-abelian-centralizer}
Under assumption~\eqref{centralizer_compact_cartan}, the Lie algebra
\[
\fl=Z_\fg(\ftq)
\]
is a fundamental Cartan subalgebra of $\fg$, that is, a maximally compact Cartan subalgebra. The group centralizer 
\[
L = Z_G(\ftq)
\]
is a connected abelian Cartan subgroup of $G$.
\end{lemma}

\begin{proof}
Because $\ftq\subset \fk$ and $\ft=Z_\fk(\ftq)$ is abelian, $\ft$ 
is a Cartan subalgebra of $\fk$ (see \cite[\S\,2, th.~1]{Bourbaki}). Since $\ft\subset \fl$ and since by assumption $\fl$ is abelian, every element of $\fl$ centralizes~$\ft$; hence
\[
\fl\subset Z_\fg(\ft).
\]
The reverse inclusion is immediate from $\ftq\subset \ft$, so $\fl=Z_\fg(\ft)$. Now $Z_\fg(\ft)$ is the $\theta$-stable Cartan subalgebra of $\fg$ associated with the Cartan subalgebra $\ft$ of $\fk$, so $\fl$ is fundamental.

The group $L=Z_G(\ftq)$ has Lie algebra $\fl$. It is connected since $L\cap K = Z_K(\ft_q)$ is connected: the centralizer of any torus in a connected compact Lie group is connected \cite[\S\,2, n°2, cor.~5]{Bourbaki}. Thus, $L\cap K=T=\exp(\ft)$ is a maximal torus in $K$. Since \(L\) is \(\theta\)-stable, its Cartan decomposition gives
\(L=(L\cap K)\exp(\mathfrak l\cap\mathfrak p)\). Thus the connectedness of
\(L\cap K\) implies the connectedness of \(L\).  Hence
\[
L=TA
\]
is an abelian Cartan subgroup of $G$ where $T=\exp(\ft) =Z_K(\ft_\fq)$ and $A=\exp(\fa)$, $\fa=\fl\cap \fp$. 
\end{proof}

\subsection{The discrete series in terms of cohomological induction}\label{sec:FJ} 

Building on fundamental results of Flensted-Jensen \cite{FJ80} and
Oshima--Matsuki \cite{OM84, M88}, one can describe the discrete series of reductive symmetric spaces  by Oshima--Matsuki
parameters, which are roughly elements of $\ftq^\ast$ satisfying certain regularity and evenness conditions. To discuss the classification we shall follow
Vogan's cohomological-induction formulation \cite[Section~2]{V88} (see also \cite{S83}), explaining what simplifies under the abelian-centralizer hypothesis~\eqref{centralizer_compact_cartan}. We shall explain the connection with Flensted-Jensen's construction in Remark~\ref{rem:FJfunc} below.

\begin{blank}
We fix once and for all a compact Cartan subspace $\ftq$ of $\fq$ satisfying the abelian centralizer condition, and set $\ft\coloneq Z_\fk(\ftq)$, $\fl\coloneq Z_\fg(\ftq)=Z_\fg(\ft)$. Let $L=Z_G(\ftq)=TA$ where $T=\exp(\ft)$ is a maximal torus in $K$. 

We also fix once and for all a positive system $\Sigma_c^+$ for the root system $\Sigma(\fk^\bC, \ftq^\bC)$, which determines a unique compatible positive system $\Sigma_K^+$ for the root system $\Sigma(\fk^\bC, \ft^\bC)$. Here, the compatibility means that for every $\alpha \in \Sigma^+_K$, the restriction $\alpha\vert_{\ftq}$ is either zero or an element of  $\Sigma^+_c$. In our setting, the restriction is never zero.

 In general, there are finitely many positive systems $\Sigma^+$ for $\Sigma(\fg^\bC, \ftq^\bC)$ that are compatible with the positive system $\Sigma_c^+$ of $\Sigma(\fk^\bC, \ftq^\bC)$. In our setting, such a choice of $\Sigma^+$ determines a unique positive system $\Sigma_G^+$ of $\Sigma(\fg^\bC, \ft^\bC)$ compatible with $\Sigma^+$. Enumerate the positive systems obtained in this way as
\[
\Sigma_{G,1}^+,\dots,\Sigma_{G,m}^+.
\]
This list consists of the systems induced from $\ftq$; it need not exhaust the positive systems of $\Sigma(\fg^\bC, \ft^\bC)$ containing $\Sigma^+_K$.

For each $j \in \{1, \dots, m\}$, define
\[
\mathfrak{u}_j=\bigoplus_{\alpha\in \Sigma_{G,j}^+}\fg_\alpha^{\mathbb C},
\qquad
\q_j=\fl^{\mathbb C}\oplus \mathfrak{u}_j.
\]
Then $\q_j$ is a $\theta$-stable Borel subalgebra of $\fg^{\mathbb C}$.
Let
\[
\rho_{G,j}=\frac{1}{2}\sum_{\alpha\in \Sigma_{G,j}^+}\alpha,
\qquad
\rho_{K}=\frac{1}{2}\sum_{\alpha\in \Sigma_{K}^+}\alpha.
\]

\begin{lemma}\label{lem-minus-sigma} We have $\sigma(\rho_K) =-\rho_K$, and $\sigma(\rho_{G,j}) = -\rho_{G,j}$ for all $j \in \{1, \dots m\}$. 
\end{lemma}
\begin{proof}
Under assumption~\eqref{centralizer_compact_cartan}, no root of
\(\Sigma(\fg^\bC,\ft^\bC)\) vanishes on \(\ftq\). For each of the compatible
positive systems under consideration, a root \(\alpha\) is positive if and
only if its restriction \(\alpha|_{\ftq}\) is positive. Since \(\sigma\) acts
as \(-1\) on \(\ftq\), we have
\[
(\sigma\alpha)|_{\ftq}=-\alpha|_{\ftq}.
\]
Consequently,
\[
\sigma(\Sigma^+_{G,j})=-\Sigma^+_{G,j},
\qquad
\sigma(\Sigma_K^+)=-\Sigma_K^+,
\]
which implies the result. 
\end{proof}
\end{blank}
Since the \((-1)\)-eigenspace of \(\sigma\) on
\((\ft^\ast)^\bC\) naturally identifies with \((\ftq^\ast)^\bC\), Lemma~\ref{lem-minus-sigma} allows us to identify $\rho_K$ and the $\rho_{G,j}$ with elements of \((\ftq^\ast)^\bC\). We shall henceforth make these identifications without further comment.

For $\lambda\in (\ftq^*)^\bC$, we continue to write $\lambda$ for its extension by zero to an element of $(\ft^*)^\bC$. 
Let $\lambda\in (\ftq^*)^\bC$. We say  $\lambda$ is \emph{regular} if $\langle \alpha, \lambda \rangle \neq  0$ for every $\alpha \in \Sigma(\fg^\bC, \ftq^\bC)$. We say $\lambda$ is \emph{regular and $\Sigma^+_{G, j}$-dominant} if $\langle \alpha, \lambda \rangle > 0$ for every $\alpha \in \Sigma^+_{G,j}$. Under our assumption \eqref{centralizer_compact_cartan}, none of the roots in $\Sigma(\fg^\bC, \ft^\bC)$ vanish on $\ftq$. Therefore, $\lambda$ is regular if and only if $\lambda$, regarded as an element in $ (\ft^*)^\bC$, satisfies $\langle \alpha, \lambda \rangle \neq 0$ for every $\alpha \in \Sigma(\fg^\bC, \ft^\bC)$.

\begin{blank}
Let us introduce  notation related to cohomological induction. Given one of the systems $\Sigma^+_{G, j}$, the elements $2\rho_{G, j}$ and $2\rho_{K}$ are the linear functionals on~$\ft^\bC$ determined by the action of~$\fl^\bC$ on the top exterior power of $\mathfrak{u}_j$ and $\mathfrak{u}_j\cap \fk^\bC$, respectively. To any regular and $\Sigma^+_{G, j}$-dominant~$\lambda\in (\ftq^*)^\bC$  such that $\lambda-\rho_{G,j}$ is analytically integral on $T$, we attach a $(\fg, K)$-module $A^{\fq_j}(\lambda)$ by cohomological induction: in the notation of \cite[p.~330, (5.6)]{KnappVogan}, we set 
\[
A^{\fq_j}(\lambda) = \mathcal{R}^{S}_{\mathfrak{q}_j}(\mathbb{C}_{\lambda-\rho_{G,j}})
\]
where  $S=\dim(\mathfrak{u}_j\cap \fk^\bC)$ and we apply the $\rho$-shift $\lambda-\rho_{G,j}$ to $\lambda$ before applying the Zuckerman functor $\mathcal{R}^{S}_{\mathfrak{q}_j}$.

For fixed $j$ and $\lambda$, set
\begin{equation}\label{eq:mu-lambda}
\mu_\lambda=\lambda+\rho_{G,j}-2\rho_{K}.
\end{equation}

All highest weights for representations of $K$ below are taken with respect to the fixed positive system $\Sigma_K^+$.

With this notation, the description of the discrete series can be summarized as follows when the symmetric space $G/H$ satisfies our condition~\eqref{centralizer_compact_cartan}. 

\end{blank}

\begin{theorem}[{\cite[Theorem~2.9]{V88}, \cite{FJ80}, \cite{S83}, \cite{OM84}}]\label{thm_Vogan_ds} Let $X=G/H$ be a reductive symmetric space. Suppose $G$ is connected. Assume \eqref{centralizer_compact_cartan}. Let
\[
L^2(X)_d\subset L^2(X)
\]
denote the Hilbert direct sum of the irreducible $G$-subrepresentations of $L^2(X)$. Then the space of $K$-finite vectors of \(L^2(X)_d\) has the
algebraic direct sum decomposition
\[
(L^2(X)_d)_K
\cong
\bigoplus_{j=1}^m\ \bigoplus_{\lambda\in\Lambda_j}
A^{\fq_j}(\lambda),
\]
where $\Lambda_j$ is the set of all $\lambda\in i\ftq^\ast$ such that:
\begin{enumerate}[(i)]
\item $\lambda$ is regular and $\Sigma_{G,j}^+$-dominant;
\item $\mu_\lambda$, as in \eqref{eq:mu-lambda}, is the highest weight of an irreducible $(K \cap H)$-spherical representation of $K$ \textup{(}i.e.\ a representation with a nonzero $K\cap H$-fixed vector\textup{)}. 
\end{enumerate}
\end{theorem}

Condition (ii) implies that the linear functional $\lambda -\rho_{G,j}= \mu_\lambda -2\rho_{G,j} + 2\rho_K$ is analytically integral. 

The individual summands $A^{\fq_j}(\lambda)$ in the theorem satisfy:

\begin{theorem}[\cite{KnappVogan}]\label{thm:Aq-properties} In the setting of Theorem \ref{thm_Vogan_ds}, for any $\lambda \in \Lambda_j$,  
\begin{enumerate}[(a)]
\item $A^{\fq_j}(\lambda)$ is a (nonzero) irreducible $(\mathfrak{g}$-$K$)-module; 
\item $A^{\fq_j}(\lambda)$ has infinitesimal character $\lambda\in (\ft^*)^\bC \subset (\fl^*)^\bC$ in Harish-Chandra's parametrization;
\item $A^{\fq_j}(\lambda)$ is unitary;
\item $A^{\fq_j}(\lambda)$ is tempered;
\item $A^{\fq_j}(\lambda)$ has a unique lowest $K$-type whose highest weight is $\mu_\lambda$.
\end{enumerate}
\end{theorem}
\begin{proof}
(a) follows from \cite[Theorem 8.2]{KnappVogan} since $\lambda$ is dominant regular; (b) is the content of \cite[Corollary 5.25 (b)]{KnappVogan}; (c) follows from \cite[Theorem 9.1]{KnappVogan}; (d)  from \cite[Theorem 11.225]{KnappVogan}; (e)  from \cite[Proposition 10.24]{KnappVogan}.
\end{proof}

\subsection{The discrete series in terms of lowest \texorpdfstring{$K$}{K}-types}

\begin{definition}[\cite{AA24}] An irreducible unitary representation of a real reductive Lie group is called \emph{tempiric} if it is tempered and has real infinitesimal character.
\end{definition}

\begin{remark} \label{rk-tempiric-regular} For a connected reductive group such as~$G$, every tempiric representation has a unique lowest $K$-type. Suppose $\pi$ is tempiric and has  \emph{regular}  infinitesimal character. Then the space of $K$-finite vectors of $\pi$ is isomorphic to $A^{\mathfrak{q}_0}(\lambda)$ for a $\theta$-stable Borel subalgebra $\fq_0=\fl^\bC\oplus \frak u$ where $\Delta^+=\Sigma(\fq_0, \ft^\bC)$ is compatible with $\Sigma^+_K$, $\lambda\in i\ft^*$ is regular and $\Delta^+$-dominant, and $\lambda-\rho_G$ is analytically integral on $T$. Here  $\rho_G$ denotes the half-sum of the roots in $\Delta^+$, counted with multiplicities. Then the unique lowest $K$-type of $\pi$ has highest weight $\mu =\lambda + \rho_{G} - 2\rho_K$. This follows from \cite[Proposition 10.24, Theorem 11.225]{KnappVogan}.
\end{remark}

\begin{proposition}\label{prop_ds_ab}  In the setting of Theorem \ref{thm_Vogan_ds}, the discrete series $\widehat G_{H,\ds}$ of $X=G/H$ is precisely the set of tempiric representations of $G$ which have regular infinitesimal character and whose lowest $K$-type is $K\cap H$-spherical. Each of these representations occurs in $L^2(X)$ with multiplicity one.
\end{proposition}

\begin{proof}
Suppose first that $[\pi]\in \widehat G_{H,\ds}$. Then the underlying $(\fg,K)$-module of $\pi$ occurs in $L^2(X)_d$, so Theorem~\ref{thm_Vogan_ds} yields an index $j$ and a parameter $\lambda\in \Lambda_j$ such that the underlying $(\fg,K)$-module of $\pi$ is isomorphic to $A^{\q_j}(\lambda)$. By Theorem~\ref{thm:Aq-properties}, the representation $\pi$ is tempiric, has regular infinitesimal character, and its unique lowest $K$-type has highest weight $\mu_\lambda$. Condition~(ii) in Theorem~\ref{thm_Vogan_ds} says precisely that this lowest $K$-type is $K\cap H$-spherical.

Conversely, let $\pi$ be a tempiric representation of $G$ with regular infinitesimal character; assume its unique lowest $K$-type is $K\cap H$-spherical. Then $\pi$ is equivalent to $A^{\fq_0}(\lambda)$ for $\fq_0$,  $\lambda \in i\ft^*$, $\Delta^+\supset \Sigma^+_K$, $\rho_G$ as in Remark \ref{rk-tempiric-regular}. The highest weight of its lowest $K$-type is $\mu=\lambda + \rho_G - 2\rho_K$. By the Cartan–Helgason criterion (see Proposition \ref{prop:appendix-connected-CH}), the $K\cap H$-sphericity of $\mu$ implies $\mu \in i\ftq^*$. Since $\rho_K\in i\ftq^*$ by Lemma \ref{lem-minus-sigma}, $\lambda+\rho_G \in i\ftq^*$. Since $\lambda$ is regular and $\Delta^+$-dominant, so is $\lambda +\rho_G$. This implies that $\Delta^+\supset \Sigma^+_K$, which is a positive system determined by $\lambda +\rho_G \in i\ftq^*$, is the one compatible with some positive system $\Sigma^+(\fg^\bC, \ftq^\bC)$ compatible with $\Sigma^+_c$. Hence, $\Delta^+=\Sigma^+_{G, j}$ for some $j$. By Lemma \ref{lem-minus-sigma}, $\rho_G=\rho_{G, j}\in i\ftq^*$. Thus, $\lambda \in i\ftq^*$. The highest weight of its unique lowest $K$-type is $\mu_\lambda$. The sphericity hypothesis is therefore exactly condition~(ii) in Theorem~\ref{thm_Vogan_ds}; hence $[\pi]\in \widehat G_{H,\ds}$.

The lowest-\(K\)-type formula~\eqref{eq:mu-lambda} shows that distinct pairs
\((j,\lambda)\) give different highest weights for the lowest \(K\)-types of
\(A^{\fq_j}(\lambda)\); therefore, these irreducible representations cannot be equivalent.  Hence the decomposition in Theorem~\ref{thm_Vogan_ds} is multiplicity-free, and every representation in $\widehat G_{H,\ds}$ occurs with multiplicity one in $L^2(X)$.
\end{proof}

\subsection{The square-integrable $H$-fixed distributions in terms of lowest \texorpdfstring{$K$}{K}-types}

It will be convenient later to have the following refinement of the multiplicity one statement for the discrete series:

\begin{proposition}\label{prop:ds-distribution-one-dimensional}
In the setting of Theorem \ref{thm_Vogan_ds}, let $(\pi,V)\in \widehat G_{H,\ds}$, and let $\tau_\pi$ be the unique lowest $K$-type of $\pi$. Then $\dim_{\mathbb C}(V^{-\infty})^H_{\ds}=1$, and every nonzero element of $(V^{-\infty})^H_{\ds}$ is nonzero on $\tau_\pi$.
\end{proposition}
\begin{proof}
By Proposition~\ref{prop_ds_ab}, the representation $\pi$ occurs in $L^2(X)$ with multiplicity one. Hence Lemma~\ref{lem_sq_int_dis} gives $\dim_{\mathbb C}(V^{-\infty})^H_{\ds}=1$.

Suppose $T\colon V\hookrightarrow L^2(X)$ is a $G$-equivariant embedding. Under this embedding, the evaluation functional 
\[
\eta\colon V^\infty\to  L^2(X)\cap C^\infty(X) \to \mathbb C,
\qquad
v\longmapsto (Tv)(eH)
\]
defines a nonzero element of $(V^{-\infty})^H_{\ds}$. Since $(V^{-\infty})^H_{\ds}$ has dimension~1, it is enough to check that $\eta$ is nonzero on $\tau_\pi$ for every such $T$.

Let $(j,\lambda)$ be the unique pair such that the underlying $(\fg,K)$-module of $\pi$ is $A^{\q_j}(\lambda)$. By the work of Flensted-Jensen (see Remark \ref{rem:FJfunc} below), the image $T(V^\infty)$ is explicitly known in our setting. A key point is that the image $T(\tau_\pi)\subset C^\infty(X)$ of the lowest $K$-type contains an explicit vector $\psi_{j,\lambda}$ such that
\[
\psi_{j,\lambda}(eH)\neq 0:
\]
see \cite{FJ80,OM84,S84} and Remark \ref{rem:FJfunc} below. Therefore $\eta$ cannot vanish on~$\tau_\pi$,  proving the assertion.
\end{proof}

\begin{remark} \label{rem:FJfunc} Given a pair $(j, \lambda)$ as in Theorem~\ref{thm:Aq-properties}, let us explain how Flensted-Jensen constructs a subrepresentation of $L^2(X)$ equivalent to $A^{\fq_j}(\lambda)$, and how the special function $\psi_{j,\lambda}$ appears. 

Let \(C=H\cap Z(G)\). Since \(C\) is central in \(G\) and contained in \(H\),
the natural identification
\[
G/H\cong (G/C)/(H/C)
\]
also identifies the corresponding regular representations. Moreover, since
\(G\) is connected and reductive, we have \(Z(G/C)=Z(G)/C\), and therefore
\[
(H/C)\cap Z(G/C)=\{e\}.
\]
Thus, after replacing \((G,H)\) by \((G/C,H/C)\), we may assume without loss
of generality that \(H\cap Z(G)=\{e\}\). By splitting off the connected part of the center of $G$ which acts on any irreducible module by a central character, we may assume $G$ is semisimple and this does not affect the assertion $\psi_{j,\lambda}(eH)\neq0$. 

We first treat the case where $H$ is connected. The starting point is the Flensted--Jensen duality (\cite[Theorem 2.3]{FJ80}, \cite[Section 8.2]{S84}):
\begin{equation} \label{eq_FJ_dual}
C_\mu^\infty(G/H) \cong C_{\mu^d}^\infty(G^d/H^d) 
\end{equation}
between the space of $C^\infty$ functions on $X=G/H$ transforming under the $K$-type $\mu$, and the space of $C^\infty$ functions on a ``dual'' symmetric space $X^d = G^d/H^d$, transforming under a certain $K^d$-type $\mu^d$, where $H^d$ is a maximal compact subgroup of $G^d$. 
 As explained in \cite[Proof of Theorem 16.1]{BS05-1}, this holds in our setting where $G$ is not assumed to be linear. 
 Specifically, $G^d, H^d, K^d$ are the analytic subgroups in the complex adjoint group $G_\bC$ whose Lie algebras, as subalgebras of $\fg_\bC$, are
\[
\fg^d = (\fk\cap \fh \oplus \fp\cap \fq) \oplus i (\fk\cap \fq \oplus \fp\cap \fh);
\]
\[
\fh^d = (\fk\cap \fh) \oplus i (\fp\cap \fh), \quad \fk^d = (\fk\cap \fh) \oplus i (\fk\cap \fq)
\]
respectively. For any irreducible representation $(\mu, V_\mu)$ of~$K$, $\mu$ extends to an irreducible holomorphic representation of $(K, \fk_\mathbb{C})$. Its restriction to $(K\cap H, \fk^d)$ exponentiates to an irreducible representation $\mu^d$ of $K^d$ with the same highest weight as $\mu$. The duality isomorphism \eqref{eq_FJ_dual} sends
$f\in C^\infty_\mu(G/H)$ to the function $f^d\in C^\infty_{\mu^d}(G^d/H^d)$  defined by
 \[
 f^d(k^d x H^d)\coloneq (\mu^d((k^d)^{-1})f)(xH),
 \]
 for $k^d\in K^d$, $x\in \exp(\fp\cap \fq)$. In particular, we have
 \begin{equation} \label{eq_FJ_dual_e}
 f(eH) = f^d(eH^d).
 \end{equation}
See \cite[Section 8.2]{S84}.

 Now let $\lambda$ be an element of $\Lambda_j$ as in Theorem \ref{thm_Vogan_ds}, and let $\mu$ be an irreducible representation of $K$ whose highest weight is the functional $\mu_\lambda$ in~\eqref{eq:mu-lambda}. Let \(\psi_{j,\lambda}\in C^\infty_\mu(G/H)\) be the function corresponding, under the isomorphism~\eqref{eq_FJ_dual}, to the function \(\psi^d_{j,\lambda}\in C^\infty_{\mu^d}(G^d/H^d)\) given by
 \[
 \psi_{j, \lambda}^d(xH^d) = \int_{K\cap H}e^{\langle-\lambda -\rho_{G,j}, t^d(x^{-1}k)\rangle}dk
 \]
 where $t^d$ in the exponent is the projection from $G^d=H^d \cdot \exp{i \ftq} \cdot N^d$ (the Iwasawa decomposition) to $T^d= \exp{i \ftq}$ followed by the logarithm map (see  \cite[Remark 2]{OM84}, \cite[Section 8.3]{S84}). Then  $\psi_{j,\lambda}$ is square-integrable on $X=G/H$ (\cite{FJ80}, \cite{OM84}, \cite[Theorem 8.3.1]{S84}).  It is clear that $\psi^d_{j,\lambda}(eH^d)>0$ and hence $\psi_{j,\lambda}(eH)>0$ by \eqref{eq_FJ_dual_e}. Flensted--Jensen proves that the closed \(G\)-subrepresentation of \(L^2(X)\) generated by \(\psi_{j,\lambda}\) is an irreducible discrete-series representation \cite{FJ80}. Under Vogan's 
parametrization in terms of cohomological induction \cite[Theorem~2.9]{V88}, its underlying \((\fg,K)\)-module is
\(A^{\fq_j}(\lambda)\). By definition   $\psi_{j,\lambda}$ lies in the $\mu$-isotypical subspace of $A^{\fq_j}(\lambda)$, where $\mu$ is the unique lowest $K$-type of  $A^{\fq_j}(\lambda)$. This completes our discussion of the Flensted-Jensen construction when $H$ is connected.

For disconnected \(H\), pullback along \(G/H_0\to G/H\) identifies
\(L^2(G/H)\) with the \(H/H_0\)-fixed subspace of \(L^2(G/H_0)\).
The corresponding orthogonal projection is the normalized averaging operator
\[
(Pf)(gH_0)
=
\frac{1}{[H:H_0]}
\sum_{hH_0\in H/H_0} f(ghH_0).
\]
If an irreducible representation \(\pi\) occurs in \(L^2(G/H)\), then \(P\)
is nonzero on the corresponding discrete summand of \(L^2(G/H_0)\).
By the connected-case multiplicity-one statement and Schur's lemma, \(P\)
acts as the identity on that summand. Hence its Flensted--Jensen function is
\(H/H_0\)-invariant and descends to \(G/H\), while its nonzero value at the
identity coset is unchanged.
\end{remark}

\section{{Well-tempered symmetric spaces}}\label{sec:welltempered}

The abelian-centralizer condition that appeared in Section~\ref{sec:dis} will turn out to interact quite well with the recursive description of the Plancherel support in Section~\ref{sec:support}. We therefore isolate it here without imposing the equal-rank hypothesis.

\subsection{Definition and a hereditary property}

\begin{lemma}\label{lem:welltemp-all-cartan}
Let $(G, H)$ be a reductive symmetric pair. The following conditions are equivalent:
\begin{enumerate}[(1)]
\item There exists a Cartan subspace $\fbq\subset \fq$ such that the Lie algebra centralizer $Z_\fg(\fbq)$ is abelian.
\item For every Cartan subspace $\fbq\subset \fq$, 
the centralizer $Z_\fg(\fbq)$ is abelian.
\item For every Cartan subspace $\fb_{\fq^\bC} \subset \fq^\bC$ of $\fq^\bC$, the centralizer $Z_{\fg^\bC}(\fb_{\fq^\bC})$ is abelian.
\end{enumerate}
\end{lemma}
\begin{proof}
 That (3) implies (2) follows from the fact that  $Z_{\fg^{\mathbb C}}({\fbq}^{\mathbb C})=Z_\fg({\fbq})^{\mathbb C}$ for every Cartan subspace $\fbq \subset \fq$. That (2) implies (1) is immediate. 
 
Let us prove that (1) implies (3).  By the Kostant--Rallis conjugacy theorem \cite[Theorem 1]{KostantRallis1971} for complex reductive symmetric
pairs, all Cartan subspaces of \(\mathfrak q^\mathbb C\) are conjugate under
the identity component of the fixed-point group in the complexified pair.
Therefore the ${\mathfrak g}^{\mathbb \C}$-centralizers of all Cartan subspaces of $\mathfrak q^\mathbb C$ are conjugate. Given a Cartan subspace $\fbq \subset \fq$, one  of the centralizers of the previous sentence is 
\[
Z_{\mathfrak g^\mathbb C}(\mathfrak b_{\mathfrak q}^\mathbb C)
=
Z_{\mathfrak g}(\mathfrak b_{\mathfrak q})^\mathbb C;
\]
if this one is abelian, all of them are abelian.
\end{proof}

\begin{definition}\label{def_welltemp}
The reductive symmetric space $X=G/H$ is called \emph{well-tempered} if the equivalent conditions in Lemma \ref{lem:welltemp-all-cartan} hold.
\end{definition}

For our purposes, one important feature of Definition~\ref{def_welltemp} is that the condition is stable under passage to the Levi factors attached to (cuspidal) $\sigma$-parabolic subgroups. This is the content of the next observation.

\begin{lemma}\label{lem:welltemp_hereditary}
Assume that $X=G/H$ is well-tempered.
Let
\[
\fbq=\ftq\oplus \faq
\]
be a $\theta$-stable Cartan subspace of $\fq$ such that $Z_\fg(\fbq)$ is abelian, and let
\[
L=Z_G(\faq)=MA
\]
be the associated Levi subgroup.
The symmetric space
\[
M/(M\cap H)
\]
is again well-tempered.
\end{lemma}

\begin{proof}
The subspace $\ftq$ is a Cartan subspace of $\fm\cap \fq$. Indeed, if \(\mathfrak c\subseteq\mathfrak m\cap\mathfrak q\) were a Cartan
subspace properly containing \(\mathfrak t_{\mathfrak q}\), then
\(\mathfrak c\oplus\mathfrak a_{\mathfrak q}\) would be an abelian subspace of
semisimple elements in \(\mathfrak q\) properly containing
\(\mathfrak b_{\mathfrak q}\), contradicting the maximality of
\(\mathfrak b_{\mathfrak q}\).  Since $\fm\subset Z_\fg(\faq)$, we obtain 
\[
Z_\fm(\ftq)=\fm\cap Z_\fg(\ftq)\subset Z_\fg(\faq)\cap Z_\fg(\ftq)=Z_\fg(\fbq).
\]
The Lie algebra $Z_\fg(\fbq)$ is abelian by assumption, hence so is $Z_\fm(\ftq)$.
Therefore $M/(M\cap H)$ is well-tempered.
\end{proof}

\begin{remark}
Lemma~\ref{lem:welltemp_hereditary} is the basic permanence statement needed later. Combining it with Theorem~\ref{thm_description} and the concrete descriptions
of the discrete series for the smaller symmetric spaces, as in Section~\ref{sec:dis} above and Section~\ref{sec:disc-welltemp} below, we shall see that the tempered spectrum of well-tempered symmetric spaces can be described in simple terms. 
\end{remark}

\subsection{Connection with tempered homogeneous spaces}

The terminology \emph{well-tempered} comes from a connection with recent work of Benoist and Kobayashi~\cite{BK15, BK21, BK22, BK23}, who study the following class of homogeneous spaces.

\begin{definition}[Benoist and Kobayashi \cite{BK15}] A $G$-homogeneous space $X$ is called \emph{$G$-tempered}, or simply \emph{tempered}, if the regular representation $\lambda_{X}$ of~$G$ on~$L^2(X)$ is weakly contained in the regular representation $\lambda_G$ of~$G$ on~$L^2(G)$; in other words, if every irreducible representation in the support of~$\lambda_{X}$ is tempered as a representation of~$G$.
\end{definition}

\begin{proposition}\label{prop:equiv-welltemp} Let $(G, H)$ be a reductive symmetric pair. The following conditions are equivalent:
\begin{enumerate}
\item\label{item-1} $G/H$ is well-tempered.
\item\label{item-2} There is an element in $\mathfrak q$ that is regular semisimple as an
element of \(\mathfrak g\);
\item\label{item-3} The set
\[
\{x \in \fq \mid  \text{$x$ is regular semisimple as an element of $\fg$} \}
\]
is dense in $\fq$ in the Euclidean topology.
\item\label{item-4} The set
\[
S=\{x \in \fq \mid  \text{$Z_\fh(x)$ is abelian} \}
\]
is dense in $\fq$ in the Euclidean topology.
\end{enumerate}
\end{proposition} 
\begin{proof}
\eqref{item-1} \(\implies\) \eqref{item-2}:
Let \(\fbq\subset\fq\) be a Cartan subspace and put
\[
\fb=Z_\fg(\fbq).
\]
By assumption, \(\fb\) is abelian. Since \(\fbq\subset\fb\), we have
\[
\fb=Z_\fg(\fbq)\supseteq Z_\fg(\fb)\supseteq\fb,
\]
and hence \(Z_\fg(\fb)=\fb\). Moreover, \(\fb\) is reductive, since it is
the centralizer of a subspace of semisimple elements. Thus \(\fb\) is a Cartan subalgebra of \(\fg\).

No root in \(\Sigma(\fg^\bC,\fb^\bC)\) vanishes identically on \(\fbq\):
otherwise the corresponding root space would be contained in
\(Z_{\fg^\bC}(\fbq^\bC)=\fb^\bC\). We may therefore choose
\(x\in\fbq\) such that
\[
\alpha(x)\neq0
\qquad
\text{for every }\alpha\in\Sigma(\fg^\bC,\fb^\bC).
\]
Then \(x\) is a regular semisimple element of \(\fg\) contained in \(\fq\).

\eqref{item-2} \(\implies\) \eqref{item-3}:
The set of regular semisimple elements of \(\fg\) is Zariski open.
Its intersection with \(\fq\) is nonempty by assumption and is therefore
Zariski open dense, hence Euclidean open dense, in \(\fq\).

\eqref{item-3} \(\implies\) \eqref{item-4}:
If \(x\in\fq\) is regular semisimple as an element of \(\fg\), then
\(Z_\fg(x)\) is an abelian Cartan subalgebra. Since it is
\(\sigma\)-stable, it decomposes as
\[
Z_\fg(x)=Z_\fh(x)\oplus Z_\fq(x).
\]
In particular, \(Z_\fh(x)\) is abelian, so \(S\) contains every
\(\fg\)-regular semisimple element of \(\fq\). Condition~\eqref{item-3}
therefore implies that \(S\) is dense in \(\fq\).

\eqref{item-4} \(\implies\) \eqref{item-1}:
Let \(\fq^{ss,\mathrm{reg}}\) denote the set of semisimple elements of
\(\fq\) that are regular for the symmetric pair \((\fg,\fh)\), equivalently
those semisimple \(x\) for which \(\dim Z_\fq(x)\) is minimal. This set is
open and dense in \(\fq\); see \cite[Theorem~5 and the preceding
discussion]{vanDijk}. Since \(S\) is dense and
\(\fq^{ss,\mathrm{reg}}\) is nonempty and open, we may choose
\[
x\in S\cap\fq^{ss,\mathrm{reg}}.
\]
Set
\[
\fbq=Z_\fq(x).
\]
Then \(\fbq\) is a Cartan subspace of \(\fq\). Relative to \(\fbq\),
regularity means that \(\alpha(x)\neq0\) for every nonzero restricted root
\(\alpha\). The restricted-root decomposition therefore gives
\[
Z_\fg(x)=Z_\fg(\fbq)
       =Z_\fh(\fbq)\oplus\fbq.
\]
In particular, \(Z_\fh(x)=Z_\fh(\fbq)\). Since \(x\in S\),
\(Z_\fh(x)\) is abelian. The subspace \(\fbq\) is also abelian, and
\(Z_\fh(\fbq)\) centralizes it. Hence \(Z_\fg(\fbq)\) is abelian, proving
that \(G/H\) is well-tempered.
\end{proof}

For the next statement, by a \emph{complex symmetric space} we mean a symmetric space $G/H$ such that $G$ is a complex  semisimple algebraic group and $\sigma$ is a holomorphic involution of $G$.

\begin{theorem}[Benoist--Kobayashi {\cite[Corollary 8.6]{BK21}}]\label{prop:complex-temp-welltemp}
Let $G/H$ be a complex symmetric space.
Then $G/H$ is tempered if and only if it is well-tempered.
\end{theorem}
\begin{proof}
By \cite[Corollary~8.6]{BK21}, which is a consequence of \cite[Theorem~1.6 and criterion (1.1)]{BK21}, for complex semisimple algebraic \(G\) and
complex reductive~\(H\), temperedness of \(L^2(G/H)\) is equivalent to the
density of points in \(\mathfrak q\) with abelian stabilizer in \(\mathfrak h\).
For a complex symmetric space, this is condition~(4) of
Proposition~\ref{prop:equiv-welltemp}, hence it is equivalent to
well-temperedness.
\end{proof}
\begin{remark} Suppose $G/H$ is a complex symmetric space,  and let $G_{\R}$ and $H_\R$ be the fixed points in $G$ and $H$ of an antiholomorphic involution of~$G$ commuting with~$\sigma$, so that we may form the real symmetric space $G_\R/H_\R$. Then it is a consequence of Lemma~\ref{lem:welltemp-all-cartan} that $G_{\R}/H_\R$ is well-tempered if and only if $G/H$ is well-tempered. Taking into account Theorem~\ref{prop:complex-temp-welltemp}, we see that \emph{$G_{\R}/H_\R$ is well-tempered if and only if $G/H$ is tempered}. 
\end{remark}

For general real symmetric spaces,  well-temperedness implies temperedness:

\begin{theorem}[Benoist--Kobayashi {\cite[Theorem 1.1]{BK21}}]\label{thm:temp-welltemp}
Let $G/H$ be a symmetric space where $G$ is a real semisimple algebraic group and $H$ is a real reductive algebraic subgroup. If $G/H$ is well-tempered, then $G/H$ is tempered. 
\end{theorem}
\begin{proof}
Suppose $G/H$ is well-tempered. By Proposition \ref{prop:equiv-welltemp}, regular semisimple elements are dense in $\fq$, and for any such $x\in \fq$, $Z_H(x)$ is real reductive algebraic and \(Z_\fh(x)\) is abelian. Hence the identity component
\(Z_H(x)_0\) is abelian and has finite index in \(Z_H(x)\), so
\(Z_H(x)\) is virtually abelian, meaning that it contains an abelian
subgroup of finite index \cite[p.~834]{BK21}.

It follows that elements $x$ in $\fq$ such that $Z_H(x)$ is virtually abelian are dense in $\fq$. By \cite[Theorem~1.1 (2)]{BK21}, this implies $G/H$ is tempered.
\end{proof}

\begin{remark} The algebraic hypotheses in Theorem~\ref{thm:temp-welltemp} are those of
\cite{BK21}. For more general reductive symmetric spaces, the fact that well-temperedness implies temperedness also follows
from the Plancherel theorem together with the description of the discrete series of well-tempered symmetric spaces. We shall omit the details.
\end{remark}

\subsection{Examples and non-examples}\label{sec:examples_welltempered}

The next examples show that many tempered symmetric spaces are well-tempered, and should give a sense of which tempered spaces are not well-tempered.

\begin{example}[{\cite[Example 8.9]{BK21}}]\label{ex:GcGr-welltemp}
Let $G$ be a connected complex reductive Lie group, regarded as a real Lie group, and let $H$ be an open subgroup of a real form of $G$.
Then $G/H$ is well-tempered.
\end{example}

\begin{proof}
Write $\fg=\fh\oplus i\fh$ as real Lie algebras, so $\fq=i\fh$. Let $\fbh\subset \fh$ be a Cartan subalgebra of $\fh$. Then $\fbq=i\fbh\subset \fq$ is a Cartan subspace of $\fq$. Since the complexification $\fbh^{\mathbb C}=\fbh \oplus i\fbh $ is a Cartan subalgebra of $\fg$, the space $\fbh^{\mathbb C}$ is its own  centralizer in $\fg$. 
Hence $Z_\fg(\fbq)=\fbh\oplus i\fbh$ is abelian.
\end{proof}

\begin{example}\label{ex:GmodK-welltemp}
Let $G$ be a connected semisimple Lie group with finite center, and let $K$ be a maximal compact subgroup of~$G$. The Riemannian symmetric space $G/K$ is well-tempered if and only if $G$ is quasi-split.
\end{example}

Note, on the other hand, that the Riemannian symmetric space $G/K$ is always tempered since $L^2(G/K)$ is contained in $L^2(G)$.

\begin{proof}
Let $\theta$ be an involution of~$G$ with fixed-point set $K$, and let $\fp$ be the $(-1)$-eigenspace of $\theta$ on $\fg$. A Cartan subspace of $\fq=\fp$ is exactly a maximal abelian subspace $\fa\subset \fp$. Hence, $G/K$ is well-tempered if and only if $Z_\fg(\fa)$ is abelian. For a connected semisimple real Lie group, this occurs if and only if $G$ is quasi-split. 
\end{proof}

\begin{example}[{\cite[Example 5.3]{BK22}}]\label{ex:GcKc-welltemp}
Let $G_1$ be a connected real semisimple Lie group with maximal compact subgroup $K_1$. Let $G=G_{1,\bC}$, and let $H=K_{1,\bC}$.
Then the complex symmetric space $G/H = G_{1,\mathbb C}/K_{1,\mathbb C}$ is well-tempered if and only if $G_1$ is quasi-split.
\end{example}

\begin{proof}
The complex symmetric space $G_{1,\mathbb C}/K_{1,\mathbb C}$ is well-tempered if and only if the real form $G_{1}/K_1$ is well-tempered: this follows from Lemma \ref{lem:welltemp-all-cartan}. As in Example \ref{ex:GmodK-welltemp}, $G_{1}/K_1$ is well-tempered exactly when $G_1$ is quasi-split.
\end{proof}

\begin{example}\label{ex:quaternionic-exception} Let $\mathbb{H}$ denote the field of quaternions. 
For every $m\geq 2$, the symmetric space
\[
\mathrm{SL}(2m-1,\mathbb H)\big/S\bigl(\mathrm{GL}(m,\mathbb H)\times \mathrm{GL}(m-1,\mathbb H)\bigr)
\]
is tempered but not well-tempered.
\end{example}

\begin{proof}
See \cite[Section~8.5(2)]{BK21} for the temperedness. To see that the
space is not well-tempered, let \(e_{r,s}\) denote the standard quaternionic
matrix unit and fix an imaginary quaternion
\(\mathbf i\in\mathbb H\) with \(\mathbf i^2=-1\).
A Cartan subspace \(\fbq\subset\fq\) is spanned over \(\bR\) by
\[
e_{r,m+r}+e_{m+r,r},
\qquad
\mathbf i(e_{r,m+r}+e_{m+r,r}),
\qquad 1\leq r\leq m-1.
\]
Its centralizer contains the nonabelian Lie subalgebra
\[
(\operatorname{Im}\mathbb H)e_{m,m}\cong\mathfrak{sp}(1),
\]
and therefore is not abelian.
\end{proof}

The next statement is a specialization to symmetric pairs of 
\cite[Theorem~1.7]{BK21}, taking into account the temperedness criterion introduced in \cite[Theorem~4.1]{BK15}. It says that, in the absence of a compact factor for~$\fh$, any tempered symmetric space that is not well-tempered is either among the examples above or among a small number of exceptional cases.

\begin{theorem}[Benoist–Kobayashi, {\cite[Theorem 1.7]{BK21}}]
Let $(G, H)$ be a symmetric pair. Suppose $\fg$ is a real simple Lie algebra and $\fh$ is a semisimple Lie subalgebra whose adjoint group has no compact factor. Assume $\fg$ is isomorphic  neither to $\mathfrak{e}_{6(-26)}$ nor to $\mathfrak{e}_{6(-14)}$. Then $G/H$ is tempered if and only if one of the following holds:
\begin{enumerate}
\item $G/H$ is well-tempered; 
\item $(\fg, \fh) \cong (\mathfrak{sl}(2m-1,\mathbb{H}),\ \mathfrak{sl}(m, \mathbb{H}) \oplus \mathfrak{sl}(m-1,\mathbb{H}))$ for some $ m \geq 2$.
\end{enumerate}
\end{theorem}

We remark, however, that the situation becomes more complicated in the presence of a compact factor. For example, let $X=G/H=\mathrm{SO}(p+q,1)/\bigl(\mathrm{SO}(p)\times \mathrm{SO}(q,1)\bigr)$, where $p\geq q+1$. Then $X$ has no discrete series, since $\mathrm{rk}(G/H)=q+1=\mathrm{rk}(K/(K\cap H))+1$. Moreover, the Plancherel support consists of principal series representations of $G$. Consequently, $X$ is tempered, whereas it is well-tempered if and only if $p\leq q+3$.


Let us finish this section with further examples of symmetric spaces related to classical groups (see \cite[Sections 8.3–8.4]{BK21}). Together with the previous proposition, these examples illustrate that for real symmetric spaces, temperedness can be a rather strong condition, but well-temperedness is not much more restrictive than temperedness. 

\begin{examples}\label{example_welltempered_classical}~
\begin{itemize}
    \item Let $p_1, p_2, q_1, q_2 \geq1$. The symmetric space \[
\mathrm{SO}(p_1+p_2,q_1+q_2)/
\bigl(\mathrm{SO}(p_1,q_1)\times \mathrm{SO}(p_2,q_2)\bigr)
\]
is tempered if and only if 
    \[
    \left|p_1+q_1-p_2-q_2\right| \leq 2;
    \]
    In that case it is well-tempered.
    \item The symmetric space
    $\mathrm{SL}(m+n, \R) / \mathrm{S}(\mathrm{GL}(m, \R) \times \mathrm{GL}(n, \R))$ is tempered if and only if $|m-n| \leq 1$; in that case it is well-tempered.
    \item The symmetric space $\mathrm{SO}(m+n, \mathbb{C})/ \mathrm{SO}(m, \mathbb{C}) \times \mathrm{SO}(n, \mathbb{C})$ is tempered if and only if $|m-n| \leq 2$; in that case it is well-tempered.
    \item The space $\mathrm{SL}(p+q, \mathbb{R}) / \mathrm{SO}(p, q)$ is well-tempered for any $p,q$.
\end{itemize}
\end{examples}

\begin{blank} For more examples of tempered homogeneous spaces, we refer the reader to \cite{BK15,BK21, BK22}. For classical symmetric spaces $X$ satisfying the equal-rank condition, Mœglin and Renard \cite[Table 1]{Moeglin_Renard} have explicit computations of the centralizer $\fl_X$ of a compact Cartan subspace. In that table, we can easily read off which spaces $X$ are well-tempered: they are well-tempered precisely when $\fl_X$ is abelian. For example, $\mathrm{Sp}(2n, \bR)/\mathrm{GL}(n, \bR)$ is an equal-rank well-tempered symmetric space for any $n\geq1$.
\end{blank}

\section{{The Plancherel support of well-tempered symmetric spaces}}\label{sec:discrete-welltemp}

In this section we describe the support of the Plancherel measure for a
well-tempered symmetric space in terms of characters of Cartan subgroups, more specifically in terms of Vogan's notion of \emph{character data} from \cite{Voganbook}. Throughout the section, unless explicitly stated otherwise, we assume \(G\) is in
\emph{Vogan's class}: by this we mean that \(G\) is a real reductive group
in Harish--Chandra's class, is linear, and has abelian Cartan subgroups,
as in \cite[(0.1.2)]{Voganbook}. As before, \(\sigma\) is an involution of \(G\), \(H\) is an open subgroup of~\(G^\sigma\), and
$
X=G/H
$.
We assume that \(X\) is well-tempered in the sense of
Definition~\ref{def_welltemp}. Whenever the discrete series of \(X\) is
discussed, we assume in addition that \(X\) has discrete series.

The section is organized as follows. First, we recall the part of
Vogan's parametrization that we use below. Second, in the equal-rank
case we reformulate the discrete-series parametrization of Section~\ref{sec:dis}
in terms of character data and the torus form of the Cartan--Helgason
criterion from the Appendix. Finally, for a general well-tempered symmetric
space, we combine this discrete-series description with the support theorem
of Section~\ref{sec:support}, and explain the resulting description of the Plancherel support and what it reveals about the geometry of the latter.

\subsection{Tempered character data for reductive groups}\label{subsec:vogan-character-data}

Let \(TA\) be a \(\theta\)-stable Cartan subgroup of \(G\), with
\(T\subset K\) and \(A\subset \exp(\fp)\). We write
\[
\ft=\operatorname{Lie}(T),
\qquad
\fa=\operatorname{Lie}(A),
\qquad
\fb=\ft\oplus\fa.
\]
Given an element \(\lambda\in i\ft^\ast\), we continue to write \(\lambda\) for its
extension by zero to an element of \(i\fb^\ast\). Here \(T\) need not be connected. Write
\[
Z_G(A)=MA
\]
for the corresponding cuspidal Levi subgroup, with Lie algebra \(\fm\oplus\fa\).
Every element \(\lambda\in i\ft^\ast\) which is \(M\)-regular, i.e. which satisfies \
\[
\langle\alpha,\lambda\rangle\neq0
\qquad
\text{for all }\alpha\in\Sigma(\fm^\bC,\ft^\bC),
\]
 determines a positive system
\(\Sigma^+(\fm^\bC,\ft^\bC)\). Set
\[
\rho_{\fm}
=
\frac12\sum_{\alpha\in\Sigma^+(\fm^\bC,\ft^\bC)}\alpha,
\qquad
\rho_{\fm\cap\fk}
=
\frac12\sum_{\alpha\in
\Sigma^+(\fm^\bC,\ft^\bC)\cap
\Sigma((\fm\cap\fk)^\bC,\ft^\bC)}\alpha,
\]
omitting the dependence on~$\lambda$.

\begin{definition} \label{def-character-datum-G} A  \emph{tempered character datum for \(G\)} is a quadruple
\[
(TA,\Gamma,\lambda,\nu),
\]
where \(TA\) is a \(\theta\)-stable Cartan subgroup as above,
\(\Gamma\) is a character of \(T\), \(\lambda\in i\ft^\ast\) is
\(M\)-regular, \(\nu\in i\fa^\ast\), and
\[
d\Gamma=\lambda+\rho_{\fm}-2\rho_{\fm\cap\fk}
\qquad\text{in }i\ft^\ast.
\]
\end{definition} 
This is Vogan's notion of character datum in the tempered case
\cite[Definition~6.6.1]{Voganbook}; see also the proof of 
\cite[Proposition~6.6.2]{Voganbook}. If we are only given a triple $(TA, \Gamma, \lambda)$ where $TA, \Gamma$ and $\lambda$ satisfy the conditions in the definition, omitting~\(\nu\), we shall call
\((TA,\Gamma,\lambda)\) a \emph{discrete character datum for \(G\)}.

To any tempered character datum $(TA,\Gamma,\lambda,\nu)$ for~$G$, Vogan associates a standard representation
\[
X_G(TA,\Gamma,\lambda,\nu),
\]
which we shall take to be the one specified in  \cite[Convention~6.6.3]{Voganbook}. By \cite[Theorem 6.6.15]{Voganbook}, if \(MAN\) is a parabolic
subgroup with Levi factor \(MA\), then
\[
X_G(TA,\Gamma,\lambda,\nu)
\simeq
\Ind_{MAN}^G(\xi_M \otimes e^\nu\otimes 1),
\]
where \(\xi_M\) is the discrete-series representation of \(M\) obtained by
cohomological induction from the \(T\)-character
\(\Gamma\otimes e^{-2\rho_{\fm\cap\fp}}\), with
\[
2\rho_{\fm\cap\fp}=2\rho_{\fm}-2\rho_{\fm\cap\fk}.
\]

\begin{lemma}\label{lem:rho-rho-eq}
Let \((TA,\Gamma,\lambda,\nu)\) be a tempered character datum for \(G\). Suppose that \(\lambda\) is \(G\)-regular, i.e.\
\[
\langle\alpha,\lambda\rangle\neq0
\qquad
\text{for every }\alpha\in\Sigma(\fg^\bC,\fb^\bC).
\]
Let \(\Sigma^+(\fg^\bC,\fb^\bC)\) be the unique positive system for which
\(\lambda\) is dominant, and let \(\Sigma^+(\fk^\bC,\ft^\bC)\) be the induced compact positive system. Let \(\rho_G\) and \(\rho_K\) be the corresponding
half-sums of positive roots. Then
\[
\rho_{\fm}-2\rho_{\fm\cap\fk}
=
\rho_G-2\rho_K
\]
in \(i\ft^\ast\).
\end{lemma}

\begin{proof}
Since \(\lambda\) is \(G\)-regular, no root of
\((\fg^\bC,\fb^\bC)\) has zero pairing with \(\lambda\). We may
write a root of \((\fg^\bC,\fb^\bC)\) as \((\alpha,\beta)\in
(\ft^\ast\oplus\fa^\ast)^\bC\). Then
\[
\fg^\bC
=
\fm^\bC\oplus
\bigoplus_{(\alpha,\beta),\,\beta\neq0}\fg^\bC_{\alpha,\beta},
\qquad
\fm^\bC
=
\fb^\bC\oplus
\bigoplus_{(\alpha,0),\,\alpha\neq0}\fg^\bC_{\alpha,0}.
\]
The positive roots are determined by the sign of
\(\langle\alpha,\lambda\rangle\). We show that
\begin{equation}\label{eq:rho-difference}
\rho_G-\rho_{\fm}
=
2(\rho_K-\rho_{\fm\cap\fk}).
\end{equation}
The Cartan involution~\(\theta\) acts by the identity on \(\ft\) and by minus the
identity on \(\fa\), and hence sends a root
\((\alpha,\beta)\) to \((\alpha,-\beta)\). Since positivity is determined
by \(\langle\alpha,\lambda\rangle\), the two roots in this pair are
simultaneously positive or simultaneously negative.

For each positive pair with \(\beta\neq0\), its contribution to
\(\rho_G-\rho_{\fm}\), regarded as a functional on \(\ft^\bC\), is
\[
\frac12\bigl((\alpha,\beta)+(\alpha,-\beta)\bigr)
 \big|_{\ft^\bC}
=
\alpha.
\]
Moreover, \(\theta\) interchanges the two root spaces
\(\fg^\bC_{\alpha,\beta}\) and
\(\fg^\bC_{\alpha,-\beta}\). Their direct sum therefore decomposes into
one-dimensional \(+1\)- and \(-1\)-eigenspaces of~\(\theta\), both of
\(\ft^\bC\)-weight \(\alpha\), lying respectively in \(\fk^\bC\) and
\(\fp^\bC\). Thus the same pair contributes
\(\frac12\alpha\) to \(\rho_K-\rho_{\fm\cap\fk}\). Summing over all
positive pairs proves \eqref{eq:rho-difference}, and rearranging this
identity proves the lemma.
\end{proof}

\begin{definition} \label{def-rep-associated-to-chardatum} Given a tempered character datum $(TA, \Gamma, \lambda, \nu)$ for~\(G\),
we denote by
\[
\Phi_G(TA,\Gamma,\lambda,\nu)\subset\widehat G_{\temp}
\]
the (finite) subset of $\widehat{G}$ consisting of the (equivalence classes of the) irreducible constituents of
\(X_G(TA,\Gamma,\lambda,\nu)\). It is a subset of $\widehat{G}_{\temp}$, the tempered dual of~$G$. 

\end{definition}

\begin{remark} There is an obvious action of $K$ on the set of all tempered character data for~$G$, and on the set of discrete character data. We denote by \[
[TA,\Gamma,\lambda,\nu].
\]
the $K$-conjugacy class of a character datum \((TA, \Gamma, \lambda, \nu)\), and by \([TA,\Gamma,\lambda]\) the $K$-conjugacy class of a discrete character datum. In Definition~\ref{def-rep-associated-to-chardatum}, the set \(X_G(TA,\Gamma,\lambda,\nu)\) depends only on the $K$-conjugacy class $[TA,\Gamma,\lambda,\nu]$.
\end{remark}

\begin{notation}\label{notation-w} Given a character datum $(TA, \Gamma, \lambda, \nu)$,
set \begin{align*}
&W(\ft\oplus\fa) =N_K(\ft\oplus\fa)/Z_K(\ft\oplus\fa),
\\
& W(\lambda)=\{w\in W(\ft\oplus\fa):w\lambda=\lambda\}, \text{ and}\\
& W(\Gamma,\lambda)=\{w\in W(\lambda):w\Gamma=\Gamma\}.
\end{align*}
Let \(\mathcal C_G\) denote the set of \(K\)-conjugacy classes
\([TA,\Gamma,\lambda]\) of discrete character data for \(G\).
\end{notation}

The next statement summarizes what we will need on the classification of irreducible tempered representations of~\(G\) in terms of character data. \\

\begin{theorem}~
\label{thm:HC_disjoint}
\begin{enumerate}[(a)]
\item Given a discrete character datum $(TA, \Gamma, \lambda)$, the set  $\Phi_G(TA,\Gamma,\lambda,\nu)$ depends only on the equivalence class of $\nu$ in $i\fa^\ast/W(\Gamma,\lambda)$.
\item
The tempered dual of~\(G\) partitions as 
\[
\widehat G_{\temp}
=
\bigsqcup_{[TA,\Gamma,\lambda,\nu]}
\Phi_G(TA,\Gamma,\lambda,\nu)
=
\bigsqcup_{[TA,\Gamma,\lambda]\in\mathcal C_G}
\bigsqcup_{\nu\in i\fa^\ast/W(\Gamma,\lambda)}
\Phi_G(TA,\Gamma,\lambda,\nu).
\]
\item Given a character datum $(TA, \Gamma, \lambda, \nu)$, all elements of $\Phi_G(TA,\Gamma,\lambda,\nu)$ have the same image in the universal Hausdorff quotient $\widehat G_{\temp,\mathrm{Haus}}$. Sending a character datum $(TA, \Gamma, \lambda, \nu)$ to the image of  $\Phi_G(TA,\Gamma,\lambda,\nu)$  in $\widehat{G}_{\temp,\mathrm{Haus}}$ yields an identification 
\[
\widehat G_{\temp,\mathrm{Haus}}
\cong
\bigsqcup_{[TA,\Gamma,\lambda]\in\mathcal C_G}
 i\fa^\ast/W(\Gamma,\lambda).
\]
\end{enumerate}
\end{theorem}

\begin{proof}
The first two assertions are Vogan's parametrization of tempered irreducible representations by
character data; see \cite[Section~6.6]{Voganbook}. The description of the
Hausdorff quotient follows from the same parametrization together with standard remarks on the connection between parabolic induction and the Fell topology. One way to obtain it is through
the \(C^\ast\)-algebraic realization of parabolic induction and the separation of the components of \(\widehat G_{\temp,\mathrm{Haus}}\) by lowest \(K\)-types, as in \cite{CHS24, CHST24}.
\end{proof}

For convenience, we recall the irreducibility of the standard representation attached to a $G$-regular character datum (for possibly disconnected $G$), and information about its unique lowest $K$-type.

We shall need the following terminology (for details see the Appendix). Given a possibly disconnected compact Lie group $K$, a Cartan subalgebra $\ft\subset \fk$, and a positive system $\Sigma_K^+~=\Sigma^+(\fk^\bC,\ft^\bC)$,  define
\[
T^l\coloneq N_K(\ft, \Sigma^+_K).
\]
This is the group of elements of $K$ that normalize $\ft$ and whose actions on the roots preserve the set $\Sigma^+_K$. It is called the \emph{large Cartan subgroup} of $K$ associated with this choice $(\ft,\Sigma_K^+)$. Let $\fn\subset \fk^\bC$ be the span of the positive root vectors.

Vogan's highest-weight classification for possibly disconnected compact groups \cite[Theorem~1.17]{VoganUnitaryBook} goes as follows: for any irreducible representation $(\tau, V)$ of $K$, let $V^{\fn}$ be the subspace of highest weight vectors, i.e. the annihilator of $\fn$. Then $V^\fn$ is naturally equipped with a representation $\mu=(\tau\mid_{T^l})\mid_{V^\fn}$ of the large Cartan subgroup $T^l$. Moreover, $(\mu, V^\fn)$ is irreducible. Let us call it the
\emph{highest-weight representation} attached to~$\pi$ (this depends on the choice of $(\ft, \Sigma_K^+$)).  When $K$ is connected, one has $T^l=T_0$ and $V^\fn$ is the usual highest weight line.

The map
\[
(\tau, V) \in \widehat{K} \mapsto (\mu, V^\fn)  \in \widehat{T^l}
\]
is an injection whose image is the set of the \emph{dominant} irreducible representation of $T^l$. Here, a representation of $T^l$ is dominant if its restriction to $T^l_0=\exp({\ft})$ is the sum of dominant characters, relative to $\Sigma^+_K$.

\begin{lemma}\label{lem:G-regular-standard-irreducible}
Let \((TA,\Gamma,\lambda,0)\) be a character datum for \(G\), and suppose
that \(\lambda\) is \(G\)-regular. Then
\[
X_G(TA,\Gamma,\lambda,0)
\]
is irreducible. It has a unique lowest \(K\)-type \((\tau, V)\). Relative to the positive
system $\Sigma^+_K=\Sigma^+(\fk^\bC, \ft^\bC)$ for which \(\lambda\) is dominant, the highest-weight representation $(\mu, V^\fn)$ of $T^l=N_K(\ft, \Sigma^+_K)$ attached to $\tau$ contains the \(T\)-character \(\Gamma\). Moreover, the highest-weight representation  $(\mu, V^\fn)$ is equivalent to the induced representation $\mathrm{Ind}_T^{T^l}\Gamma$.
\end{lemma}
\begin{proof}
By definition, the representation $X_G(TA,\Gamma,\lambda,0)$ is a finite direct sum of tempiric representations. Now \cite[Thm.~6.5.9]{Voganbook} describes a specific subset $A(\fq,\delta)\subset \widehat{K}$ of the lowest $K$-types of $X_G(TA,\Gamma,\lambda,0)$. Each element of $A(\fq,\delta)$ occurs with multiplicity one in $X_G(TA,\Gamma,\lambda,0)$. Furthermore, every irreducible submodule of $X_G(TA,\Gamma,\lambda,0)$ contains some $\tau\in A(\fq,\delta)$. Here, $\delta=\Gamma \otimes \Lambda^{\dim(\mathfrak u \cap \fp^\bC) }(\mathfrak u \cap \fp^\bC)^\ast\in \widehat{T}$ where $\mathfrak u$ is the nil-radical of the relevant $\theta$-stable parabolic subalgebra of $\fg^\bC$ as in \cite[Convention~6.6.3]{Voganbook}. Vogan also describes a bijection between  $A(\fq,\delta)$ and a finite set  $A_L(\delta)\subset \widehat{L\cap K}=\widehat{T}$, defined in \cite[Definition 4.3.15]{Voganbook}: see \cite[Definition~6.5.5]{Voganbook} for the correspondence. In our case, where $L=TA$ is a Cartan subgroup, running through the definitions reveals that $A_L(\delta)=\{\delta\}$. Together, these facts imply that $X_G(TA,\Gamma,\lambda,0)$ is irreducible, hence tempiric, and has a unique lowest $K$-type $\tau$. The correspondence \cite[Definition~6.5.5]{Voganbook} says that $\tau$ has highest weight $d\Gamma=d\delta + 2\rho_{\mathfrak u \cap \fp^\bC}$ and $\tau\mid_T$ contains $\Gamma = \delta \otimes \Lambda^{\dim(\mathfrak u \cap \fp^\bC) }(\mathfrak u \cap \fp^\bC)$. These facts imply that the highest-weight representation of $T^l$ attached to $\tau$ contains the $T$-character~$\Gamma$.

We show that the irreducible representation $(\mu, V^\fn)$ of $T^l$ is equivalent to the induced representation $\mathrm{Ind}_T^{T^l}\Gamma$. We show this by proving $\mathrm{Ind}_T^{T^l}\Gamma$ is irreducible. If $n \in T^l$ stabilizes $\Gamma \in \widehat{T}$, it must fix $d\Gamma = \lambda +\rho_G-2\rho_K$. Because $G$ is in Harish-Chandra's class and $\lambda$ is $G$-regular, the action of $n$ on $\ft_{\mathbb C}\oplus \fa_\bC$ is represented by a Weyl group element $w \in W_G$ that preserves $\Sigma_K^+$. Thus, we have $w\cdot \rho_K = \rho_K$ and $w\cdot d\Gamma = d\Gamma$. Hence, $w\cdot (\lambda +\rho_G) =\lambda + \rho_G$. Since $\lambda$ is $G$-regular and dominant, $\lambda + \rho_G$ is $G$-regular. The freeness of the Weyl action on regular elements yields $w=1$, forcing $n \in Z_K(\ft) = T$. Therefore, the stabilizer of \(\Gamma\) in \(T^l\) is exactly \(T\). It follows $\mathrm{Ind}_T^{T^l}\Gamma$ is irreducible. Consequently, $(\mu, V^\fn)$ is equivalent to $\mathrm{Ind}_T^{T^l}\Gamma$.
\end{proof}

\subsection{The discrete spectra of well-tempered symmetric spaces}\label{sec:disc-welltemp}

Assume in this subsection that \(X=G/H\) is well-tempered and has discrete
series. Fix a compact Cartan subspace
\(\ftq\subset\fk\cap\fq\). Since \(X\) is well-tempered, the centralizer
\[
\fl=Z_{\fg}(\ftq)
\]
is abelian. By the Lie-algebra assertion of
Lemma~\ref{lem:fundamental-cartan-abelian-centralizer}, the space \(\fl\) is a
fundamental Cartan subalgebra of \(\fg\). We write
\[
\fl=\ft\oplus\fa,
\qquad
\ft=\fl\cap\fk=\fth\oplus\ftq,
\qquad
\fa=\fl\cap\fp=\fah.
\]
Thus the split part of this Cartan subgroup lies in \(\fh\). Since \(G\) is in
Vogan's class, the group
\[
TA=Z_G(\fl)
\]
is an abelian Cartan subgroup of \(G\).

\begin{theorem*}\label{thm:ds_para}
Let \(X=G/H\) be a well-tempered symmetric space. Suppose that \(G\) is a
real reductive linear group in Vogan's class and that \(X\) has discrete
series. For an irreducible unitary representation \((\pi,V)\) of \(G\), the
following conditions are equivalent:
\begin{enumerate}[(1)]
\item \(\pi\) belongs to the discrete series of~$X$; 
\item \(\pi\) is tempiric with regular infinitesimal character, and its unique 
lowest \(K\)-type~\(\tau_\pi\) satisfies
\[
\tau_\pi^{K\cap H}\neq0;
\]
\item \(\pi\) is equivalent to the irreducible representation
\[
\X_G(TA,\Gamma,\lambda,0)
\]
for a discrete character datum based on the fundamental Cartan subgroup
\(TA=Z_G(\fl)\), where \(\lambda\in i\ftq^\ast\) is \(G\)-regular and, after
replacing the datum by an equivalent \(T^l\)-conjugate if necessary, the character~$\Gamma$ satisfies
\[
\Gamma|_{T\cap H}=1.
\]
\end{enumerate}
Moreover, for any \((\pi,V)\) satisfying these equivalent conditions, we have
\[
\dim_\bC (V^{-\infty})^H_{\ds}
=
\dim_\bC \tau_\pi^{K\cap H}.
\]
\end{theorem*}

\begin{proof}
We first prove the equivalence of (1) and (2), together with the dimension
formula. If \(G\) and \(H\) are connected, this is exactly
Proposition~\ref{prop_ds_ab} together with Proposition~\ref{prop:ds-distribution-one-dimensional} and Proposition~\ref{prop:appendix-connected-CH}. For the general case considered here, we inspect what happens upon restriction to the identity components. The restriction of \(V\) to \(G_0\) is a finite direct sum of irreducible
representations in the discrete series of $G_0/H_0$, so the assertion that $(\pi,V)$ is tempiric and has regular infinitesimal character follows directly from Proposition~\ref{prop_ds_ab}. By Lemma ~\ref{lem:G-regular-standard-irreducible}, $\pi$ has unique lowest $K$-type $\tau_\pi$. Applying
Proposition~\ref{prop:ds-distribution-one-dimensional} to each summand for the connected
symmetric space \(G_0/H_0\) gives an $K\cap H$-equivariant isomorphism
\[
(V^{-\infty})^{H_0}_{\ds}
\cong
(\tau_\pi^\ast)^{K\cap H_0}.
\]
by restriction of linear functionals to the
lowest \(K_0\)-types, where the star denotes contragredient and the
right-hand side is understood after restricting~\(\tau_\pi\) to $K_0\supset (K\cap H)_0 = K\cap H_0$.
Since \(H/H_0\) is finite and \(H=H_0(K\cap H)\), taking invariants gives
\[
(V^{-\infty})^H_{\ds}
=\bigl((V^{-\infty})^{H_0}_{\ds}\bigr)^{K\cap H}
\cong
\bigl((\tau_\pi^\ast)^{K\cap H_0}\bigr)^{K\cap H}
=
(\tau_\pi^\ast)^{K\cap H}.
\]
Thus \((V^{-\infty})^H_{\ds}\neq0\) if and only if
\(\tau_\pi^{K\cap H}\neq0\), and the displayed dimension formula follows.

It remains to compare (2) and (3). We use a form of the Cartan–Helgason criterion, spelled out in the Appendix. Let \(\pi\) satisfy (2). Since \(\pi\) is tempiric with regular infinitesimal character, it is equivalent with a standard representation 
\[ X_G(TA,{}^\backprime\Gamma,{}^\backprime \lambda,0) \] 
based on the fundamental Cartan subgroup $TA=Z_G(\fl)$ where $\fl=Z_{\fg}(\ftq)$ for a compact Cartan subspace $\ftq\subset \fk\cap \fq$, and \({}^\backprime\lambda\in i\ft^\ast\) is \(G\)-regular. Here, ${}^\backprime\lambda$ may not belong to $i\ftq^*$. Let ${}^\backprime\Sigma_K^+={}^\backprime\Sigma^+(\fk^\bC,\ft^\bC)$, ${}^\backprime\Sigma_G^+={}^\backprime\Sigma^+(\fg^\bC,\ft^\bC)$ be the unique positive systems for which \({}^\backprime\lambda\) is dominant. Let ${}^\backprime\rho_K, {}^\backprime\rho_G$ be the corresponding half-sums of positive roots. Let \({}^\backprime\fn\) denote the sum of the positive root spaces in $\fk^\bC$ determined by \({}^\backprime\Sigma_K^+\). Let ${}^\backprime T^l$ be the large Cartan subgroup of $K$ associated to ${}^\backprime\Sigma^+_K$.

By Lemma~\ref{lem:G-regular-standard-irreducible}, the standard representation $ X_G(TA,{}^\backprime\Gamma,{}^\backprime \lambda,0)$ is irreducible and has a unique lowest \(K\)-type \(\tau_\pi\). Moreover, the highest-weight representation of $
{}^\backprime T^l$ attached to $\tau_\pi$ is equivalent to the induced representation $\mathrm{Ind}_T^{{}^\backprime T^l}{}^\backprime \Gamma$.

Now, we choose a possibly different positive system  $\Sigma_{K}^{+}=\Sigma^+(\fk^\bC,\ft^\bC)$ that is (uniquely) compatible with some positive system $\Sigma_{c}^{+}=\Sigma^+(\fk^\bC,\ftq^\bC)$. Consider $\fn$ and $T^{l}$ accordingly for this system. There is $w_0\in N_K(\ft)$ such that $w_0\cdot {}^\backprime \Sigma^+_K = \Sigma^+_K$. By the previous part, the highest-weight representation $(\mu, W^\fn)$ attached to $(\tau_\pi, W)$ with respect to this positive system is also an irreducible induced $T^{l}$-representation. Thus we have
\[
(\mu, W^\fn)\cong \mathrm{Ind}_T^{T^l}\bC_{\Gamma},
\]
where $\Gamma = w_0\cdot {}^\backprime \Gamma$. Note that $w_0\in N_K(\ft\oplus \fa)$ automatically, and this implies $w_0\in N_K(T)$. Hence, $w_0$ conjugate $(TA, {}^\backprime \Gamma, {}^\backprime \lambda, 0)$ to   $(TA,  \Gamma, \lambda, 0)$ where $\lambda = w_0\cdot {}^\backprime \lambda$. 
Since $\Sigma_{K}^{+}$ is compatible with $\Sigma_{c}^{+}$,  Corollary~\ref{cor:appendix-parametrization} applies; therefore,
\[
\tau_\pi^{K\cap H}\neq0
\quad\Longleftrightarrow\quad
\left(\Ind_T^{T^l}\bC_\Gamma\right)^{T^l\cap H}\neq0.
\]

Since \(T\) is a normal subgroup of \(T^l\), we have 
\[
( \operatorname{Ind}^{T^l}_T \mathbb C_\Gamma)\mid_{T^l\cap H} \, 
\simeq
\bigoplus_{s\in (T^l\cap H)\backslash T^l/T}
\operatorname{Ind}^{T^l\cap H}_{(T^l\cap H)\cap T}\mathbb C_{s\Gamma}.
\]
Therefore there are nonzero \((T^l\cap H)\)-fixed vectors  if and only if
\((s\Gamma)|_{(T^l\cap H)\cap T}\) is trivial for some \(s\in T^l\), i.e.\ if and only if
\((s\Gamma)|_{T\cap H}=1\).

Replacing the character datum $(TA,  \Gamma, \lambda, 0)$ further by the equivalent \(s\)-conjugate datum, we
may therefore assume \(\Gamma|_{T\cap H}=1\). This implies that \(d\Gamma\)
vanishes on \(\fth\) and $d\Gamma \in i\ftq^*$.

Moreover, the corresponding root shifts \(\rho_G\) and \(\rho_K\) lie in \((\ftq^\ast)^\bC\): to see this, by Definition~\ref{def-character-datum-G} and Lemma~\ref{lem:rho-rho-eq}, we have
\[
d\Gamma + 2 \rho_K = \lambda + \rho_G 
\]
in $i\ft^*$. Compatibility of $\Sigma^+_K$ with $\Sigma^+_c$ implies $\rho_K\in i\ftq^*$ as in Lemma~\ref{lem-minus-sigma}. Therefore, $\lambda + \rho_G\in i\ftq^*$. The positive systems determined by $\lambda$ and $\lambda +\rho_G$ coincide. Therefore, $\lambda +\rho_G \in i\ftq^*$ implies that $\rho_G\in i\ftq^*$ as in Lemma~\ref{lem-minus-sigma}. 

Hence
\[
\lambda=d\Gamma-\rho_G+2\rho_K
\]
lies in \(i\ftq^\ast\). This proves (2) \(\Rightarrow\) (3). 

Conversely, if $\pi$ satisfies (3), then $\pi$ is tempiric with regular infinitesimal character, and its unique lowest $K$-type $\tau_\pi$ is $K\cap H$-spherical by the argument above.
\end{proof}

\subsection{Character data for symmetric spaces}\label{subsec-chardata-X}

We now return to a general well-tempered symmetric space
$
X=G/H
$,
not necessarily of equal rank. The following definition is designed so that the description of the Plancherel support in 
Theorem~\ref{thm_description} can be rewritten in terms of tempered character data. We use
a prime on \(M'\) to distinguish the Levi subgroup attached to the
symmetric-space parameter \(\faq\) from the Levi subgroup \(MA=Z_G(\fa)\) attached to the full Cartan subgroup \(TA\).

\begin{definition}\label{def:dataC}
A tempered \emph{character datum for \(X\)} is a quadruple
\((TA,\Gamma,\lambda,\nu)\) with the following properties:
\begin{enumerate}[(1)]
\item \(TA\) is a \(\sigma\)-stable and \(\theta\)-stable Cartan subgroup of
\(G\), with Lie algebra
\[
\fb=\ft\oplus\fa,
\qquad
\ft=\fth\oplus\ftq,
\qquad
\fa=\fah\oplus\faq,
\]
and
\[
\fbq=\ftq\oplus\faq
\]
is a Cartan subspace of \(\fq\);
\item \(\Gamma\) is a character of \(T\) such that
\[
\Gamma|_{T\cap H}=1;
\]
\item \(\lambda\in i\ftq^\ast\), regarded as an element of \(i\ft^\ast\) after extending it by zero on \(\fth\), is \(M'\)-regular, where
\[
M'A'=Z_G(\faq)
\]
is the Langlands decomposition of the centralizer of \(\faq\);
\item \(\nu\in i\faq^\ast\), regarded as an element of \(i\fa^\ast\) after extending it  by zero on \(\fah\);
\item the differential of \(\Gamma\) satisfies
\[
d\Gamma=
\lambda+\rho_{\fm'}-2\rho_{\fm'\cap\fk}
\qquad\text{in }i\ft^\ast.
\]
Here the positive systems defining \(\rho_{\fm'}\) and
\(\rho_{\fm'\cap\fk}\) are the ones for which \(\lambda\) is dominant and
regular on \(\fm'\).
\end{enumerate}
Omitting \(\nu\), we call \((TA,\Gamma,\lambda)\) a \emph{discrete character
datum for \(X\)}. We call $i\faq^\ast$ the \emph{relative continuous parameter space} for the  datum $(TA, \Gamma, \lambda)$, and $i\fa^\ast$ the \emph{full continuous parameter space}. 
\end{definition}

We note that, in the notation above, $\fa'=\mathrm{Lie}(A')$ contains $\faq$ and a subspace $\fah'$ of $\fah$ which may or may not be zero.\\

The next lemma connects character data for~$X$ with the notion of tempered character datum for~$G$ in Definition~\ref{def-character-datum-G}. 

\begin{lemma}\label{lem:rho_rho_eq}
Let \((TA,\Gamma,\lambda,\nu)\) be a character datum for \(X\). Let
\[
MA=Z_G(\fa)
\]
be the Langlands decomposition of the centralizer of the full split part
\(\fa\). Then \(\lambda\) is \(M\)-regular and
\[
\rho_{\fm'}-2\rho_{\fm'\cap\fk}
=
\rho_{\fm}-2\rho_{\fm\cap\fk}
\]
in \(i\ft^\ast\). Consequently \((TA,\Gamma,\lambda,\nu)\), with \(\nu\)
extended by zero on \(\fah\), is a character datum for \(G\).
\end{lemma}

\begin{proof}
Since \(\fm\subset\fm'\), \(M'\)-regularity of \(\lambda\) implies
\(M\)-regularity. Inside the reductive group \(M'\), the relevant Cartan subgroup is
\(T\exp(\fah \cap \fm')\). Its split part is \(\exp(\fah \cap \fm')\), and \(M\) is the Levi
subgroup obtained by centralizing this split part. Applying Lemma~\ref{lem:rho-rho-eq} inside \(M'\) gives the displayed
identity. Consequently, the defining differential relation for a character
datum of \(X\) in Definition~\ref{def:dataC}\textup{(5)} becomes the
defining differential relation for a character datum of \(G\) in Definition~\ref{def-character-datum-G}.
\end{proof}

\begin{theorem}\label{thm:tempered spec}
Let \(X=G/H\) be a well-tempered symmetric space, and suppose that \(G\) is
a real reductive linear group in Vogan's class. Let \(\pi\in\widehat G\) be an
irreducible unitary representation of \(G\). Then \(\pi\) belongs to
\(\operatorname{supp}(\lambda_X)\) if and only if \(\pi\) is equivalent to an
irreducible constituent of \(X_G(TA,\Gamma,\lambda,\nu)\), where \((TA,\Gamma,\lambda,\nu)\) is a character datum for the symmetric space \(X\).
\end{theorem}

\begin{proof}
Suppose first that \((TA,\Gamma,\lambda,\nu)\) is a character datum for
\(X\). Let
\[
M'A'=Z_G(\faq),
\]
and let \(P'=M'A'N'\) be a cuspidal \(\sigma\)-parabolic subgroup with this
Levi factor. By Lemma~\ref{lem:welltemp_hereditary}, the smaller symmetric
space
\[
M'/(M'\cap H)
\]
is well-tempered. It has discrete series because \(\ftq\) is a compact Cartan
subspace of \(\fm'\cap\fq\). Applying Theorem~\ref{thm:ds_para} to
\(M'/(M'\cap H)\), with Cartan subgroup \(T\exp(\fah \cap \fm')\subset M'\), together with Lemma \ref{lem:G-regular-standard-irreducible}, shows
that the standard representation
\[
\xi= X_{M'}(T\exp(\fah \cap \fm'),\Gamma,\lambda,0)
\]
is irreducible and belongs to \(\widehat{M'}_{M'\cap H,\ds}\). By induction in
stages, the standard representation
\(X_G(TA,\Gamma,\lambda,\nu)\) is 
\[
\Ind_{P'}^G(\xi\otimes e^\nu\otimes 1).
\]
Here \(\nu\in i\faq^\ast\) is viewed as a unitary character of \(A'\) by extension by zero on the remaining split directions $\fah'$. Hence
every irreducible constituent of the standard representation \(X_G(TA,\Gamma,\lambda,\nu)\) belongs to
\(\operatorname{supp}(\lambda_X)\) by Theorem~\ref{thm_description}.

Conversely, let \(\pi\in\operatorname{supp}(\lambda_X)\). By
Theorem~\ref{thm_description}, there exist a \(\theta\)-stable Cartan
subspace \(\ftq\oplus\faq\subset\fq\), a cuspidal \(\sigma\)-parabolic subgroup
\[
P'=M'A'N',
\qquad
M'A'=Z_G(\faq),
\]
a  representation of~$M'$ in the discrete spectrum
\[
\xi\in\widehat{M'}_{M'\cap H,\ds},
\]
and \(\nu\in i\faq^\ast\), such that \(\pi\) is an irreducible constituent of
\[
\Ind_{P'}^G(\xi\otimes e^\nu\otimes1).
\]
The symmetric space \(M'/(M'\cap H)\) is well-tempered by
Lemma~\ref{lem:welltemp_hereditary} and has discrete series. Therefore
Theorem~\ref{thm:ds_para}, applied to \(M'/(M'\cap H)\), gives a discrete
character datum \((T\exp(\fah \cap \fm'),\Gamma,\lambda)\) for \(M'\) such that
\[
\xi \cong X_{M'}(T\exp(\fah\cap \fm'),\Gamma,\lambda,0),
\]
with \(\Gamma|_{T\cap H}=1\) and with \(\lambda\in i\ftq^\ast\) \(M'\)-regular.
Adjoining the split factor \(\exp(\fa')\) and keeping the same
\(\nu\in i\faq^\ast\), we obtain a character datum
\[
(TA,\Gamma,\lambda,\nu)
\]
for \(X\), and the induced representation above is
\(X_G(TA,\Gamma,\lambda,\nu)\).
\end{proof}

\subsection{The shape of the Plancherel support}\label{subsec-shape-support}

Theorem~\ref{thm:tempered spec} describes the Plancherel support
\(\operatorname{supp}(\lambda_X)\) by character data for \(X\). To understand what this tells us about the geometry of \(\operatorname{supp}(\lambda_X)\), we shall 
describe the image of this support in the universal Hausdorff quotient \(\widehat G_{\temp,\mathrm{Haus}}\) of the tempered
dual. As we shall see, the corresponding subset of \(\widehat G_{\temp,\mathrm{Haus}}\) has a rather simple description. Therefore the results of this section will yield a simple description of \(\operatorname{supp}(\lambda_X)\), {modulo} the reducibility of the standard representations \(X_G(TA,\Gamma,\lambda,\nu)\).

\begin{blank} Recall from Theorem~\ref{thm:HC_disjoint} that the connected components of \(\widehat G_{\temp,\mathrm{Haus}}\) are in natural bijection with the set $\mathcal{C}_G$ of   \(K\)-conjugacy classes
\([TA,\Gamma,\lambda]\) of discrete character data for \(G\). If \(c\) is an element of $\mathcal{C}_G$, with discrete character datum representative \((TA,\Gamma,\lambda)\), the corresponding connected component \(\widehat G_{\temp,\mathrm{Haus}}(c)\) of \(\widehat G_{\temp,\mathrm{Haus}}\)  has the form 
\[
i\fa^\ast/W(\Gamma,\lambda)
\]
where $\fa$ is the (real) Lie algebra of~$A$, and $W(\Gamma, \lambda)$ is the finite group in Notation~\ref{notation-w}.
 We have also seen that every discrete character datum \(D\)
for \(X\) (Definition \ref{def:dataC})  defines, after certain extensions by zero, a discrete
character datum for \(G\): see Lemma~\ref{lem:rho_rho_eq}.
\end{blank}

\begin{blank}
\label{def:CX-and-DX}
Let \(\mathcal D_X\) be the set of discrete character data for \(X\). We say that two discrete character data $(T_1A_1,\Gamma_1,\lambda_1)$ and $(T_2A_2,\Gamma_2,\lambda_2)$ for $X$ are \emph{$\sigma$-conjugate} if there is $k\in K$ such that 
\[
\mathrm{Ad}_k (T_1A_1, \Gamma_1, \lambda_1) = (T_2A_2, \Gamma_2, \lambda_2) \,\, \text{and} \,\,  \mathrm{Ad}_k({\fa_1}_\fq)={\fa_2}_\fq.
\]
Let \(\mathcal D_X / {\sim_\sigma}\) be the set of $\sigma$-conjugacy classes $[TA,\Gamma,\lambda]_\sigma$ of discrete character data for \(X\). If
\(D=(TA,\Gamma,\lambda)\in\mathcal D_X\), write
\[
[D]_G=[TA,\Gamma,\lambda]\in\mathcal C_G
\]
for the corresponding $K$-conjugacy class of discrete character data for $G$.  The map 
\begin{align*} \kappa\colon\quad  \mathcal{D}_X & \to \mathcal{C}_G\\ 
D & \mapsto  [D]_G
\end{align*}
descends to a map
\[
\overline{\kappa}\colon\quad  \mathcal{D}_X /{\sim_\sigma} \to \mathcal C_G.
\]
The latter map $\overline{\kappa}$ is finite-to-one, but not injective in general. 

Let $\mathcal{C}_X \subset \mathcal{C}_G$ be the image of $\kappa$, and for $c \in \mathcal{C}_X$, let $\mathcal{D}_X(c)$ be the fiber of $\kappa$ over~$c$, i.e. $\mathcal{D}_X(c) =  \{ D \in \mathcal{D}_X  \ : \ [D]_G = c\}$.

The intersection \( \supp(\lambda_X)_{\mathrm{Haus}} \cap \widehat G_{\temp,\mathrm{Haus}}(c)\) is a single connected component of \(\supp(\lambda_X)_{\mathrm{Haus}} \). We shall see that if we identify  \(\widehat G_{\temp,\mathrm{Haus}}(c)\) with  \(i\fa^\ast/W(\Gamma,\lambda)\), the corresponding component of \(\supp(\lambda_X)_{\mathrm{Haus}}\)  can be described as the image in \( i\fa^\ast/W(\Gamma,\lambda)\) of a finite (explicit) union of linear subspaces of~$i\fa^\ast$.
\end{blank}

\begin{notation}
Fix \(c\in\mathcal C_X\) and $D_0= (TA,\Gamma,\lambda)\in \mathcal D_X(c)$.

Write
\[
\fb=\ft\oplus\fa,
\qquad
\ft=\fth\oplus\ftq,
\qquad
\fa=\fah\oplus\faq,
\]
for the corresponding \(\sigma\)-decompositions, and put
\[
V_\fq:=i\faq^\ast\subset i\fa^\ast.
\]
Let $W(\lambda)$ and $
W(\Gamma,\lambda)$ be as in Notation~\ref{notation-w}, so that 
\(\widehat G_{\temp,\mathrm{Haus}}(c)\) is homeomorphic to $
i\fa^\ast/W(\Gamma,\lambda)$.The group $W(\lambda)$ acts on $i\fa^\ast$, and we let $W_\fq(\lambda)$ be the global stabilizer of the subspace $V_\fq$ for that action:
\[
W_\fq(\lambda)
=
\{u\in W(\lambda):u(V_\fq)=V_\fq\}.
\]
\end{notation}

\begin{lemma}\label{lem:weak-weyl-normal-form}
Let \(D\in\mathcal D_X(c)\). Choose a \(K\)-conjugacy carrying the
discrete character datum for \(G\) underlying \(D\) to
\(D_0=(TA,\Gamma,\lambda)\), and use it to identify the full continuous
parameter space attached to \(D\) with the fixed space \(i\fa^\ast\).
Under this identification, the image of the relative continuous parameter
space attached to \(D\) has the form
\[
w^{-1}V_\fq\subset i\fa^\ast
\]
for some \(w\in W(\lambda)\). The resulting class of \(w\) in
\[
W_\fq(\lambda)\backslash W(\lambda)/W(\Gamma,\lambda)
\]
depends only on the $\sigma$-conjugacy class $[D]_\sigma \in \mathcal{D}_X/{\sim_\sigma}$, and not on the choice of $K$-conjugacy or the choice of $w$.
\end{lemma}

\begin{proof}
Write \(D=(T_DA_D,\Gamma_D,\lambda_D)\), and choose \(k\in K\)
such that \(kD=D_0\) as discrete character data for \(G\).  By adjusting
\(k\) on the left by a representative of an element
\(w\in W(\lambda)\), we can arrange that $g:=wk$ carries the \(\sigma\)-eigenspace decomposition of
\(\operatorname{Lie}(T_DA_D)\) to the $\sigma$-eigenspace  decomposition of \(\fb\). To see this, since $Z_{\fg^\lambda}(\faq)$ is abelian, a generic element in $\faq$ determines a positive system $\Sigma^+(\fg^\lambda, \fa)$ such that $\sigma(\Sigma^+(\fg^\lambda, \fa))= - \Sigma^+(\fg^\lambda, \fa)$. Choose $\Sigma^+(\fg^{\lambda_D}, \fa_D)$ similarly. Then, there is $w\in  W(\lambda)$ so that $g=wk$ carries $\Sigma^+(\fg^\lambda, \fa)$ to $\Sigma^+(\fg^{\lambda_D}, \fa_D)$. This, together with $(wk)\cdot \lambda_D = \lambda$, implies $\mathrm{Ad}_g$ and $\mathrm{Ad}_{\sigma(g)} =\sigma(\mathrm{Ad}_g)$ coincide as maps from $\ft_D\oplus \fa_D$ to $\ft \oplus \fa$. Hence, $\mathrm{Ad}_g$ respects the $\sigma$-eigenspace decompositions. We also have \(gD=wD_0\).  Since \(g\) carries the source relative continuous
parameter subspace to \(V_\fq\), transport by \(k=w^{-1}g\) carries it
to \(w^{-1}V_\fq\).

Suppose that \(k_iD=D_0\) and that
\(g_i=w_i k_i\), for \(i=1,2\), are two such eigenspace-preserving
transports.  The Weyl class of \(k_2k_1^{-1}\) lies in
\(W(\Gamma,\lambda)\), because this element stabilizes \(D_0\).
The Weyl class of \(g_2g_1^{-1}\) lies in \(W_\fq(\lambda)\), because
this element preserves the fixed relative parameter subspace.  Since
\[
g_2g_1^{-1}
=
w_2(k_2k_1^{-1})w_1^{-1},
\]
the elements \(w_1\) and \(w_2\) lie in the same
\(W_\fq(\lambda)\)--\(W(\Gamma,\lambda)\) double coset. 

Given $D\in \mathcal{D}_X(c)$ and $g=wk$ as above, if $D=k'D'$ for $D' \in \mathcal{D}_X(c)$ where $k'\in K$ preserves relative continuous parameter spaces, $g'=wkk'$ sends $D'$ to $D_0$ preserving the $\sigma$-eigenspace decomposition of the relevant Cartan subspaces for $D'$ and $D_0$. Hence, the resulting class of $w$ depends only on the $\sigma$-conjugacy class of $D$.   
\end{proof}

\begin{definition}\label{def:branch-set}
The set of \emph{branch-indexing double cosets} for in the component of \(\widehat G_{\temp,\mathrm{Haus}}(c)\), relative to the fixed representative $D_0=(TA,\Gamma,\lambda)\in \mathcal{D}_X(c)$, is the subset
\[
\operatorname{Br}_X(D_0)
\subseteq
W_\fq(\lambda)\backslash W(\lambda)/W(\Gamma,\lambda)
\]
consisting of the double-coset classes produced by
Lemma~\ref{lem:weak-weyl-normal-form} as \(D\) ranges over
\(\mathcal D_X(c)\).
Let $\mathcal B_X(D_0)$ be a subset of $W(\lambda)$ which contains one representative for each double coset in $\operatorname{Br}_X(D_0)$. Define a subset $\Sigma_X(D_0)$ of $i\fa^\ast/W(\Gamma,\lambda)$ by
\[
\Sigma_X(D_0)
=
\left(
\bigcup_{w\in\mathcal B_X(D_0)} w^{-1}V_\fq
\right)\Big/ W(\Gamma,\lambda).
\]
In other words, \(\Sigma_X(D_0)\) is the union, over
\([w]\in\operatorname{Br}_X(D_0)\), of the images of 
\(w^{-1}V_\fq\) in \(i\fa^\ast/W(\Gamma,\lambda)\). We may call these images the ``branches'' of \(\Sigma_X(D_0)\).

Let $\Sigma_X(c)$ be the subset of the connected component $\widehat G_{\temp,\mathrm{Haus}}(c)$ that corresponds to \(\Sigma_X(D_0)\) under the identification $G_{\temp,\mathrm{Haus}}(c)\simeq i\fa^\ast/W(\Gamma,\lambda)$.
\end{definition}

The subset \(\Sigma_X(D_0)\) is independent of the chosen representative
set \(\mathcal B_X(D_0)\). Indeed, if
\[
w'=uwv,
\qquad
u\in W_\fq(\lambda),\quad
v\in W(\Gamma,\lambda),
\]
then
\[
(w')^{-1}V_\fq
=
v^{-1}w^{-1}u^{-1}V_\fq
=
v^{-1}w^{-1}V_\fq.
\]
Thus \(w^{-1}V_\fq\) and \((w')^{-1}V_\fq\) have the same image in
\(i\fa^\ast/W(\Gamma,\lambda)\). The subset $\Sigma_X(c)$ of $\widehat G_{\temp,\mathrm{Haus}}(c)$ is independent of all choices.

\begin{theorem}[Shape of the Hausdorff support]\label{thm:shape-practical}
Let \(X=G/H\) be a well-tempered symmetric space, and assume that \(G\) is
in Vogan's class.   
\begin{enumerate}
    \item The image of \(\operatorname{supp}(\lambda_X)\)
in the Hausdorff quotient \(\widehat G_{\temp,\mathrm{Haus}}\) is
\[
\operatorname{supp}(\lambda_X)_{\mathrm{Haus}}
\cong
\bigsqcup_{c\in\mathcal C_X}\Sigma_X(c).
\]
\item For each \(c\in\mathcal C_X\), choose a representative $D_0 \in \mathcal{D}_X(c)$; then the subset
\(\Sigma_X(D_0)\) is a finite union of the images of the linear subspaces
\(w^{-1}V_\fq\subset i\fa^\ast\) in 
\(i\fa^\ast/W(\Gamma,\lambda)\). Thus $\Sigma_X(c)$ is connected. 
\item 
In particular, the connected components of $\operatorname{supp}(\lambda_X)_{\mathrm{Haus}}$ are the sets $\Sigma_X(c)$, for $c \in \mathcal{C}_X$.
\end{enumerate}

\end{theorem}

\begin{proof}
By Theorem~\ref{thm:tempered spec}, the support is the union of the
irreducible constituents of standard representations attached to character
data for \(X\). The Vogan parametrization for \(G\) in
Theorem~\ref{thm:HC_disjoint} identifies the component of
\(\widehat G_{\temp,\mathrm{Haus}}\) indexed by \(c=[D_0]_G\) with
\[
i\fa^\ast/W(\Gamma,\lambda).
\]
By Lemma~\ref{lem:weak-weyl-normal-form}, for each \(D\in\mathcal D_X(c)\), the relative continuous parameter
space of \(D\) contributes to this component the image of
\(w^{-1}V_\fq\) for some \(w\in W(\lambda)\), where the double-coset class $W_\fq(\lambda)wW(\Gamma,\lambda)$ is determined by $D$. Taking the union
over all \(D\in\mathcal D_X(c)\) gives exactly \(\Sigma_X(D_0)\).  Finally, taking the disjoint union over \(c\in\mathcal C_X\) gives the asserted image of the support in the Hausdorff quotient.
\end{proof}

\begin{corollary}\label{cor_singleton}
Let $X=G/H$ be a well-tempered symmetric space, and assume $G$ is in Vogan's class. A representation in $\supp(\lambda_X)$ belongs to the discrete series of $X$ if and only if it is an isolated point of $\supp(\lambda_X)$.
\end{corollary} 
\begin{proof} 
By Theorem~\ref{thm:shape-practical}, the component of
\(\operatorname{supp}(\lambda_X)_{\mathrm{Haus}}\) indexed by a discrete character datum 
\([TA,\Gamma,\lambda]\) for $X$ is the image of a finite union of linear
subspaces \(w^{-1}(i\mathfrak a_{\mathfrak q}^\ast)\), all containing
\(0\). It is a singleton exactly when
\(\mathfrak a_{\mathfrak q}=0\).

If \(\mathfrak a_{\mathfrak q}\neq0\), continuity of parabolic
induction shows that each constituent over a given parameter is a
Fell limit of constituents over nearby relative parameters, and
therefore is not isolated. If
\(\mathfrak a_{\mathfrak q}=0\), the inducing
\(\sigma\)-parabolic is \(G\) itself, and
Lemma~\ref{lem:G-regular-standard-irreducible} shows that the
corresponding standard representation is irreducible. Its component
is therefore a single isolated representation. These are precisely the discrete-series representations of
\(X\).
\end{proof}

In the rest of this subsection, we shall spell out three observations which may be helpful in understanding what Theorem~\ref{thm:shape-practical} says about the geometry of $\supp(\lambda_X)$. \\

Our first theme is  the overlap between the branches in Definition~\ref{def:branch-set}: we shall see that their regular parts do not overlap (see Proposition~\ref{prop:generic-branch-separation} for the precise statement). 

Let $D_0=(TA,\Gamma,\lambda)\in\mathcal D_X(c)$. Let \(\fg^\lambda\) denote the complex reductive subalgebra of \(\fg^\bC\) generated
by \(\fb^\bC\) and the root spaces for roots whose restrictions to \(\ft^\bC\)
have zero pairing with \(\lambda\).  Since \(D_0\) is a discrete character datum for $X$,
\(\lambda\) is regular for the Levi attached to \(\faq\), and therefore
\[
Z_{\fg^\lambda}(\faq)=\fb^\bC.
\]
Define
\[
\faq^{\mathrm{reg},\lambda}
=
\{Y\in\faq:Z_{\fg^\lambda}(Y)=\fb^\bC\}.
\]
This is a Zariski open dense subset of \(\faq\).  After choosing a
\(K\)-invariant inner product to identify \(V_\fq = i\faq^\ast\) with \(\faq\), let
\[
V_\fq^{\mathrm{reg},\lambda}
\subset V_\fq
\]
be the corresponding open dense subset.

\begin{lemma}\label{lem:regular-branch-stabilizer}
Let \(w\in W(\lambda)\).  If
\[
w(V_\fq^{\mathrm{reg},\lambda})\cap V_\fq^{\mathrm{reg},\lambda}
\neq\varnothing,
\]
then \(w\in W_\fq(\lambda)\).
\end{lemma}

\begin{proof}
We give the argument on \(\faq\), using the invariant inner product to pass
between \(\faq\) and \(i\faq^\ast\).  Choose a representative
\(n\in N_K(\fb)\) of \(w\), and suppose that
\(Y\in\faq^{\mathrm{reg},\lambda}\) and \(\Ad(n)Y\in\faq^{\mathrm{reg},\lambda}\).
Since \(Y,\Ad(n)Y\in\faq\), both are in the
\((-1)\)-eigenspace of \(\sigma\). Moreover,
\(\sigma\circ\Ad(n)=\Ad(\sigma(n))\circ\sigma\). Therefore
\[
-\Ad(n)Y
=
\sigma(\Ad(n)Y)
=
\Ad(\sigma(n))\sigma(Y)
=
-\Ad(\sigma(n))Y.
\]
It follows that
\(\Ad(n)Y=\Ad(\sigma(n))Y\). Applying
\(\Ad(\sigma(n)^{-1})\) to this equality gives
\[
\Ad(\sigma(n)^{-1}n)Y=Y.
\]
The element \(\sigma(n)^{-1}n\) represents an element of the Weyl group of
\(\fg^\lambda\), and it fixes the regular element \(Y\).  Since $Y$ is regular, its Weyl-group stabilizer in \(\fg^\lambda\) is trivial.  Thus
\(\sigma(n)^{-1}n\) acts trivially on \(\fb\).  Consequently, \(\Ad(n)\) and \(\Ad(\sigma(n))\) have the same action on
\(\fb\). Hence, for every \(Z\in\fb\),
\[
\sigma(\Ad(n)Z)
=
\Ad(\sigma(n))\sigma(Z)
=
\Ad(n)\sigma(Z),
\]
where the last equality uses \(\sigma(Z)\in\fb\). Thus \(\Ad(n)\)
commutes with \(\sigma\) on~\(\fb\). It follows that \(\Ad(n)\) preserves the
\(\sigma\)-eigenspace decomposition of \(\fa\), in particular
\(\Ad(n)\faq=\faq\).  Hence \(w(V_\fq)=V_\fq\), i.e.\ \(w\in W_\fq(\lambda)\).
\end{proof}

\begin{proposition}[Generic separation of branches]\label{prop:generic-branch-separation}
Let \(w_1,w_2\in W(\lambda)\).  If the images of
\(w_1^{-1}V_\fq^{\mathrm{reg},\lambda}\) and
\(w_2^{-1}V_\fq^{\mathrm{reg},\lambda}\) in
\(i\fa^\ast/W(\Gamma,\lambda)\) intersect, then
\[
W_\fq(\lambda)w_1W(\Gamma,\lambda)
=
W_\fq(\lambda)w_2W(\Gamma,\lambda).
\]

\end{proposition}
Thus the “branches” attached to distinct elements of 
\(W_\fq(\lambda)\backslash W(\lambda)/W(\Gamma,\lambda)\) have no regular elements in common. 

\begin{proof}
If the images meet, then for some \(u\in W(\Gamma,\lambda)\) and some
\(\nu_1,\nu_2\in V_\fq^{\mathrm{reg},\lambda}\),
\[
\nu_2=w_2u w_1^{-1}\nu_1.
\]
By Lemma~\ref{lem:regular-branch-stabilizer}, the element
\(w_2u w_1^{-1}\) belongs to \(W_\fq(\lambda)\).  Thus, $w_1$ and $w_2$ belong to the same double coset.
\end{proof}

\begin{remark}\label{rem:branch-intersections}
All branches contain the image of
\(0\), and they may meet along singular lower-dimensional strata.  The
preceding proposition says that this is the only possible kind of overlap:
distinct branch labels have disjoint regular parts.  Example \ref{subsec_example_SL3} below exhibits two branch images whose union has three one-dimensional arms meeting at one point.
\end{remark}

Our second group of remarks  on Theorem~\ref{thm:shape-practical} concerns the situation where all branches come from discrete $X$-character data which have the same underlying Cartan subgroup. 

\begin{definition}\label{def:admissible-w}
Let \(D_0=(TA,\Gamma,\lambda)\in\mathcal D_X(c)\).  Set
\[
\mathcal A_X^{\mathrm{sc}}(D_0)
=
\{w\in W(\lambda):(TA,w\Gamma,\lambda)\in\mathcal D_X\}.
\]
and 
\[
\operatorname{Br}^{\mathrm{sc}}_X(D_0)
:=
\{W_\fq(\lambda)wW(\Gamma,\lambda):w\in
\mathcal A_X^{\mathrm{sc}}(D_0)\}.
\]
\end{definition}

The condition for an element $w$ of $W(\lambda)$ to belong to  $A_X^{\mathrm{sc}}(D_0)$  is simply that $(TA, w\Gamma, \lambda)$ satisfies the
conditions of Definition~\ref{def:dataC}; in particular, this is equivalent to the 
condition
\[
(w\Gamma)|_{T\cap H}=1.
\]
\

\begin{proposition}\label{lem:same-cartan-weyl-form}
The inclusion $\operatorname{Br}^{\mathrm{sc}}_X(D_0) \subset \operatorname{Br}_X(D_0)$ holds for any $D_0$. Furthermore, 
if $\operatorname{Br}^{\mathrm{sc}}_X(D_0)=\operatorname{Br}_X(D_0)$, then
\[
\Sigma_X(D_0)
=
\left(
\bigcup_{w\in\mathcal A_X^{\mathrm{sc}}(D_0)} w^{-1}V_\fq
\right)\Big/ W(\Gamma,\lambda).
\]
\end{proposition}

\begin{proof}
If \(w\in\mathcal A_X^{\mathrm{sc}}(D_0)\), then
\((TA,w\Gamma,\lambda)\) is itself a discrete character datum for~$X$ whose discrete class for \(G\) is
\(c=[D_0]_G\). Its relative continuous parameter space is
\(V_\fq=i\faq^\ast\). Transporting its underlying datum for \(G\) back to the representative \(D_0\) by a representative of \(w^{-1}\) sends this parameter space to
\(w^{-1}V_\fq\).  This proves that its double-coset class belongs to
\(\operatorname{Br}_X(D_0)\).  The final formula follows immediately if all
branch classes are represented in this way.
\end{proof}

When  Proposition~\ref{lem:same-cartan-weyl-form} applies, i.e. when $\operatorname{Br}^{\mathrm{sc}}_X(D_0)=\operatorname{Br}_X(D_0)$, we shall say that the \emph{same-Cartan situation} occurs for the component of $\supp(\lambda_X)_{\mathrm{Haus}}$ attached to $D_0$. Our next statement shows that this situation always occurs as soon as $H$ is the full fixed-point subgroup $G^\sigma$. 

\begin{proposition}\label{prop:same-cartan} Suppose $H=G^\sigma$. Then for any discrete character datum $D_0=(TA, \Gamma, \lambda)$ for $X$, we have $\operatorname{Br}^{\mathrm{sc}}_X(D_0)=\operatorname{Br}_X(D_0)$. Thus, 
\[
\Sigma_X(D_0)
=
\left(
\bigcup_{w\in\mathcal A_X^{\mathrm{sc}}(D_0)} w^{-1}V_\fq
\right)\Big/ W(\Gamma,\lambda).
\]
\end{proposition}
\begin{proof} By the proof of Lemma~\ref{lem:weak-weyl-normal-form}, for any $w\in\operatorname{Br}_X(D_0)$ there are $k\in K$ and a discrete character datum $D=(T_DA_D, \Gamma_D, \lambda_D)$ for $X$ with the property $(wk)\cdot D=w\cdot D_0 = (TA,w\Gamma,\lambda)$ where $wk$ maps $\ft_D\oplus \fa_D$ to $\ft\oplus \fa$ preserving the $\sigma$-eigenspaces. Thus $\mathrm{Ad}_{wk}\colon \ft_D\oplus \fa_D \to \ft\oplus \fa$ coincides with $\sigma(\mathrm{Ad}_{wk})=\mathrm{Ad}_{\sigma(wk)}$. It follows that $\sigma(wk)\cdot (wk)^{-1} \in Z_K(\ft\oplus \fa)=T$. Thus $\sigma(wk)=zwk$ for some $z\in T$. Using this, we can see that \[ \sigma(\mathrm{Ad}_{wk}(T_D \cap G^{\sigma})) = \mathrm{Ad}_{zwk}(T_D \cap G^{\sigma}) = \mathrm{Ad}_{wk}(T_D \cap G^{\sigma});\] therefore, $\mathrm{Ad}_{wk}(T_D \cap G^{\sigma}) \subset T\cap G^\sigma$. Similarly, $\mathrm{Ad}_{(wk)^{-1}}(T \cap G^{\sigma}) \subset T_D\cap G^\sigma$. These imply $\mathrm{Ad}_{wk}(T_D \cap G^{\sigma}) = T\cap G^\sigma$. 

Since $\Gamma_D$ is trivial on $T_D\cap G^{\sigma}$, $w\cdot \Gamma = (wk)\cdot \Gamma_D$ is trivial on $T\cap G^\sigma$. Thus, $(TA, w\Gamma, \lambda)=w\cdot (TA, \Gamma,\lambda) = (wk)\cdot (T_DA_D, \Gamma_D, \lambda_D)$ is a discrete character datum for $X$. Here, we used the assumption $H=G^\sigma$. The first assertion follows from this. The second assertion follows from Lemma~\ref{lem:same-cartan-weyl-form}.
\end{proof}

Our third and final comment on Theorem~\ref{thm:shape-practical} is that there are many interesting situations where all branches in a component $\Sigma_X(c)$ collapse to a single branch. In the following statement, we do not assume that $H$ is the full fixed-point group $G^\sigma$. 

\begin{proposition}[Collapse of branches]\label{prop:branch-collapse}
Suppose that
\[
W(\lambda)=W_\fq(\lambda)W(\Gamma,\lambda).
\]
Then $\operatorname{Br}_X(D_0)$ consists of just the identity coset $
\{W_\fq(\lambda)W(\Gamma,\lambda)\}$,
and
\[
\Sigma_X(D_0)
=
\operatorname{image}
\bigl(V_\fq\longrightarrow i\fa^\ast/W(\Gamma,\lambda)\bigr).
\]
In particular, the subspaces \(w^{-1}V_\fq\), for
\(w\in W(\lambda)\), all have the same image in the component
\(i\fa^\ast/W(\Gamma,\lambda)\) of
\(\widehat G_{\temp,\mathrm{Haus}}\).
\end{proposition}

\begin{proof}
The displayed factorization says exactly that
\(W_\fq(\lambda)\backslash W(\lambda)/W(\Gamma,\lambda)\) is a
singleton.  The datum \(D_0\) contributes this class, so the branch set is
the displayed singleton.  Equivalently, if \(w=uv\) with
\(u\in W_\fq(\lambda)\) and \(v\in W(\Gamma,\lambda)\), then
\[
w^{-1}V_\fq=v^{-1}u^{-1}V_\fq=v^{-1}V_\fq,
\]
which has the same image as \(V_\fq\) in
\(i\fa^\ast/W(\Gamma,\lambda)\).  The formula for \(\Sigma_X(D_0)\)
follows.
\end{proof}

\begin{remark}\label{rem:when-same-cartan-simplifies}
The factorization hypothesis of Proposition~\ref{prop:branch-collapse}
holds in particular if either
\(W_\fq(\lambda)=W(\lambda)\) or
\(W(\Gamma,\lambda)=W(\lambda)\). If $\lambda$ is $G$-regular, both conditions hold. 

The first case, \(W_\fq(\lambda)=W(\lambda)\), occurs for instance  when
\(\faq=\fa\) or \(0\). These two limiting cases are useful to keep in mind. If \(\faq=\fa\), then \(V_\fq=i\fa^\ast\), and
\(\Sigma_X(D_0)\) is the entire component
\(i\fa^\ast/W(\Gamma,\lambda)\) of
\(\widehat G_{\temp,\mathrm{Haus}}\). If \(\faq=0\), then we have
\(V_\fq=\{0\}\), and the corresponding connected component
\(\Sigma_X(D_0)\) of
\(\operatorname{supp}(\lambda_X)_{\mathrm{Haus}}\) is a point.

The second case \(W(\Gamma,\lambda)=W(\lambda)\) is the situation which will occur when we turn to spaces of type $G_\C/G_\R$ in  Section~\ref{sec:complex}: the relevant Cartan
subgroups are connected, so the character \(\Gamma\) is determined by its
differential, and any element which stabilizes \(\lambda\) also stabilizes \(\Gamma\).

\end{remark}

By contrast,
Example~\ref{subsec_example_SL3} below has
\(W(\Gamma,\lambda)\subsetneq W(\lambda)\), and $\operatorname{Br}_X(D_0)$ has two elements: one branch contributes a line while the other contributes a half-line to the same component of the spectrum.

\subsection{Example: \texorpdfstring{$X=\mathrm{SL}(3,\R)/\mathrm{GL}(2,\R)_0$}{SL(3,R)/GL(2,R)0}}
\label{subsec_example_SL3}

We illustrate Theorem~\ref{thm:shape-practical} in a case where the part
of \(\operatorname{supp}(\lambda_X)_{\mathrm{Haus}}\) lying in a single
connected component of \(\widehat G_{\temp,\mathrm{Haus}}\) is a union of several
one-dimensional branches. This example falls under the same-Cartan situation of Proposition~\ref{lem:same-cartan-weyl-form}.  

Let $G=\SL(3,\bR)$, and let $H\cong\GL(2,\bR)_0$ be the identity
component of the fixed-point group of
\[
\sigma=\Ad\!\begin{bmatrix}
1&0&0\\
0&-1&0\\
0&0&-1
\end{bmatrix}.
\]
Let
\[
V_\fq =\faq
=
\bR(e_{12}+e_{21})
=
\bR\begin{bmatrix}
0&1&0\\
1&0&0\\
0&0&0
\end{bmatrix}.
\]
Then $Z_\fg(\faq)=\fa$ is a split Cartan subalgebra of $\fg$, with
\[
\fah
=
\bR\begin{bmatrix}
1&0&0\\
0&1&0\\
0&0&-2
\end{bmatrix}.
\]
Thus $X$ is well-tempered. Let $TA=Z_G(\fa)$. Then
\[
T=Z_K(\fa)
=
\left\langle
\begin{bmatrix}
-1&0&0\\
0&-1&0\\
0&0&1
\end{bmatrix},
\begin{bmatrix}
0&1&0\\
1&0&0\\
0&0&-1
\end{bmatrix}
\right\rangle
\]
has order~$4$, and one checks that
\[
T\cap H=\{e\}.
\]

In this example, the Levi subgroup appearing in
Definition~\ref{def:dataC} is determined by
\[
M'A'=Z_G(\faq)=Z_G(\fa)=TA.
\]
Thus \(A'=A\) and \(M'=T\). Since \(T\) is finite, \(\fm'=0\), and hence
\[
\Sigma(\fm'^\bC,\ft^\bC)=\varnothing.
\]
The \(M'\)-regularity condition is therefore vacuous, so
\(\lambda=0\) is \(M'\)-regular.

Write
\[
\widehat T=\{\Gamma_0,\Gamma_1,\Gamma_2,\Gamma_3\},
\]
where $\Gamma_0$ is the trivial character. Then every
$(TA,\Gamma_i,0)$ is a discrete character datum for $X$. The attached Weyl group is
\[
W(0)=W_K(\fa)\cong S_3,
\]
the $A_2$-Weyl group. It acts faithfully and transitively on the three
nontrivial characters $\Gamma_1,\Gamma_2,\Gamma_3$. In particular,
\[
[TA,\Gamma_1,0]=[TA,\Gamma_2,0]=[TA,\Gamma_3,0]
\quad\text{in }\mathcal C_G.
\]
Put \(D_0=(TA,\Gamma_1,0)\). Arrange the labeling of the nontrivial
characters so that, if  \(w_i\in W(0)\) denotes the root reflection fixing
\(\Gamma_i\), 
\[
W(\Gamma_1,0)=\{1,w_1\},
\qquad
W_\fq(0)=\{1,w_3\},
\qquad
w_2\Gamma_1=\Gamma_3,
\qquad
w_3\Gamma_1=\Gamma_2.
\]

Since $T\cap H$ is trivial, we are in the the
same-Cartan situation of Proposition~\ref{lem:same-cartan-weyl-form}.
The branch set therefore consists of the full double cosets:
\[
\operatorname{Br}_X(D_0)
=
\left\{
W_\fq(0)W(\Gamma_1,0),
\quad
W_\fq(0)w_2W(\Gamma_1,0)
\right\} = W_\fq(0)\backslash W(0) /W(\Gamma_1,0).
\]
The component \(\Sigma_X(D_0)\) of
\(\operatorname{supp}(\lambda_X)_{\mathrm{Haus}}\) indexed by
\(c=[D_0]_G\), is
\[
\operatorname{image}
\bigl(V_\fq\longrightarrow i\fa^\ast/\langle w_1\rangle\bigr)
\;\cup\;
\operatorname{image}
\bigl(w_2^{-1}V_\fq\longrightarrow
      i\fa^\ast/\langle w_1\rangle\bigr).
\]
The first image is a full line, because \(w_1\) exchanges \(V_\fq\)
with a second root line.  The second image is a closed half-line, because
\(w_1\) stabilizes the third root line \(w_2^{-1}V_\fq\) and acts on it
by sign.  Their union is therefore a tripod: three closed half-lines meeting
at the image of \(0\).  The regular parts of these two branch images are
disjoint in the sense of
Proposition~\ref{prop:generic-branch-separation}.

Finally, the standard representations $X_G(TA,\Gamma,0,\nu)$ are
irreducible by \cite[Theorem~5.1]{Wallach71}. Hence, in this example, the
Hausdorff-quotient picture coincides with the actual
representation-theoretic support.

Let $Y$ denote this tripod. The standard $C^\ast$-algebraic description of
tempered blocks identifies the quotient of the relevant tempered block whose spectrum is the closed support subset \(Y\), up to Morita equivalence, with $C_0(Y)$. The one-point compactification
$Y^+$ is obtained by adjoining a common endpoint at infinity to the three
half-lines. It is therefore the graph consisting of two vertices joined by
three edges, and hence is homotopy equivalent to $S^1\vee S^1$. Consequently,
\[
K_0\bigl(C_0(Y)\bigr)
=
\widetilde K^0(Y^+)
=
0,
\qquad
K_1\bigl(C_0(Y)\bigr)
=
K^1(Y^+)
\cong
\mathbb Z^2.
\]
Thus the branching of \(Y\) is detected by \(K\)-theory: \(Y\)
contributes two independent odd \(K\)-theory classes and no even
\(K\)-theory class.

\section{{Symmetric spaces of type \texorpdfstring{$G_{\mathbb C}/G_{\mathbb R}$}{GC/GR}}}\label{sec:complex}

In this section we specialize the support description of Section~\ref{sec:discrete-welltemp}
to symmetric spaces attached to real forms of connected complex semisimple
Lie groups.  In this class several complications disappear.  

Cartan subgroups of the ambient complex group are connected, so the character \(\Gamma\) in a character datum $(TA, \Gamma, \lambda, \nu)$ is determined by \(\lambda\); moreover, the tempered standard representations are irreducible. Consequently, the ambient tempered dual is Hausdorff, and the branch description of Section~\ref{sec:discrete-welltemp}
collapses to a single quotient for each component of \(\widehat G_{\temp}\).

\subsection{The components of the Plancherel support}\label{sec-complex-support}

Throughout this section, let \(G\) be a connected complex semisimple Lie group
with finite center, regarded as a real reductive group.  Let \(\sigma\) be an antiholomorphic involutive automorphism of~\(G\), and let \(H\) be an open
subgroup of the real form \(G^\sigma\).  We write
$X=G/H$, and  call such an \(X\) a symmetric space of type \(G_{\mathbb C}/G_{\mathbb R}\).

Let \(\theta\) be a Cartan involution of \(G\) commuting with \(\sigma\), and
let \(K=G^\theta\).
Let
\[
\fbh=\fth\oplus\fah\subset\fh
\]
be a \(\theta\)-stable Cartan subalgebra of the real form \(\fh\).  Then
\[
\fbq=i\fbh=\ftq\oplus\faq,
\quad \text{where}\quad
\ftq=i\fah,
\quad
\faq=i\fth,
\]
and $\faq$ is a \(\theta\)-stable Cartan subspace of \(\fq=i\fh\).  Conversely, every
\(\theta\)-stable Cartan subspace of \(\fq\) is obtained in this way from a
\(\theta\)-stable Cartan subalgebra of~\(\fh\).  We set
\[
\fb=\fbh\oplus i\fbh=\ft\oplus\fa,
\qquad
\ft=\fth\oplus\ftq,
\qquad
\fa=\fah\oplus\faq,
\]
and denote by
\[
TA=Z_G(\fb)
\]
the corresponding \(\theta\)-stable and \(\sigma\)-stable Cartan subgroup
of \(G\). Since \(G\) is a connected complex semisimple group, \(TA\) is
connected, and hence so is \(T\).

We shall use the Weyl group
\[
W^\sigma(\fbh)
=
W(\ft\oplus\fa)^\sigma
=
N_K(\fbh)/Z_K(\fbh),
\]
where the superscript \(\sigma\) denotes the subgroup whose action on
\(\fb\) commutes with \(\sigma\). Equivalently, it preserves the real Cartan
subalgebra \(\fbh\), and hence both summands in the decomposition
\[
\fb=\fbh\oplus i\fbh.
\]
This group should not, in general, be confused with
$
N_{K\cap H}(\fbh)/Z_{K\cap H}(\fbh).
$
Let \((TA,\Gamma,\lambda)\) be a discrete character datum  for \(X\) based on $TA$. Because \(T\) is connected, the character \(\Gamma\) is determined by its
differential. If \(w\in W(\lambda)\), then \(w\) preserves the
\(\lambda\)-positive root systems defining
\(\rho_{\fm}\) and \(\rho_{\fm\cap\fk}\). Hence, using
Lemma~\ref{lem:rho_rho_eq},
\[
w(d\Gamma)
=
w\bigl(\lambda+\rho_{\fm}-2\rho_{\fm\cap\fk}\bigr)
=
d\Gamma.
\]
Thus \(w\Gamma=\Gamma\), and consequently
\[
W(\Gamma,\lambda)=W(\lambda).
\]
Proposition~\ref{prop:branch-collapse} identifies the part of
\(\supp(\lambda_X)\) lying in the component of \(\widehat G_{\temp}\)
indexed by \([TA,\Gamma,\lambda]\) with the image of $i\faq^*$ under the quotient map
\begin{equation}\label{eq:qlambda}
q_\lambda \colon i\fa^\ast \to i\fa^*/W(\lambda).
\end{equation}
We put
\[
W^\sigma({\lambda})
=
\{w\in W^\sigma(\fbh) \mid w \cdot \lambda=\lambda\} \subset W(\lambda).
\]

\begin{lemma}\label{lem:relative-Weyl} The restriction of the quotient map $q_\lambda$ \eqref{eq:qlambda} to $i\faq^*$ descends to
\[
\overline{q_\lambda} \colon i\faq^*/W^\sigma(\lambda) \to  i\fa^*/W(\lambda).
\]
The map $\overline{q_\lambda}$ is injective; therefore, the image of $i\faq^*$ under $q_\lambda$ is canonically homeomorphic to $i\faq^*/W^\sigma(\lambda)$.
\end{lemma}
\begin{proof} Since $W^\sigma(\lambda)\subset W(\lambda)$ preserves $i\faq^*$, $q_\lambda\mid_{i\faq^*}\colon i\faq^*\to i\fa^* /W(\lambda)$ descends to $\overline{q_\lambda} \colon i\faq^*/W^\sigma(\lambda) \to  i\fa^*/W(\lambda)$. We show that this map is injective. Let $\fg^{\lambda}=Z_\fg(\lambda)$. We note that $Z_{\fg^{\lambda}}(\faq)=\fb$ since $\lambda$ is $M'$-regular where we recall $M'A'=Z_G(\faq)$. Below, we identify $i\fa^*$ with $\fa$ using an invariant metric.

Without loss of generality, we show the injectivity in the case when $\lambda=0$ and $Z_{\fg}(\faq)=\fb$. This reduction can be made by working inside the semisimple part of $\fg^{\lambda}$ after splitting off the center of $\fg^{\lambda}$ on which $W(\lambda)$ acts trivially. 

Let $W=W(\ft\oplus \fa)$ and let $W^\sigma=W(\ft\oplus \fa)^\sigma$. We fix compatible positive systems $\Sigma^+(\fg, \fa)$, $\Sigma^+(\fg, \faq)$. By $Z_\fg(\faq)=\fb$, for any root  $(\alpha, \beta)\in \fah^*\oplus \faq^*$, we have $\beta\neq0$. Thus, the root is positive if and only if $\beta \in  \Sigma^+(\fg, \faq)$. Let $\fa^{++}\subset \fa$ be the dominant closed Weyl chamber. The compatibility between our positive systems implies that $\fa^{++}\cap \faq = \faq^{++}$ is the dominant closed Weyl chamber for the restricted, possibly non-reduced, root system $\Sigma(\fg,\faq)$. Since $\fa^{++}$ is a fundamental domain of the $W$-action on $\fa$, we may identify $i\fa^*/W$ with $\fa^{++}$. By \cite[Theorem~5]{Rossmann79}, which applies since $\faq$ is maximal abelian in $\fp \cap \fq$ under the assumption $Z_\fg(\faq)=\fb$, the Weyl group of the possibly non-reduced root system $\Sigma(\fg,\faq)$ is $W_K(\faq)$ and thus $i\faq^*/W_K(\faq)$ identifies with $\faq^{++}$. Moreover, the action of $W^{\sigma}$ on~$\faq^*$ factors through the surjection $W^{\sigma} \to W_K(\faq)$: the map is surjective since for any $w\in N_K(\faq)$, our assumption that  $Z_\fg(\faq)=\fb$ implies that $w\in N_K(\fa)$. Hence, we have canonical identifications $i\faq^*/W^\sigma \cong i\faq^*/W_K(\faq) \cong \faq^{++}$ under which the map $\overline{q_\lambda}$ identifies with the embedding $\faq^{++}\to \fa^{++}$,  proving the injectivity. The homeomorphism claim follows since the topological embedding $\faq^{++}\to \fa^{++}$ is a~homeomorphism onto its image.
\end{proof}

\subsection{The support and the discrete series}

For connected complex semisimple groups, tempered standard representations are
irreducible by Wallach's irreducibility theorem for the full unitary principal
series \cite{Wallach71}.  Thus the sets \(\Phi_G(TA,\Gamma,\lambda,\nu)\) of
irreducible constituents are singletons.  In particular \(\widehat G_{\temp}\)
coincides with its universal Hausdorff quotient.  By
Theorem~\ref{thm:HC_disjoint}, the tempered dual is the disjoint union of the Vogan parameter spaces \(i\fa^\ast/W(\Gamma,\lambda)\) for \(G\).  Hence
\(\widehat G_{\temp}\) is Hausdorff, and the closed subspace
\(\supp(\lambda_X)\subset\widehat G_{\temp}\) is a locally compact Hausdorff
space.

The following statement summarizes the description of $\supp(\lambda_X)$ that comes out of the above results. Much of it was well known, and we recall that many finer results are known for harmonic analysis on spaces of type $G_\C/G_\R$, in particular thanks to the work of Harinck~\cite{Harinck92, Harinck95, Harinck98base, Harinck98orbital}.

\begin{theorem}\label{thm:GcGr-main}
Let \(X=G/H\) be a symmetric space of type \(G_{\mathbb C}/G_{\mathbb R}\).
Then:
\begin{enumerate}[(1)]
\item \(X\) is well-tempered.

\item For every \(\theta\)-stable Cartan subalgebra
\(\fbh=\fth\oplus\fah\subset\fh\), and for every discrete character datum
\((TA,\Gamma,\lambda)\) for \(X\) based on
\(TA=Z_G(\fbh\oplus i\fbh)\), the map
\[
 i\faq^\ast/W^\sigma(\lambda)
 \longrightarrow \widehat G_{\temp},
 \qquad
 \overline\nu\longmapsto X_G(TA,\Gamma,\lambda,\nu),
\]
is a homeomorphism onto its image.  Denote this image by
$
\widehat X_{\fbh,\lambda}
$.
The support of the Plancherel measure is the disjoint union
\[
\supp(\lambda_X)
=
\bigsqcup_{[\fbh]}
\ \bigsqcup_{[(\Gamma,\lambda)]}
\widehat X_{\fbh,\lambda},
\]
where \([\fbh]\) runs over the \(K\)-conjugacy classes of \(\theta\)-stable
Cartan subalgebras of \(\fh\), and \([(\Gamma,\lambda)]\) runs over the
\(W^\sigma(\fbh)\)-orbits of discrete character data for \(X\) based on
\(TA=Z_G(\fbh\oplus i\fbh)\).

\item \(X\) has discrete series if and only if the real Lie algebra \(\fh\) is
split.  In that case the discrete series of \(X\) is precisely the set of
representations
\[
X_G(TA,\Gamma,\lambda,0)
\]
attached to split Cartan subalgebras of \(\fh\) and to the corresponding
discrete character data for \(X\).  It coincides with the set of tempiric
representations of \(G\) with regular infinitesimal character whose unique
lowest \(K\)-type is \(K\cap H\)-spherical.

\item A representation in \(\supp(\lambda_X)\) belongs to the discrete series
of \(X\) if and only if it is an isolated point of \(\supp(\lambda_X)\).
\end{enumerate}
\end{theorem}

\begin{proof}
Part~(1) is Example~\ref{ex:GcGr-welltemp}.

For part~(2), Theorem~\ref{thm:tempered spec} identifies
\(\supp(\lambda_X)\) with the union of the tempered standard representations
attached to character data for \(X\). By the discussion above,
\(W(\Gamma,\lambda)=W(\lambda)\). Let
\[
q_\lambda:i\fa^\ast\longrightarrow i\fa^\ast/W(\lambda)
\]
be the quotient map. Proposition~\ref{prop:branch-collapse} collapses the
branch union in the component of
\(\widehat G_{\temp,\mathrm{Haus}}\) indexed by
\([TA,\Gamma,\lambda]\) to
\[
q_\lambda(i\faq^\ast) \cong i\faq^\ast/W^\sigma(\lambda),
\]
where the homeomorphism is by Lemma~\ref{lem:relative-Weyl}. 

By Wallach's irreducibility theorem and Vogan's parametrization, this
homeomorphism corresponds precisely to the map
\[
\overline\nu\longmapsto X_G(TA,\Gamma,\lambda,\nu).
\]
Re-indexing
Cartan data by \(K\)-conjugacy classes of \(\theta\)-stable Cartan subalgebras
\(\fbh\subset\fh\) gives the stated disjoint decomposition.

For part~(3), Theorem~\ref{thm:equal-rank} says that \(X\) has discrete series
if and only if \(\fq\) contains a compact Cartan subspace.  In the present
setting compact Cartan subspaces of \(\fq=i\fh\) are precisely the subspaces
\[
\ftq=i\fah,
\]
where \(\fah\) is a split Cartan subalgebra of \(\fh\).  Thus \(X\) has
discrete series if and only if \(\fh\) is split.  Once this is known, the
characterization of the discrete series is Theorem~\ref{thm:ds_para}
specialized to the present class of symmetric spaces.

Part~(4) is the content of Corollary~\ref{cor_singleton}.
\end{proof}

\begin{remark}
Although all Cartan subgroups of the complex group $G$ are conjugate under
$G$, the support decomposition in Theorem~\ref{thm:GcGr-main}(2) is indexed by
$K$-conjugacy classes of Cartan subalgebras of the real form $\fh$, or
 equivalently by $K$-conjugacy classes of Cartan subspaces of $\fq$.  
\end{remark}

\subsection{Example: \texorpdfstring{$\SL(n,\C)/\SL(n,\R)$}{SL(n,C)/SL(n,R)}}

\label{ex:SLnC-SLnR-character-data}
Let
\[
G=\SL(n,\C),
\qquad
H=\SL(n,\R),
\qquad
K=\SU(n).
\]
The \(K\)-conjugacy classes of \(\theta\)-stable Cartan subalgebras of
\(\fh=\fsl(n,\R)\) are indexed by the integers $m$ such that
\[
0\le m\le \lfloor n/2\rfloor.
\]

A representative of type \(m\) is
\[
\fbh^{(m)}
=
\left\{
\begin{pmatrix}
a_1&b_1&&&&&&\\
-b_1&a_1&&&&&&\\
&&\ddots&&&&&\\
&&&a_m&b_m&&&\\
&&&-b_m&a_m&&&\\
&&&&&c_1&&\\
&&&&&&\ddots&\\
&&&&&&&c_{n-2m}
\end{pmatrix}
\in\fsl(n,\mathbb R)
\right\}.
\]
Here \(\fth^{(m)}\) is obtained by setting
\(a_1,\ldots,a_m,c_1,\ldots,c_{n-2m}\) equal to zero, while
\(\fah^{(m)}\) is obtained by setting \(b_1,\ldots,b_m\) equal to zero. If
\[
\fbh^{(m)}=\fth^{(m)}\oplus\fah^{(m)}
\]
is such a Cartan subalgebra, then the corresponding Cartan subspace of \(\fq\)
is
\[
\fbq^{(m)}=i\fbh^{(m)}=\ftq^{(m)}\oplus\faq^{(m)},
\]
with
\[
\dim \ftq^{(m)}=n-m-1,
\qquad
\dim \faq^{(m)}=m.
\]
Thus the connected components of \(\operatorname{supp}(\lambda_X)\) have possible dimensions
\[
0,1,\dots,\lfloor n/2\rfloor.
\]

Let \(\fah^0\subset\fsl(n,\R)\) be the standard split Cartan subalgebra, and
let
\[
T^0=\exp(i\fah^0)\subset\SU(n)
\]
be the corresponding compact torus.  Write a dominant integral weight of
\(\SU(n)\) as
\[
\mu=[\mu_1,\dots,\mu_n]
\in \Z^n/\Z(1,
\dots,1),
\qquad
\mu_1\geq \mu_2\geq\cdots\geq\mu_n.
\]
Let \(\pi_\mu\) be the tempiric representation of \(G\) whose lowest
\(K\)-type has highest weight \(\mu\).  To locate the connected component of \(\operatorname{supp}(\lambda_X)\) containing \(\pi_\mu\), choose a maximal family of disjoint pairs
\[
(i_1,j_1),\dots,(i_m,j_m),
\qquad
 i_1<j_1<i_2<j_2<\cdots<j_m,
\]
with
\[
\mu_{i_\ell}=\mu_{j_\ell}
\qquad (1\leq \ell\leq m).
\]
Applying the corresponding Cayley transforms to the split Cartan produces a
\(\theta\)-stable Cartan subalgebra
\[
\fbh^\mu=\fth^\mu\oplus\fah^\mu\subset\fsl(n,\R),
\]
where
\[
\fth^\mu
=
\operatorname{span}_{\R}
\{E_{i_\ell j_\ell}-E_{j_\ell i_\ell}:1\leq\ell\leq m\},
\]
and
\[
\fah^\mu
=
\{(a_1,\dots,a_n)\in\fah^0:a_{i_\ell}=a_{j_\ell}
\text{ for }1\leq\ell\leq m\}.
\]
After reordering the basis so that the paired coordinates
\((i_\ell,j_\ell)\) occur first, \(\fbh^\mu\) has exactly the block-matrix
form displayed above. Set 
\[
T^\mu=\exp(\fth^\mu \oplus i\fah^\mu ), \quad  A^\mu=\exp(\fah^\mu\oplus i\fth^\mu).
\]

The Cayley transform transports $\mu$ as a functional on $i\fah^0$ to a functional $\tilde{\mu}$ on $i\fah^\mu\oplus \fth^\mu$ vanishing on $\fth^\mu$. The functional $\tilde{\mu}$ exponentiates to a character $e^{\tilde{\mu}}$ on $T^{\mu}$, and 
\[
(T^\mu A^\mu, e^{\tilde\mu}, \tilde{\mu})
\]
is a discrete character datum for $X$ if and only if $e^{\tilde\mu}$ is trivial on $T\cap H$. The $M'$-regularity amounts to the condition  that $\tilde{\mu}$ be regular in $Z_{\fg}(\faq^\mu)=Z_{\fg}(\fth^\mu)$, and follows from 
\[
\mu_{r}\neq \mu_{s},\  \text{for any $r\neq s$ not in $\{i_1,j_1, \ldots, i_m,j_m \}$ }.
\]
The sphericity condition $e^{\tilde\mu}\in \widehat{T^\mu}_{T^\mu\cap H}$ reads as
\begin{equation}\label{eq:SLn-sphericity}
\mu_r-\mu_s\in2\Z
\quad\text{whenever  \(\mu_r\) and \(\mu_s\)
occur with odd multiplicity.}
\end{equation}

Thus, when \(\pi_\mu\) satisfies the sphericity condition~\eqref{eq:SLn-sphericity}, the connected component of \(\operatorname{supp}(\lambda_X)\) containing \(\pi_\mu\) is attached to \(\fbh^\mu\) and is homeomorphic to
\[
i(\faq^\mu)^\ast/W_\mu^\sigma,
\qquad
\dim \faq^\mu=m.
\]
Here \(W_\mu^\sigma\) is the stabilizer of $\tilde{\mu}$ inside \(W^\sigma(\fbh^\mu)=W^\sigma_K(\ft_h^\mu\oplus \ftq^\mu)\).

For example, if \(\mu\) is regular, then \(m=0\) (discrete series), and the criterion $\mu\in \widehat{T}_{T\cap H}$ becomes the familiar $\mathrm{SO}(n)$-sphericity condition for the representation of $\mathrm{SU}(n)$ with highest weight $\mu$:
\[
\pi_\mu\in\widehat G_{H,\ds}
\quad\Longleftrightarrow\quad \text{$\mu$ is regular; and} \
\mu_1,\dots,\mu_n\text{ all have the same parity.}
\]

In general, $\pi_\mu\in\supp(\lambda_X)$ if and only if $\mu$ satisfies the parity condition~\eqref{eq:SLn-sphericity}.

\subsection{\texorpdfstring{$C^*$}{C*}-algebraic corollaries}

The following $C^\ast$-algebraic statement follows from Theorem~\ref{thm:GcGr-main} together with the
standard block decomposition of the reduced group \(C^*\)-algebra of a
connected complex semisimple group \cite{PP83, CCH16}.

\begin{proposition}\label{cor:GcGr-Cstar}
With the notation of Theorem~\ref{thm:GcGr-main}, there is a canonical
isomorphism
\[
C^*_{\lambda_X}(G)
\cong
\bigoplus_{[\fbh]}
\ \bigoplus_{[(\Gamma,\lambda)]}
C_0\bigl(i\faq^\ast/W^\sigma(\lambda),\Compact\bigr).
\]
\end{proposition}

\begin{proof}
For connected complex semisimple groups, the block decomposition of
\(C_r^*(G)\) into trivial homogeneous algebras over the tempered parameter
spaces is due to Penington--Plymen \cite{PP83}; it is also recovered from the
\(C^*\)-algebraic parabolic-induction framework in
\cite{CCH16,CHS24,CHST24}. The $C^\ast$-algebra  \(C^*_{\lambda_X}(G)\) identifies with the quotient of $C^\ast_r(G)$ by the annihilator of the unitary representation of $C^\ast_r(G)$ determined by $\lambda_X$. By Theorem~\ref{thm:GcGr-main}(2), the spectrum
of the quotient \(C^*_{\lambda_X}(G)\) of $C^*_r(G)$ is the closed Hausdorff subspace
\(\supp(\lambda_X)\) of $\widehat{G}_\temp$, whose connected pieces are the spaces
\(i\faq^\ast/W^\sigma(\lambda)\).  Passing to the quotient
\(C^*_{\lambda_X}(G)\) restricts the block functions to this closed subspace,
which gives the displayed direct sum.
\end{proof}

Let \(\widehat K_{K\cap H,\mathrm{reg}}\) denote the set of irreducible
\(K\)-representations with nonzero \(K\cap H\)-fixed vectors whose highest
weights are regular for the ambient complex root system.

\begin{corollary}\label{thm_K_GcGr}
Let \(X=G/H\) be a symmetric space of type \(G_{\mathbb C}/G_{\mathbb R}\).
Then the following are equivalent:
\begin{enumerate}[(1)]
\item \(K_\ast(C^*_{\lambda_X}(G))\neq 0\);
\item for a Cartan subspace \(\ftq\subset\fk\cap\fq\), equivalently for every
such Cartan subspace, \(Z_\fk(\ftq)\) is abelian;
\item for a Cartan subspace \(\fah\subset\fp\cap\fh\), equivalently for every
such Cartan subspace, \(Z_\fh(\fah)\) is abelian;
\item the real Lie algebra \(\fh\) is quasi-split.
\end{enumerate}
If these conditions hold, and if \(\ftq\oplus\faq\) is the most compact Cartan
subspace of \(\fq\), then
\[
K_\ast(C^*_{\lambda_X}(G))
\cong
\begin{cases}
\bigoplus_{\widehat K_{K\cap H,\mathrm{reg}}}\bZ,
& \ast\equiv \dim(\faq)\pmod 2,\\[1ex]
0,
& \ast\equiv \dim(\faq)+1\pmod 2.
\end{cases}
\]
\end{corollary}

\begin{proof}
By Proposition~\ref{cor:GcGr-Cstar}, the \(K\)-theory is the direct sum of the topological \(K\)-theories of the quotient spaces
\(i\faq^\ast/W^\sigma(\lambda)\).  For the Weyl-group quotients
appearing here, a summand contributes trivially to $K$-theory when the
stabilizer \(W^\sigma(\lambda)\) acts non-trivially on $i\faq^\ast$. To see this, using the notations from the proof of Lemma~\ref{lem:relative-Weyl}, by splitting $\faq$ as $\faq^1\oplus \faq^0$ where $\faq^0$ is the intersection of $\faq$ with the center of $\fg^{\lambda}$ on which $W(\lambda)$ acts trivially, $i\faq^\ast/W^\sigma(\lambda)$ is identified as the product of $\faq^0$ and the closed dominant Weyl chamber $\faq^{1,++}$ of $\faq^1$ for the possibly non-reduced root system $\Sigma(\fg^{\lambda}, \faq^1)$. The system is non-empty; therefore, for the half-sum $\rho$ of the positive roots, translation by $t\rho$ for $t\geq0$ shows that $C_0(i\faq^{*}/W^\sigma(\lambda))\cong C_0(\faq^0)\otimes C_0(\faq^{1,++})$ is contractible (homotopic to the zero $C^*$-algebra).

On the other hand, a summand contributes non-trivially when the
stabilizer \(W^\sigma(\lambda)\) acts trivially on \(i\faq^\ast\); in
that case the contribution is \(\bZ\) in parity \(\dim(\faq)\). These summands correspond exactly to the regular \(K\cap H\)-spherical lowest \(K\)-types.  Their
existence is equivalent to the existence of an element of \(\fk\cap\fq\) which
is regular in \(\fg\), or equivalently to \(Z_\fk(\ftq)\) being abelian for a
Cartan subspace \(\ftq\subset\fk\cap\fq\).  Under the correspondence
\(\ftq=i\fah\), this is equivalent to \(Z_\fh(\fah)\) being abelian for a Cartan
subspace \(\fah\subset\fp\cap\fh\).  For a real semisimple Lie algebra this is
the standard quasi-split condition.  This proves the equivalences and the
stated degree shift.
\end{proof}

\section{Appendix: a torus form of the Cartan--Helgason criterion}\label{sec:appendix}

The contents of this appendix, including the final statements in Propo-sition~\ref{prop:appendix-general} and Corollary~\ref{cor:appendix-parametrization}, could be well-known to experts. We provide details, as we have been unable to find a suitable reference in the literature.

We shall spell out a version the compact Cartan--Helgason criterion for possibly disconnected groups. The result is used in the main part of the paper to prove Theorem~\ref{thm:ds_para}. For a connected compact symmetric
pair $(K, H)$, the Cartan–Helgason criterion is the following asssertion: an irreducible representation is $H$-spherical if and
only if its highest weight is trivial on $T_0\cap H$, where $T_0$ is a maximally
$\sigma$-split maximal torus.  One point that matters for us is that $T_0\cap H$
may be disconnected, even when $K$ and $H$ are connected (for instance, for
$SU(2)/SO(2)$ one has $T_0\cap H=\{\pm 1\}$).  We take the connected
case as input and then extend it to possibly disconnected compact groups by
elementary averaging arguments.

\subsection{Setup and Notation}

Let $K$ be a compact Lie group, let $\sigma$ be an involutive automorphism of $K$, and let
$H$ be an open subgroup of $K^\sigma$.  Write
\[
\fk=\fh\oplus\fq
\]
for the $(\pm1)$-eigenspace decomposition of $\sigma$ on $\fk$.  Choose a maximal abelian
subspace $\ft_\fq\subset \fq$, and a maximal abelian subspace
$\ft_\fh\subset Z_{\fh}(\ft_\fq)$.  Put
\[
\ft=\ft_\fh\oplus\ft_\fq,
\qquad
T_0=\exp(\ft),
\]
and write
\[
X^*(T_0):=\operatorname{Hom}_{\mathrm{cts}}(T_0,U(1))
\]
for its character lattice. Fix a positive system
\[
\Sigma^+=\Sigma^+(\fk^\bC,\ft_\fq^\bC)
\]
of restricted roots, and a compatible positive system
\[
\Delta^+=\Delta^+(\fk^\bC,\ft^\bC),
\]
meaning that for every root $\alpha\in \Delta(\fk^\bC,\ft^\bC)$ with
$\alpha|_{\ft_\fq}\neq 0$, one has
\[
\alpha\in \Delta^+
\quad\Longleftrightarrow\quad
\alpha|_{\ft_\fq}\in \Sigma^+.
\]
Let
\[
\fb=\ft^\bC\oplus \sum_{\alpha\in\Delta^+}\fk^\bC_\alpha,
\qquad
\fn=\sum_{\alpha\in\Delta^+}\fk^\bC_\alpha.
\]
Set
\[
T^l:=N_K(\fb)=N_K(\ft,\Delta^+).
\]
Following Vogan \cite[Definition 1.14(e)]{VoganUnitaryBook}, we call \(T^l\) a large Cartan subgroup of \(K\).  When
\(K\) is connected, one has \(T^l=T_0\). For disconnected compact groups, the highest-weight parametrization of
\(\widehat K\) in terms of irreducible representations of \(T^l\) can be found in \cite[Theorem~1.17]{VoganUnitaryBook}: for any irreducible finite-dimensional
representation $(\pi,V)$ of $K$, the space
\[
V^{\fn}:=\{v\in V: d\pi(X)v=0 \text{ for all } X\in \fn\}
\]
is an irreducible representation of $T^l$; let us call it the
\emph{highest weight representation} attached to~$\pi$ (this depends on the choice of $(\ft, \Delta^+$)).  When $K$ is connected, one
has $T^l=T_0$ and $V^\fn$ is the usual highest weight line.

\subsection{The connected case}

\begin{proposition}[Connected compact Cartan--Helgason]
\label{prop:appendix-connected-CH}
Assume that $K$ and $H$ are connected.  Let $(\pi,V)$ be an irreducible
finite-dimensional representation of $K$, and let
$\lambda\in X^*(T_0)$ be its highest weight with respect to~$\Delta^+$.  Then
\[
V^H\neq 0
\quad\Longleftrightarrow\quad
\lambda|_{T_0\cap H}=1.
\]
If these equivalent conditions hold, then $\dim_\bC V^H=1$, and the averaging
operator
\[
P_H(v):=\int_H \pi(h)v\,dh
\]
is nonzero on the highest weight line.
\end{proposition}

\begin{proof}
This is the compact Cartan--Helgason theorem; see
\cite[Theorem~4.1 and Corollary~4.2]{helgason}.  For the torus formulation used here, namely the criterion
\(\lambda|_{T_0\cap H}=1\), see also \cite[Theorem 12.3.13]{GoodmanWallach}.
\end{proof}

\begin{remark}
\label{rem:appendix-TcapH-disconnected}
Even when $H$ is connected, the subgroup $T_0\cap H$ need not be connected.
Thus the condition \(\lambda|_{T_0\cap H}=1\) in Proposition~\ref{prop:appendix-connected-CH} is genuinely stronger than the
infinitesimal condition $d\lambda|_{\ft_\fh}=0$.
\end{remark}

\subsection{Disconnected $K$, connected $H$}

\begin{proposition}
\label{prop:appendix-connected-H}
Assume that $H$ is connected, but allow $K$ to be disconnected.  Let
$(\pi,V)$ be an irreducible finite-dimensional representation of~$K$.  Then the
averaging operator over $H$ restricts to an isomorphism
\[
P_H : (V^\fn)^{T^l\cap H}\xrightarrow{\sim} V^H.
\]
Hence,
\[
\dim_\bC V^H=\dim_\bC (V^\fn)^{T^l\cap H}.
\]
In particular, $V$ is $H$-spherical if and only if the dominant representation
$V^\fn$ contains the trivial representation of $T^l\cap H$.
\end{proposition}

\begin{proof}
Since $H$ is connected, one has $H\subset K_0$.  Moreover $T^l\cap K_0=T_0$,
and therefore
$
T^l\cap H=T_0\cap H$.

Decompose $V$ as a finite direct sum $\bigoplus_{j=1}^N V_j$ of irreducible $K_0$-modules. 
Then
\[
V^\fn=\bigoplus_{j=1}^N V_j^\fn,
\qquad
V^H=\bigoplus_{j=1}^N V_j^H.
\]
For each $j$, Proposition~\ref{prop:appendix-connected-CH} shows that
\[
P_H:(V_j^\fn)^{T_0\cap H}\longrightarrow V_j^H
\]
is an isomorphism.
Summing over $j$, we see that $
P_H:(V^\fn)^{T_0\cap H}\xrightarrow{\sim} V^H$ is an isomorphism. 
Since $T^l\cap H=T_0\cap H$, this is exactly the required statement.
\end{proof}

\subsection{The component group of $H$}

The passage from connected to disconnected $H$ is controlled by the following
simple structural lemma.

\begin{lemma}
\label{lem:appendix-components}
$H=H_0\,(T^l\cap H).$
\end{lemma}
\begin{proof}

Set
\[
M':=Z_H(\ft_\fq),
\qquad
M_H^l:=N_H(\ft_\fq,\Sigma^+),
\]
and let \(M'_0\) denote the identity component of \(M'\).

We first show that \(H=H_0M_H^l\). It is enough to prove that $H \subset H_0 M_H^l$. 
Let $h\in H$.  Since $\Ad(h)\ft_\fq$ is maximal abelian in $\fq$,
there exists $h_1\in H_0$ such that
\[
\Ad(h_1h)\ft_\fq=\ft_\fq.
\]
Thus $h_1h\in N_H(\ft_\fq)$.  The restricted Weyl group
\[
W(\fk,\ft_\fq)=N_{H_0}(\ft_\fq)/Z_{H_0}(\ft_\fq)
\]
acts transitively on the Weyl chambers of
$\Sigma(\fk^\bC,\ft_\fq^\bC)$, so there exists
$n\in N_{H_0}(\ft_\fq)$ such that
\[
\Ad(nh_1h)\Sigma^+=\Sigma^+.
\]
Hence $nh_1h\in M_H^l$, and this concludes the proof that $
H=H_0\,M_H^l$.

Now let $m\in M_H^l$.  Since $m$ preserves $\ft_\fq$, it normalizes
$Z_{\fh}(\ft_\fq)=\fm'$.  Therefore $\Ad(m)\ft_\fh$ is a maximal abelian
subspace of $Z_{\fh}(\ft_\fq)$.  By conjugacy of maximal tori in the connected
compact group $M'_0$, there exists $m_1\in M'_0$ such that
\[
\Ad(m_1m)\ft_\fh=\ft_\fh.
\]
Thus $m_1m$ normalizes $\ft_\fh$.  The Weyl group
\[
W(M'_0,\ft_\fh)=N_{M'_0}(\ft_\fh)/Z_{M'_0}(\ft_\fh)
\]
acts transitively on the Weyl chambers of the root system
$\Delta_0=\{\alpha\in \Delta(\fk^\bC,\ft^\bC):\alpha|_{\ft_\fq}=0\}$, so there
exists $n_1\in N_{M'_0}(\ft_\fh)$ such that
\[
\Ad(n_1m_1m)\Delta_0^+=\Delta_0^+,
\]
where
$
\Delta_0^+:=\{\alpha\in \Delta^+:\alpha|_{\ft_\fq}=0\}$.
Set
\[
x:=n_1m_1m.
\]
Because \(n_1,m_1\in M'_0\subset Z_H(\ft_\fq)\), both these elements centralize \(\ft_\fq\) and hence preserve \(\Sigma^+\). Since $m\in M_H^l$, this element also preserves $(\ft_\fq,\Sigma^+)$.
Thus $x$ preserves $(\ft_\fq,\Sigma^+)$.  By construction, $x$ also preserves
$(\ft_\fh,\Delta_0^+)$.  Since the full positive system $\Delta^+$ is compatible
with $\Sigma^+$, these two facts imply that $x$ preserves the pair
$(\ft,\Delta^+)$.  Equivalently,
\[
x\in N_K(\ft,\Delta^+)\cap H=T^l\cap H.
\]
As $n_1m_1\in M'_0$, we obtain
$m\in M'_0\,(T^l\cap H)$.
This proves 
\[
M_H^l\subset M'_0\,(T^l\cap H).
\]

The reverse inclusion is clear: $M'_0$ centralizes $\ft_\fq$, hence lies in
$M_H^l$, and every element of $T^l\cap H$ preserves $(\ft,\Delta^+)$, in
particular preserves $(\ft_\fq,\Sigma^+)$.  Therefore
\[
M_H^l=M'_0\,(T^l\cap H).
\]
Combining this with $H=H_0M_H^l$ gives $
H=H_0\,(T^l\cap H)$.
\end{proof}

\subsection{The general case}

\begin{proposition}
\label{prop:appendix-general}
Let $K$ be a compact Lie group, let $H$ be an open subgroup of~$K^\sigma$, and
let $(\pi,V)$ be an irreducible finite-dimensional representation of~$K$.
Then the averaging map over $H_0$ induces an isomorphism of
$T^l\cap H$-modules
\[
P_{H_0}:(V^\fn)^{T_0\cap H_0}\xrightarrow{\sim} V^{H_0}.
\]
Consequently,
\[
V^H \simeq (V^\fn)^{T^l\cap H},
\qquad
\dim_\bC V^H=\dim_\bC (V^\fn)^{T^l\cap H}.
\]
In particular, $V$ is $H$-spherical if and only if the dominant representation
$V^\fn$ contains the trivial representation of $T^l\cap H$.
\end{proposition}

\begin{proof}
Apply Proposition~\ref{prop:appendix-connected-H} to the connected subgroup
$H_0$.  This gives a $T^l\cap H$-equivariant isomorphism
\[
P_{H_0}:(V^\fn)^{T^l\cap H_0}\xrightarrow{\sim} V^{H_0}.
\]
The equivariance follows because every element of $T^l\cap H$ normalizes both
$H_0$ and the Borel subalgebra $\fb$.

Moreover, by Lemma~\ref{lem:appendix-components} we have
 $
H=H_0\,(T^l\cap H)
$, hence
\[
V^H=(V^{H_0})^{T^l\cap H}.
\]
Taking $T^l\cap H$-fixed vectors in the isomorphism above, we obtain
\[
V^H
\simeq
\bigl((V^\fn)^{T^l\cap H_0}\bigr)^{T^l\cap H}
=
(V^\fn)^{T^l\cap H},
\]
because $T^l\cap H_0\subset T^l\cap H$.
\end{proof}

\begin{corollary}
\label{cor:appendix-parametrization}
Under Vogan's highest-weight classification \cite[Theorem 1.17]{VoganUnitaryBook} for possibly disconnected compact
groups, the assignment
\[
(\pi,V)\longmapsto V^\fn
\]
restricts to a bijection between \(H\)-spherical irreducible representations
of \(K\) and dominant irreducible representations \((\mu, V_\mu)\) of \(T^l\) satisfying
\(V_\mu^{T^l\cap H}\neq0\).

If $\pi$ and $\mu$ are matched by that bijection, then 
\[
\dim_\bC V^H=\dim_\bC V_\mu^{T^l\cap H}.
\]
When $K$ is connected, $\mu$ is a character of $T_0$, and the condition above
reduces to
\[
\mu|_{T_0\cap H}=1.
\]
\end{corollary}

\begin{proof}
This is immediate from Proposition~\ref{prop:appendix-general}.
\end{proof}

\begin{remark} Let \(K\) be a product of finitely many copies of \(O(2)\). Then
\(\mathfrak k^\mathbb C\) is abelian, so the chosen Borel subalgebra is
\(\mathfrak b=\mathfrak k^\mathbb C\), and hence
\[
T^l=N_K(\mathfrak b)=K.
\]
In this case Corollary~\ref{cor:appendix-parametrization} is tautological.
\end{remark}

\bibliographystyle{plain}
\bibliography{bib_paper1}
\end{document}